\documentclass[onefignum, onetabnum]{siamonline220329}
\usepackage{braket,amsfonts}
\usepackage{array}
\usepackage[caption=false]{subfig}
\usepackage{ntheorem}
\newsiamthm{claim}{Claim}
\newsiamremark{remark}{Remark}
\newsiamremark{hypothesis}{Hypothesis}
\crefname{hypothesis}{Hypothesis}{Hypotheses}
\usepackage{graphicx}
\Crefname{ALC@unique}{Line}{Lines}
\usepackage{amsopn}

\usepackage{amsmath,amssymb}
\usepackage{color}
\usepackage{tikz}
\usetikzlibrary{plotmarks}
\usepackage{pgfplots}
\pgfplotsset{compat=1.16} 
\usepackage{siunitx} 
\usepackage{url}
\usepackage{algorithmic}
\usepackage{algorithm}
\usepackage{dsfont}

\newcommand{\N}{\mathbb{N}}
\newcommand{\R}{\mathbb{R}}

\newcommand{\dd}[1]{\frac{\partial}{\partial #1}}

\newcommand{\dx}{\, \textnormal{d}}
\newcommand{\ind}[1]{\mathds{1}_{#1}}
\newcommand{\chr}[1]{\chi_{#1}}
\title{Dynamic image reconstruction with motion priors in application to 3D magnetic particle imaging \thanks{\funding{T. Kluth acknowledges funding by the German Research Foundation (DFG, Deutsche Forschungsgemeinschaft) - project 426078691.}}
}

\author{Christina Brandt \thanks{Department of Mathematics, Universität Hamburg, 20146 Hamburg, Germany (\email{christina.brandt@uni-hamburg.de}, \email{lena.westen@uni-hamburg.de}).}
\and Tobias Kluth \thanks{Center for Industrial Mathematics, University of Bremen, 28359 Bremen, Germany (\email{tkluth@math.uni-bremen.de}).}
\and Tobias Knopp \thanks{Institute for Biomedical Imaging, Hamburg University of Technology and University Medical Center Hamburg-Eppendorf, 22529 Hamburg, Germany (\email{tobias.knopp@tuhh.de})}
\and Lena Westen \footnotemark[2]}

\headers{Dynamic 3D magnetic particle imaging}{Christina Brandt, Tobias Kluth, Tobias Knopp and Lena Westen}

\begin{document}
\maketitle

\begin{abstract}
Various imaging modalities allow for time-dependent image reconstructions from measurements where its acquisition also has a time-dependent nature. Magnetic particle imaging (MPI) falls into this class of imaging modalities and it thus also provides a dynamic inverse problem. Without proper consideration of the dynamic behavior, motion artifacts in the reconstruction become an issue. More sophisticated methods need to be developed and applied to the reconstruction of the time-dependent sequences of images. In this context, we investigate the incorporation of motion priors in terms of certain flow-parameter-dependent PDEs in the reconstruction process of time-dependent 3D images in magnetic particle imaging. The present work comprises the method development for a general 3D+time  setting for time-dependent linear forward operators, analytical investigation of necessary properties in the MPI forward operator, modeling aspects in dynamic MPI, and extensive numerical experiments on 3D+time imaging including simulated data as well as measurements from a rotation phantom and in-vivo data from a mouse.  

\end{abstract}




\section{Introduction}

Dynamic inverse problems emerge naturally in several imaging modalities (such as CT, MRI, SPECT, magnetic particle imaging (MPI), etc.) as often both, cause and observation, show a time-dependent behavior. Due to this fact, the field of dynamic inverse problems has been driven strongly by tomographic applications in the past. The dynamic behavior of the object in the time-dependent data acquisition process often leads to motion artifacts when using standard reconstruction techniques being designed for static problems. 
Image quality then suffers such that 
dedicated methods taking into account the dynamic nature of the problem need to be used and further developed in this setting.

A general regularization framework for dynamic inverse problems is still not established, but several promising attempts are made in this direction.  
Approaches exploiting concepts like, for example, temporal smoothness \cite{schmitt1,schmitt2}, dynamic programming techniques \cite{leitao1}, explicit deformation models \cite{Hahnnonlinear}, or operator inexactness \cite{blanke2020inverse} have been proposed.    
One rather general approach for time-dependent parameter identification in dynamic systems has been made in \cite{Kaltenbacher17,Tram19} by distinguishing reduced and all-at-once approaches.

For specific imaging modalities several methods have been proposed in the literature. Depending on the characteristics of the considered modality, approaches are either based on a variational formulation (e.g., \cite{Gris19,Kluth2019IWMPI,zhang,DynamicMRI}),  
exact analytic methods (e.g., \cite{exact1,exact3,Hahnaffine}), iterative methods (e.g., \cite{iterative2,iterative1}), approximate inversion formulas (e.g., \cite{Arridge2022,hahn2017motion,katsevich_accurate,katsevich11}), and recently on machine learning techniques (e.g., \cite{hauptmann2019real,ilg2017flownet, LI2022}). One explicit way to compensate for the dynamics is the incorporation of 
explicit motion models such as deformation, optical flow, etc. in the image reconstruction process, which then require the identification of motion parameters prior to or within the reconstruction step (see, e.g., \cite{Burger17,Burger.2018,Hahn_estimation,katsevich11,lu_mackie,Lucka2018enhancing,motion_registration_manke,reyes}). 
For a more detailed review on variational approaches taking into account motion priors and also learning-based approaches, we refer to the excellent recent survey \cite{hauptmann2021image}. 

The particular imaging modality being considered in the present work is MPI \cite{Gleich2005}. It is a tracer-based imaging modality exploiting the magnetization behavior of magnetic nanoparticles in a highly dynamic applied magnetic field, see also the surveys \cite{Kluth2018a,knopp2017review} for an overview. 
From a mathematical point of view, the forward operator for static images is given by a linear Fredholm integral operator of the first kind, mapping a spatially distributed concentration (image) to a time-dependent voltage measurement. 
The static image reconstruction problem is already a severely ill-posed problem \cite{erb2018mathematical,Kluth2018b}, which becomes even more ill-posed in the dynamic scenario.
So far, MPI has been used primarily for pre-clinical medical applications but it holds a significant potential for clinical applications illustrated by an increasing number of potential applications, which also rely on the solution of dynamic image reconstruction problems.
One application, already suggested at the very beginning, is vascular imaging \cite{Gleich2005}.
In addition, in in-vivo experiments, the potential for imaging blood flow was demonstrated using healthy mice \cite{weizenecker2009three}. 
Further promising applications include long-term circulating tracers \cite{khandhar2017evaluation}, tracking medical instruments \cite{haegele2012magnetic},
e.g., in angioplasty \cite{Salamon:2016cz}, cancer detection \cite{Yu2017}, cancer treatment by hyperthermia \cite{murase2015usefulness}, and stroke monitoring \cite{graser2019human,ludewig2022magnetic}.

The time-dependent measurement process of MPI together with dynamic imaging applications requires the consideration of problems across temporal scales. In the field of MPI, a few works considered the dynamic image reconstruction problem therein so far. The general dynamic setting has been investigated from a modeling point of view in \cite{brandt2021modeling}. A reconstruction  approach designed for periodic motion \cite{Gdaniec2017,Gdaniec.2020} has been proposed which is restricted to a limited class of motion only. More recently, temporal splines have been included in the reconstruction process to increase the temporal resolution \cite{brandt2022motion}. Motion models have been used for extracting optical flow information from previously obtained image reconstructions \cite{Franke2017}. 

The present work considers the dynamic 3D image reconstruction problem in the imaging modality MPI. 
In particular, we consider a variational approach taking into account motion priors which allow for a simultaneous identification of parameters therein. This paper is motivated by previous work for time-dependent 2D images in MPI that was published in \cite{Kluth2019IWMPI} as a short abstract. The numerical treatment can be interpreted as a special case of the general all-at-once approach discussed in \cite{Kaltenbacher17,Tram19}. In contrast to the prior work, we particularly consider the 3D case and focus further on the analysis of the MPI model in the framework of the general joint reconstruction approach. Specific choices of motion priors are an optical flow constraint being suitable, for example, for instrument tracking and mass preservation, e.g. in the case of for blood flow visualization. 
The theoretical part is complemented by the algorithmic solution via alternating minimization using (stochastic) primal-dual methods combined with multi-scale schemes and an intense numerical study taking into account simulated as well as measured 3D phantoms and in-vivo experiments evolving in time.


The manuscript is structured as follows: 
In \Cref{sec:problemsetting} the general mathematical setting of the dynamic reconstruction approach is considered and the corresponding theoretical investigation of the MPI model is provided. \Cref{sec:algorithm} provides the numerical solution to the underlying minimization problem and in \Cref{sec:numerical} the performance of the proposed method is illustrated in a series of numerical experiments. The manuscript concludes with a conclusion and an outlook in \Cref{sec:conclusion}.

\section{Mathematical problem setting and solution approach}
\label{sec:problemsetting}
We start with a precise definition of the problem and the theoretical investigation of the underlying model. In \Cref{subsec:genproblem} we fix the general setting and define the problem to be solved. Moreover, we state the theorem for existence of minimizers of our problem. In the following \Cref{subsec:jointMPIreco}, we introduce a model for the forward operator of magnetic particle imaging and prove compactness as well as some required regularity assumptions. \Cref{subsec:dynamicMPI} then outlines differences between static and dynamic MPI and derives the associated problems. 

\subsection{General joint image reconstruction and motion estimation problem}
\label{subsec:genproblem}
In the following, we consider a measurable time-dependent image function $c$ on a bounded space-time domain $\Omega \times \left[ 0,T\right] \subset \mathbb{R}^n \times \mathbb{R}^{+}$, $n \in \left\lbrace 1,2,3\right\rbrace $, $c \in L^p\left( 0, T; BV\left( \Omega\right) \right) $, $p>1$ . We choose the space of functions with bounded variation, as those can describe edges in images and are well-suited for images with large homogeneous regions. Both features are particularly important for a wide range of medical images, which is one important field of application we have in mind. 

We aim at reconstructing this image function from measured data $u: \left[ 0,T\right]  \rightarrow Y$, $u \in L^2\left( 0,T; Y\right) $ for a reflexive Banach space $Y$, which is obtained by inserting $c$ as an argument into the first component of an operator $A: L^{\hat{p}}\left(0,T;L^l\left( \Omega\right)\right) \times \left[ 0,T\right]  \rightarrow L^2\left(0,T; Y\right)$, $\hat{p}=\min(2,p)$, $l\leq \frac{n}{n-1}$, and corrupting the results with random noise $\delta \in L^2\left(0,T; Y\right)$, i.e. 
 \begin{equation}
 	A\left(c,t\right) + \delta\left( t \right) = u\left( t \right), \quad t \in \left[ 0,T \right].
 \end{equation}
 In the present work, we consider $A(\cdot,t)$ to be a linear and bounded operator for any $t\in [0,T]$. 
Note that time-dependency enters our inverse problem for two reasons: first, the image to be recovered depends on time and second, the forward operator is time-dependent
as in other imaging applications such as in blurring  of signals with time-dependent convolution kernels or time-dependent Radon data in dynamic CT.

Suppose we consider an ill-posed image reconstruction task, such that we add a spatial regularizer $R: BV\left( \Omega\right) \rightarrow \mathbb{R}$, which is assumed to be proper, convex and lower semicontinuous and fulfills 
\begin{equation}
	\label{eq:desC}
	R\left(x\right) \geq \left| x \right| _{BV}^p \qquad \mathrm{for \enspace any \enspace} x \in \mathrm{BV}\left(\Omega\right) .
\end{equation} 
In addition to the image reconstruction task, we aim at simultaneously determining a velocity field $v : \Omega \times \left[ 0,T\right] \rightarrow \mathbb{R}^{n},$ $v \in L^q\left( 0, T; BV\left( \Omega\right)^n \right)  $, $q>1$, describing the motion in the data. We incorporate $v$ into our model by using a motion model $m\left( c, v\right) $ linking images and motion as an additional constraint. Depending on the underlying assumptions, this will be either an optical flow term or a mass conservation constraint. Again, motion estimation is an ill-posed task in general such that we add a proper, convex and lower semicontinuous regularizer $S: BV\left( \Omega\right)^n \rightarrow \mathbb{R} $ working on the motion field. We will restrict ourselves to regularizers which again fulfill 
\begin{equation}
	\label{eq:desS}
	S\left( y\right) \geq \left| y \right| _{BV}^q \qquad \mathrm{for \enspace any \enspace} y \in \mathrm{BV}\left(\Omega\right)^n .
\end{equation}  
The joint image reconstruction and motion estimation problem can then be described as 
\begin{align}
	\label{eq:min_prob1} 
 \min_{c,v} &\displaystyle \int_{0}^{T} D\left( A\left( c, t\right), u\left( t\right) \right) + \alpha R\left( c(\cdot,t)\right) + \beta S\left( v(\cdot,t)\right) \mathrm{d}t,\\ 
	\label{eq:nb} & \mathrm{s. t. } \enspace  m\left( c, v\right) =0 \quad \text{in } \mathcal{D}'(\Omega \times [0,T]), 
\end{align}
 where $D: Y \times Y \rightarrow \mathbb{R}$ denotes a proper, convex and lower semicontinuous discrepancy term between measured and modeled data and $\mathcal{D}'(\Omega \times [0,T])$ denotes the space of distributions, i.e. of compactly supported $C^\infty$ functions on $\Omega \times [0,T]$. 
 
 We consider two different motion models: one assuming voxel intensity constancy and one assuming mass preservation during motion. 
 Consider first conservation of the gray-value $c$ under motion, i.e. 
 \begin{equation}\label{eq:conv_gray_value}
 	c\left( x,t\right) - c\left( x+\delta_t v\left(x,t\right), t+\delta_t\right)  =0,
 \end{equation}
 for all $x \in \Omega$, $t \in \left[0,T\right]$, $\delta_t>0$ small. By Taylor expansion we deduce the optical flow constraint 
 \begin{equation}
 	\label{eq:m1}
 	 m_1\left( c,v\right) = \dd{t}c + \nabla c \cdot v   =0.
 \end{equation}
 To expand our setting to applications where the gray-value constancy is not fulfilled (having in mind for example a divergent blood flow), we introduce the mass-conservation assumption, i.e. we assume
 \begin{equation}
 	\int_{\Omega} c(x,t) \mathrm{d}x = K \qquad \forall t \in \left[ 0, T\right] 
 \end{equation}
for a constant $K \in \mathbb{R}$. From the local form 
\begin{equation}
\label{eq:mc_part1}
	\frac{\mathrm{d}}{\mathrm{d}t}\int_{\mathcal{X}} c(x,t) \mathrm{d}x = \int_{\mathcal{X}} \dd{t}c(x,t) \mathrm{d}x
\end{equation}
for an arbitrary subset $\mathcal{X}\subset \Omega$, we can derive the so-called mass conservation constraint by
\begin{equation}
\label{eq:mc_part2}
    \int_{\mathcal{X}} \dd{t}c \mathrm{d}x \overset{!}{=} \int_{\partial \mathcal{X}} cv \cdot \nu \mathrm{d}\mathcal{X} = -\int_\mathcal{X} v\nabla c \mathrm{d}x- \int_\mathcal{X} c \nabla \cdot v \mathrm{d}x = -\int_\mathcal{X} \nabla \cdot \left( cv\right) \mathrm{d}x,
\end{equation}
where $\nu$ denotes the normal vector to $\mathrm{d}\mathcal{X}$. For the first equality, we used that the variation of mass over time must equal the flow through the boundary $\partial \mathcal{X}$. The second equality follows by the divergence theorem. Combining now \cref{eq:mc_part1} and \cref{eq:mc_part2} yields the mass conservation constraint
\begin{equation}
	\label{eq:m2}
	m_2\left( c,v\right) = \dd{t}c + \nabla \cdot \left(cv \right) = 0.
\end{equation}
The joint image reconstruction and motion estimation problem is well-defined as minimizers exist for both motion models, as is shown in the following theorem.
For clarity, note that in the following the characteristic function of a set $X$ is denoted by $\chr{X}$, where
\begin{equation*}
    \chr{X}(x) = \left\lbrace \begin{array}{cc}
	     1, & x \in X \\
	   0, & \mathrm{otherwise}
	 \end{array} \right. ,
\end{equation*}
whereas the indicator function of a set $X$ is denoted by $\ind{X}$, i.e.
 \begin{equation*}
     \ind{X} \left(x\right) = \left \lbrace \begin{array}{cc}
         0, &  x \in X\\
         \infty, & \mathrm{otherwise}
     \end{array} \right. .
 \end{equation*}
\begin{theorem}[Existence of minimizers]
\label{thm:existence}
	Consider the minimization problem \cref{eq:min_prob1}-\cref{eq:nb}: 
	\begin{align*}
		\min_{c,v} J\left( c,v\right) := &\displaystyle\int_{0}^{T} D\left( A\left( c,t \right), u\left( t\right) \right) + \alpha R\left( c(\cdot,t)\right) + \beta S\left( v(\cdot,t)\right) \mathrm{d}t, \\
		&\mathrm{s. t. } \enspace m\left( c, v\right) =0 \quad \text{in } \mathcal{D}'(\Omega \times [0,T])
	\end{align*} 
	 in $n \in \left\lbrace 2,3\right\rbrace $ dimensions, with $c \in L^p\left( 0,T; \mathrm{BV}(\Omega)\right) $, $v \in L^q\left( 0,T; \mathrm{BV}(\Omega)^n\right)$, $u \in L^2\left( 0,T; Y\right) $ over the bounded space-time domain $\Omega \times \left[ 0,T\right] $, $\Omega \subset \mathbb{R}^n$ for a reflexive Banach space $Y$, and $1 < p,q < \infty$, $\hat{p} = \min\left(2,p\right)$. 
	 Let $A: L^{\hat{p}} \left(0,T; L^l\left(\Omega \right)\right) \times \left[ 0,T\right] \rightarrow L^2\left( 0,T; Y\right)$ be generated by bounded operators $A_t: L^l\left( \Omega \right) \rightarrow Y$, $t \in \left[ 0,T\right]$, $l \leq \frac{n}{n-1}$, such that $A_t \left( c(t)\right) = A\left( c,t\right)$ and let $A_t \chr{\Omega} \neq 0\quad \text{for all } t \in \left[ 0,T\right] $. 	 
	Let $R$ and $S$ fulfill \cref{eq:desC} and \cref{eq:desS}, i.e. 
	\begin{flalign*}
	R\left(x\right) &\geq \left| x \right| _{BV}^p \qquad \mathrm{for \enspace any \enspace} x \in \mathrm{BV}\left(\Omega\right)  ,\\
	S\left( y\right) &\geq \left| y \right| _{BV}^q \qquad \mathrm{for \enspace any \enspace} y \in \mathrm{BV}\left(\Omega\right)^n .
	\end{flalign*} 
	Moreover, let $D\left( u_1, u_2 \right) = \frac{1}{2}\left\Vert u_1 - u_2 \right\Vert^2_Y$ be the squared norm distance in $Y$ and $m$ be either $m_1$ or $m_2$ as defined in \cref{eq:m1} and \cref{eq:m2}, respectively. 
	Moreover, assume there exist constants $k_\infty, k_\theta <\infty$ such that $\left\Vert v\right\Vert _{L^\infty\left(0,T,L^{\infty}\left(\Omega\right)^n\right)}\leq k_{\infty}$  and $\left\Vert \nabla\cdot v \right\Vert _{\theta}\leq k_{\theta}$, 	where $\theta=  L^{p^*s}\left( 0,T; L^{l^*k}\left( \Omega\right) \right) $ with $1<s,k<\infty$ and $\frac{1}{p}+\frac{1}{p^*}=1$.\\
	Then there exists a minimizer of the problem in the set of admissible solutions
	\begin{multline}
\left\lbrace \left(c,v\right)\in L^{\hat{p}}\left(0,T;BV\left(\Omega\right)\right)\times L^{q}\left(0,T;BV\left(\Omega\right)^n\right)\middle| \enspace \left\Vert v\right\Vert _{L^\infty\left(0,T;L^{\infty}\left(\Omega\right)^n\right)}\leq k_{\infty},\right.\\
\Big.\left.\left\Vert \nabla\cdot v\right\Vert _{\theta}\leq k_{\theta} , \enspace m\left(c,v\right)=0 \text{ in } \mathcal{D}'(\Omega \times [0,T]) \right. \bigg\}.
\end{multline}
\end{theorem}
\begin{proof}
	The proof is based on \cite{Burger.2018, Dirks.2015}, but we have three main differences in our assumptions. First, our image sequence is in up to $3+t$ dimensions instead of $2+t$; second, our regularizers are of a more general form and third, our linear operator is time-dependent and has a more general range. 
	The proof is stated in Appendix \ref{app:proofExistence} for completeness.
	
\end{proof}

\begin{remark}
    In practice, existence of those constants is no restrictive assumption. Bounding $\left\Vert v\right\Vert _{L^\infty(0,T;L^{\infty}\left(\Omega\right)^n)}$ is basically assuming a finite maximum speed, which is a physically necessary assumption. Moreover, assuming a bound on $\left\Vert \nabla\cdot v\right\Vert _{\theta}$ in an applied sense is bounding the compressibility of the flow, which is again a reasonable assumption. 
\end{remark}

\begin{remark}
    The operator $A_t$ is defined on $L^l\left( \Omega \right)$ for $l \leq \frac{n}{n-1}$. This limitation of topologies compatible with the operator is necessary in order to ensure the continuous embedding of $\mathrm{BV}\left( \Omega \right) $, which is the spatial domain in which we are searching for a minimizer, into the definition space of the operator.
\end{remark}

For computational purposes, either model constraint is incorporated into the variational problem \cref{eq:min_prob1} by adding an additional penalty term $T$ defined by
\begin{equation}
	T\left( c\left( \cdot,t \right), \dd{t}c\left( \cdot,t \right), v\left( \cdot,t \right)\right) = \left\| m_i\left( c,v\right)\left( \cdot,t \right)\right\| ^s_{L^r(\Omega)}, \enspace r, s \geq 1, \enspace i \in \left\lbrace 1,2\right\rbrace, \enspace t \in \left[0,T\right]  
\end{equation}
 such that we arrive at the unconstrained minimization problem 
 \begin{equation}
 	\label{eq:min_prob2} \min_{c,v} \int_{0}^{T} D\left( A\left( c, t\right), u\left( t\right) \right) + \alpha R\left( c(\cdot,t)\right) + \beta S\left( v(\cdot,t)\right) + \gamma T\left( c\left( \cdot,t\right),\dd{t}c\left( \cdot,t \right) ,v\left( \cdot, t\right) \right)  \mathrm{d}t\,,
 \end{equation}
 which can be solved by means of alternating minimization. In \cite[Lemma 3.7]{Burger.2018} it is shown that solutions of the unconstrained minimization problem \cref{eq:min_prob2} converge for $\gamma \rightarrow \infty$ to solutions of the constrained minimization problem \cref{eq:min_prob1}, \cref{eq:nb} in case of the optical flow constraint in $n=2$ dimensions. However, the proof which is based on showing $\Gamma$-convergence of the objective function of the unconstrained problem to the one of the constrained problem, can be carried out analogously for the mass conservation constraint and in $n=3$ dimensions.
 We will apply this approach to the problem of dynamic MPI reconstruction in the following section and thus use the forward operator of MPI as $A_t$. This operator and some of its properties will be derived in the following.

 \begin{remark}\label{rem:regularization_properties}
Analysis of the regularization properties of the unconstrained approach \eqref{eq:min_prob2} might be carried out by exploiting the established theory for Tikhonov-type regularization in Banach spaces in the nonlinear problem setting (see, e.g., \cite{hofmann2007convergence, schuster2012regularization}). 
PDE-constraint parameter identification in the dynamic setting has also been considered as an all-at-once approach to dynamic inverse problems from a more general point of view \cite{Kaltenbacher17,Tram19}.
A more general investigation in this direction is beyond the scope of the present work but an extended discussion of this relation can be found in Appendix~\ref{app:extended_discussion_rem}.
\end{remark}
 
 \subsection{Joint reconstruction for MPI}
 \label{subsec:jointMPIreco}
%
We aim for an application to the imaging modality magnetic particle imaging and thus need an mathematical formulation of the forward operator.
The functionality is based on the non-linear response of magnetic nanoparticles to an applied magnetic field. Modeling the nanoparticles' magnetization behavior sufficiently accurate for imaging tasks when using Lissajous-type excitations is still an open research challenge, but a simplified model, the equilibrium model (see the survey paper \cite{Kluth2018a}), has been extensively used to analyze the system behavior. In line with this, we exploit the equilibrium model to verify necessary requirements of the proposed approach. In order to derive the forward operator of MPI as well as some properties, we start by giving more detailed insights into the basic principles of MPI, see also \cite{knopp2012MPIbook}. 

In this work, we consider an MPI scanner consisting of a time independent magnetic field (selection field) $H_S:\R^3 \to \R^3$, which has a field-free point (FFP) in the center of the imaging device, see \cref{fig:1a}. Other geometries like a field-free line can be considered but are omitted for simplicity. This field is superimposed by time dependent magnetic fields (drive fields) $H_D:[0,T] \to \R^3$, which shift the FFP in space through the field of view $\Omega \subset \R^3$ (FOV), see \Cref{fig:1b}. 
By superposition, we describe the effective magnetic field $H:\R^3 \times [0,T] \to \R^3 $ built by an MPI scanner by  
\begin{equation*}
	H\left( x,t\right) = H_D\left( t\right) + H_S\left( x\right)  .
\end{equation*}
We assume that the selection field is a linear gradient field (in good approximation), meaning it exists a matrix $G \in \mathbb{R}^{3 \times 3}$ such that $
	H_S\left( x\right) = Gx.
$
As we consider an FFP MPI device, the matrix $G$ needs to be a full-rank matrix. Other geometries like a field-free-line scanner \cite{weizenecker2008magnetic} are not considered in this work. The trajectory $x_s:[0,T] \to \R^3 $ of the FFP can be described by 
\begin{align*}
	0 	&=  H_D\left( t\right) + H_S\left( x\right)  
	=  H_D\left( t\right) + G x_s\left( t\right) \\
	\Leftrightarrow x_s\left( t\right) &= -G^{-1}H_D\left( t\right) .
\end{align*}
\begin{figure}[tbhp]
    \centering
    \subfloat[]{\label{fig:1a}
        \includegraphics[width=0.3\textwidth]{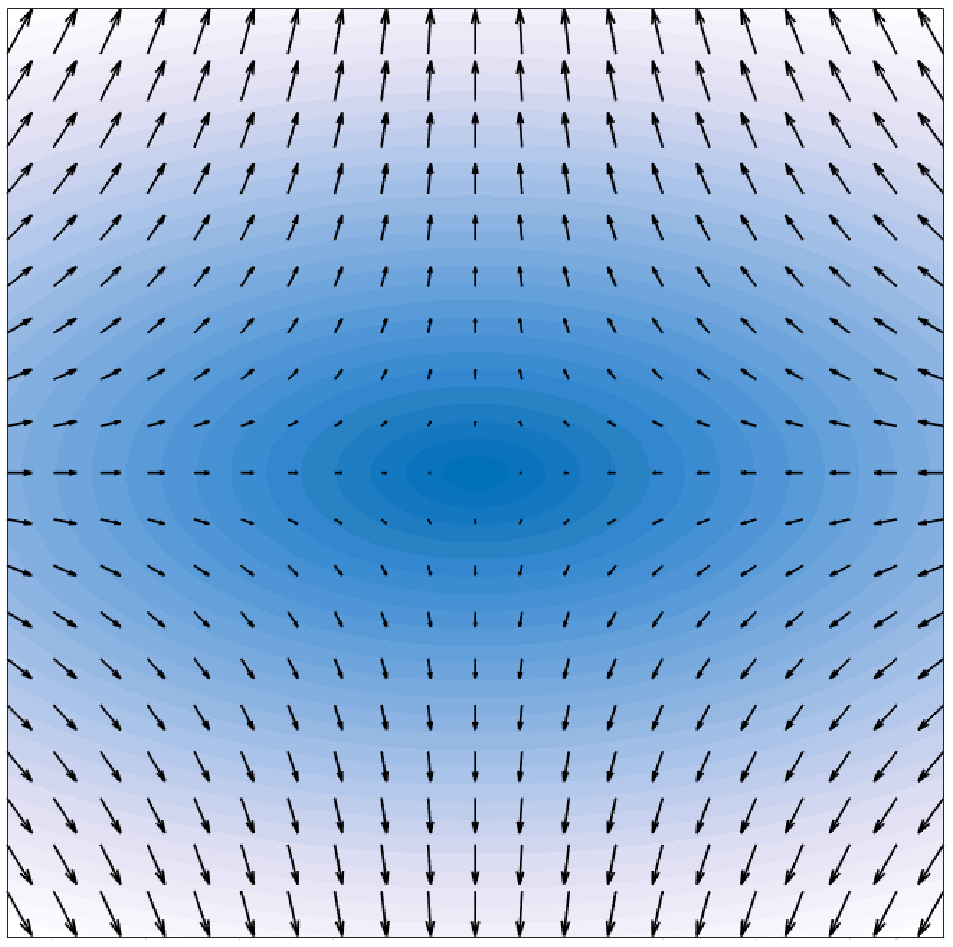}}
    \subfloat[]{\label{fig:1b}
      \includegraphics[width=0.3\textwidth]{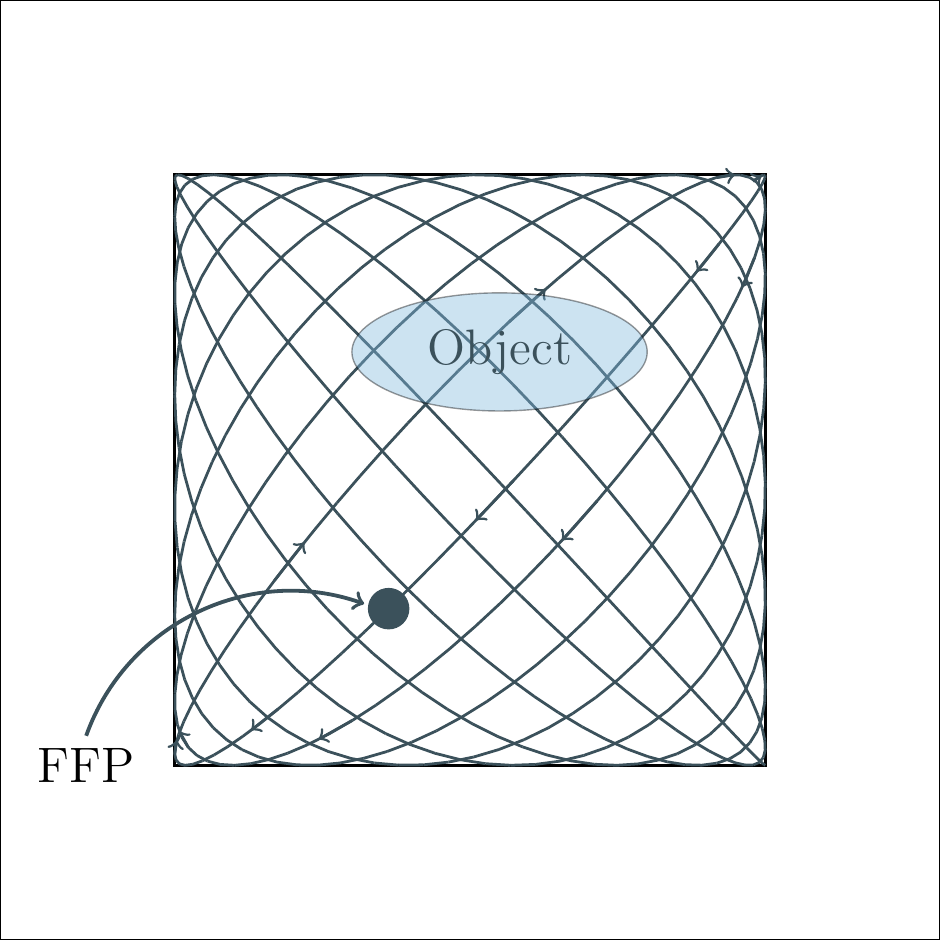}}
    \caption{The time independent selection field $H_S$ in an MPI scanner has a field-free point (FFP) in the middle and increasing field strength towards the boundaries (\cref{fig:1a}). The FFP is shifted through the region of interest by time-dependent drive-fields. It typically follows a Lissajous-type curve (\cref{fig:1b}).}
    \label{fig:selField_lissajous}
\end{figure}
Using 
	$H\left(x,t\right)  = G\left( x- x_s\left( t\right) \right) \,$,
we can compute the time derivative of the magnetic field as 
$
	\dot{H}\left( x,t\right) = -G\dot{x}_s\left( t\right) .
$

Typically, by the use of sinusoidal excitation functions, the movement of the FFP follows Lissajous curves. One cycle of the FFP along that curve is performed within the repetition time $T_R$. We note that in this paper $T_R \ll T$, i.e., we consider that the object is imaged in multiple cycles of the Lissajous curve. The change of magnetization caused by the FFP movement, induces a voltage signal by law of induction \cite{knopp2012MPIbook}. This signal is measured by the receive coils of the scanner and then used for image reconstruction. 

The static MPI forward operator $A : L^{\hat{p}}\left(0,T; L^l\left(\Omega\right)\right) \times \left[ 0,T\right]  \rightarrow \left( L^2\left(0,T;Y \right)\right) ^L$ for $t \in \left[ 0,T\right] $ for $L\in \mathbb{N}$ receive coil units is defined by 
\begin{equation}
\begin{aligned}
	A_i\left( c_s,t\right) = &- \int_\Omega c_s\left( x\right) \mu_0 m_0 R_i^T \left[ \left( 
		\frac{\mathcal{L}'_{\beta} \left( \left\| H\left( x,t\right) \right\|_2\right) }
		{\left\| H\left( x,t\right) \right\|_2^2 } - 
		\frac{\mathcal{L}_{\beta} \left( \left\| H\left( x,t\right) \right\|_2\right)}
		{\left\| H\left( x,t\right) \right\|_2^3 } \right) 
		H\left( x,t\right) H\left( x,t\right)^T \right.\\
		&+ \left. 
		\frac{\mathcal{L}_{\beta} \left( \left\| H\left( x,t\right) \right\|_2\right)}
		{\left\| H\left( x,t\right) \right\|_2  }I_3 \right]\dot{H}\left( x,t\right) \dx{x}
\end{aligned} 
\label{eq:static_forward_model}
\end{equation}
for $i=1,...,L$, where $\Omega \subset \mathbb{R}^{3}$ denotes the FOV, $c_s: \Omega \rightarrow \mathbb{R}^{+}$ denotes the static particle concentration of tracer material, $\mu_0$ describes the permeability constant, $m_0$ is the absolute of a particle's magnetic moment, $R \in \mathbb{R}^{3\times L}$ denotes the spatially homogeneous receive coil sensitivities ($R_i$ being the $i$-th column of $R$) and $\mathcal{L}_\beta$ describes the dilated Langevin function, which is part of the used equilibrium model assumed as the particle magnetization model. By $I_3$ we denote the $\mathbb{R}^{3x3}$ identity matrix.
The (dilated) Langevin function $\mathcal{L}_\beta : \mathbb{R} \rightarrow \mathbb{R}$ is defined by 
\begin{equation*}
    \mathcal{L}_\beta \left(z\right) = \left \lbrace \begin{array}{ll}
        \coth{\left(\beta z \right)} - \frac{1}{\beta z}, & \mathrm{for} \quad z\neq 0 \\
        0, & \mathrm{else},
    \end{array} \right. 
\end{equation*}
for a given positive parameter $\beta$.

This forward model for static MPI in 3 spatial dimensions can be similarly expressed by
\begin{equation}
	u\left( t\right) = \int_{\Omega} s(x,t) c_s(x) \,\mathrm{d}x
\end{equation}
in the time domain, where $u: \left[0,T \right] \rightarrow \mathbb{R}^L$ denotes the induced signal.  
The system function $s: \Omega \times \left[0,T \right] \rightarrow \mathbb{R}^L$ is described by
\begin{equation}\label{eq:system_function_per_coil}
	s_i\left( x,t\right) = -\mu_0 R_i^T\frac{\partial \bar{m}(x,t)}{\partial t}
\end{equation}
 for $i=1,...,L$, and $\bar{m}: \Omega \times \left[ 0,T\right] \rightarrow \mathbb{R}^3$ denotes the mean magnetic moment of the magnetic nanoparticles, i.e. 
 \begin{equation*}
 	\bar{m}\left( x,t\right) = m_0\mathcal{L}_{\beta}\left( \left\| H\left( x,t\right) \right\|_2 \right) \frac{H\left( x,t\right) 	}{ \left\| H\left( x,t\right) \right\|_2}.
 \end{equation*}

We introduce this simplified notation in order to shorten the derivation of the dynamic MPI forward model in the following. 
 
\begin{remark}
When measuring more than one drive-field cycle, the acquired signal is $T_R$-periodic in the static case. In case of low particle concentrations, the measurements are often averaged over some drive-field cycles to get a more robust signal with a higher signal-to-noise ratio (SNR) to reconstruct with, as the method of MPI defines a severely ill-posed problem. Note that we need measurements from at least one complete cycle of length $T_R$ to reconstruct data on the full FOV.
\end{remark}

We continue with certain properties of the forward operator, where we consider a specific choice of $A_t$, respectively $A_{i,t}$ being discussed in more detail in \cref{subsec:dynamicMPI} on the dynamic MPI problem. By $I_{t_s}\subset \left[0,T\right]$ we denote a compact interval with nonzero measure that includes the time point $t_s$. We will specify the interval more precisely when describing the different time scales in MPI.

\begin{lemma}[Compactness]
\label{lem:compactness}
	Let $I_{t_s} \subset \left[0,T\right]$ be compact and with nonzero measure for arbitrary time point $t_s$. Further let $\Omega \subset \mathbb{R}^3$ be simply connected and bounded. Moreover, let  $\left\| R_i\right\|_2 \neq 0$ for $i=1,...,L$, $x_s \in \mathcal{C}^1\left( I_{t_s}\right)^3 $, and $\dot{x}_s \in \mathcal{C}_b\left( I_{t_s}\right)^3$, i.e. $\dot{x}_s$ continuous and bounded. Assume that $G\in\R^{n\times n}$ is regular 
	and $c \in L^l\left(\Omega  \right) $ for arbitrary $l>1$.
	
	Then the operators $A_{i,t_s}: L^l\left( \Omega\right)  \rightarrow L^2\left( I_{t_s}\right) $, $i=1,\hdots,L$, with $c \mapsto \int_\Omega c(x) s_i(x,t) \dx{x}, \  t \in I_{t_s},$  are compact.  
\end{lemma}
\begin{proof}
Let the operator inside the square brackets in \cref{eq:static_forward_model} be denoted by $F\left(x,t\right)$ to reduce the notational complexity.
The MPI forward operator for the $i-th$ receive coil unit is defined by a linear Fredholm integral equation of first kind with the kernel $k: \Omega \times I_{t_s} \rightarrow \mathbb{R}$ given by 
	\begin{equation*}
		k\left( x,t\right) = -\mu_0 m_0 R_i^T F\left( x,t\right) G\dot{x}_s\left( t\right) \quad \mathrm{for} \enspace t \in I_{t_s}.
	\end{equation*}
By \cite[Theorem 4.1.]{Kluth2018b}, we deduce that $k \in H^{0}\left( I_{t_s}; L^{\infty}\left( \Omega \right) \right)$. The assumptions required for the theorem are fulfilled in our setting, as we shortly outline in the following. 
By $x_s \in \mathcal{C}^1\left( I_{t_s}\right) $, we know that the drive-field $H_D\left(t \right) \in \mathcal{C}^1\left( I_{t_s}\right)$ and therefore $H_D\left(t \right) \in H^1\left( I_{t_s}\right)$. Moreover, the selection field $H_S\left( x\right) = Gx$ for a full rank matrix $G$ fulfills $H_S \in L^{\infty}\left( \Omega \right)^n $ and the receive coil sensitivity $R$ also fulfills $R \in L^{\infty}\left( \Omega \right)^n$. Thus \cite[Equation (4.3)]{Kluth2018b} yields $k \in H^{0}\left( I_{t_s}; L^{\infty}\left( \Omega \right) \right) = L^2\left( I_{t_s}; L^{\infty}\left( \Omega \right) \right)$, i.e., $k \in L^2\left( I_{t_s}; L^{l^*}\left( \Omega \right) \right)$ for any $1 \leq l^\ast \leq \infty$. 

We now consider $k$ as a Hilbert-Schmidt integral operator with 
\begin{equation}
     \left\| k \right\| := \left( \int_{I_{t_s}}\left( \int_\Omega \left| k\left( x,t\right) \right| ^{l^*} \dx{x} \right) ^{\frac{2}{l^*}} \dx{t} \right) ^{\frac{1}{2}} = \|k\|_{L^2\left(I_{t_s}; L^{l^*}\left( \Omega\right) \right)} < \infty, 
\end{equation}
with $\frac{1}{l} + \frac{1}{l^*} = 1$, $1<l<\infty$. Then by standard results from functional analysis (cf. \cite[Section 5.12]{alt2016linear}) it follows that the operator 
\begin{equation}
   \left( A_{t_s} c\right)\left(t \right) = \int_\Omega k\left(x,t\right) c\left(x\right) \dx{x} 
\end{equation}
defines a compact operator from $L^l\left( \Omega \right)$ to $L^2\left(I_{t_s}\right)$.
\end{proof}

\begin{remark}
    By the compactness of the MPI forward operator, it is clear that the image reconstruction task is ill-posed. The degree of ill-posedness in terms of decay of the singular values was analyzed in \cite{Kluth2018b}. In a standard setting, using the equilibrium magnetization model, trigonometric FFP trajectories and a linear selection field, the singular values decay exponentially yielding a severely ill-posed problem.
\end{remark}

\begin{remark}
    Although \cref{lem:compactness} is formulated for $n=3$ dimensions, it can be transferred to lower dimensional MPI. This can be achieved by assuming the concentration to be Dirac $\delta$-distributed with respect to the orthogonal complement of the lower dimensional affine subspace of $\mathbb{R}^3$. A detailed description and derivation can be found in \cite[Section 2.2]{Kluth2018b}. 
\end{remark}

Our goal is now to prove the regularity assumption on the forward operator needed for \cref{thm:existence}, i.e. $A_{t_s} \chr{\Omega} \neq 0$ for arbitrary time points $t_s$. In order to do so, we carry out some analysis on a shifted domain. For fixed $t$, we use the transformation 
\begin{equation*}
	\Phi_t\left( \xi\right) := G^{-1}\xi + x_s\left( t\right). 
\end{equation*}
Applying the change of variables formula to the forward operator leads to
\begin{align*}
	\left( A_{t_s} \chr{\Omega}\right) \left( t\right)& = - \int_{\Omega} \mu_0 m_0 R^T F\left( x,t\right) G\dot{x}_s\left( t\right) \dx{x}\\
	&=- \int_{\Phi_t^{-1}\left( \Omega\right) } \mu_0 m_0 R^T F\left( \Phi_t\left( \xi\right),t\right)\left| \mathrm{det} \nabla \Phi_t\left( \xi\right)\right|  G\dot{x}_s\left( t\right) \dx{\xi}\\
	&=- \int_{\Omega_t } \mu_0 m_0 R^T \tilde{F}\left( \xi\right) \left|\mathrm{det} G^{-1}\right|  G\dot{x}_s\left( t\right) \dx{\xi}
\end{align*}
with $\Omega_t:= \Phi_t^{-1}\left( \Omega \right) = G\left(\Omega -  x_s\left( t\right)\right) $, $\nabla \Phi_t\left(\xi\right) = G^{-1}$ and 
\begin{equation}
\label{eq:def_ftilde}
	\tilde{F}\left( \xi\right) := \left( 
	\frac{\mathcal{L}'_{\beta} \left( \left\| \xi \right\|_2\right) }
	{\left\| \xi \right\|_2^2 } - 
	\frac{\mathcal{L}_{\beta} \left( \left\| \xi \right\|_2\right)}
	{\left\| \xi \right\|_2^3 } \right) 
	\xi \xi^T + 
	\frac{\mathcal{L}_{\beta} \left( \left\| \xi \right\|_2\right)}
	{\left\| \xi \right\|_2  }I_3.
\end{equation}
For this matrix we can verify the following auxiliary lemma.
\begin{lemma}[Positive definiteness]
	Let $\xi \in \mathbb{R}^3\backslash\left\lbrace 0\right\rbrace $ be arbitrary. Then $\tilde{F}\left( \xi\right) \in \R^{3\times 3} $ is positive definite. 
\end{lemma}
\begin{proof}
	We consider $\tilde{F}$ as defined in \cref{eq:def_ftilde}.
	Let $\xi \in \mathbb{R}^3\backslash\left\lbrace 0\right\rbrace $ and $x  \in \mathbb{R}^3$ arbitrary. 
    Consider first 
	\begin{align*}
		x^T \left( 
		\frac{\mathcal{L}'_{\beta} \left( \left\| \xi \right\|_2\right) }
		{\left\| \xi \right\|_2^2 } \right.& \left.- 
		\frac{\mathcal{L}_{\beta} \left( \left\| \xi \right\|_2\right)}
		{\left\| \xi \right\|_2^3 } \right) 
		\xi \xi^T x 
		 = \left(\mathcal{L}'_{\beta} \left( \left\| \xi \right\|_2\right) - 
		\frac{\mathcal{L}_{\beta} \left( \left\| \xi \right\|_2\right)}
		{\left\| \xi \right\|_2 }\right) \frac{\left(x^T \xi\right)^2}{\left\| \xi \right\|_2^2} . 
	\end{align*}
	Moreover,
	\begin{align*}
		x^T \frac{\mathcal{L}_{\beta} \left( \left\| \xi \right\|_2\right)}
		{\left\| \xi \right\|_2  }I_3 x 
		&= \frac{\mathcal{L}_{\beta} \left( \left\| \xi \right\|_2\right)}{\left\| \xi \right\|_2  } x^T x,
	\end{align*} 
    yields
	\begin{align*}
		x^T \tilde{F}\left( \xi\right) x & = \left(\mathcal{L}'_{\beta} \left( \left\| \xi \right\|_2\right) - 
		\frac{\mathcal{L}_{\beta} \left( \left\| \xi \right\|_2\right)}
		{\left\| \xi \right\|_2 }\right) \frac{\left(x^T\xi\right)^2}{\left\| \xi \right\|_2^2} + \frac{\mathcal{L}_{\beta} \left( \left\| \xi \right\|_2\right)}{\left\| \xi \right\|_2  }\left \Vert x \right \Vert_2^2 \\
		&= \underbrace{\mathcal{L}'_{\beta} \left( \left\| \xi \right\|_2\right)}_{>0} \underbrace{\frac{\left(x^T\xi\right)^2}{\left\| \xi \right\|_2^2}}_{\geq 0}  + \underbrace{\frac{\mathcal{L}_{\beta} \left( \left\| \xi \right\|_2\right)}{\left\| \xi \right\|_2 }}_{>0}\underbrace{
		\left(\left \Vert x \right \Vert_2^2 - \frac{\left(x^T\xi\right)^2}{\left\| \xi \right\|_2^2}\right)}_{\geq 0, \enspace \mathrm{as} \enspace \left(x^T\xi\right)^2 \leq \left\Vert \xi\right\Vert_2^2\left\Vert x\right\Vert_2^2}, 
	\end{align*}
	where we used the positivity of the Langevin function everywhere but in $0$, the positivity of the derivative of the Langevin function and the Cauchy-Schwarz inequality. 
	We observe the following: If $\frac{\left(x^T\xi\right)^2}{\left\| \xi \right\|_2^2} = 0$, it follows that $\left(x^T\xi\right)^2 = 0$ and thus $\left(\left \Vert x \right \Vert_2^2 - \frac{\left(x^T\xi\right)^2}{\left\| \xi \right\|_2^2}\right) \neq 0$ if $x\neq 0$. \\
    If $\left(\left \Vert x \right \Vert_2^2 - \frac{\left(x^T\xi\right)^2}{\left\| \xi \right\|_2^2}\right) = 0$, then $\frac{\left(x^T\xi\right)^2}{\left\| \xi \right\|_2^2} \neq 0$ for $x\neq 0$. 
	We have thus shown 
	\begin{equation*}
		x^T \tilde{F}\left( \xi\right) x >0 \qquad \mathrm{for} \enspace x \neq 0.
	\end{equation*}
 
\end{proof}

This allows us now to verify the desired regularity property of the forward operator.
\begin{theorem}[Regularity of the MPI forward operator]
	Let $I_{t_s} \subset \left[0,T\right]$ be compact and with nonzero measure for arbitrary time point $t_s$, $\Omega \subset \mathbb{R}^3$ be simply connected and bounded. Assume that the coil sensitivities $R_i \in \mathbb{R}^3, \enspace i = 1,...,3$, fulfill
	\begin{equation*}
		\mathrm{span}\left\lbrace R_i, \enspace i=1,...,3\right\rbrace = \mathbb{R}^3.
	\end{equation*}
	Moreover, let $x_s \in \mathcal{C}^1\left( I_{t_s}\right)$ and assume there exists an inner point $t^* \in I_{t_s}$ such that $\left\| \dot{x}_s\left( t^*\right)\right\|_{2} \neq 0$. 
	Let the operator $A_{t_s}: L^l\left( \Omega\right)  \rightarrow L^2\left( I_{t_s}\right)^3 $, be given by $c \mapsto \left( \int_\Omega c(x) s_i(x,t) \dx{x}\right)_{i=1,\hdots,3}, \  t \in I_{t_s},$ according to \eqref{eq:system_function_per_coil}.
	\\
	Then it holds that $$\left\| A_{t_s} \chr{\Omega}\right\|_{L^2(I_{t_s})^3} >0$$ 
	and thus $A_{t_s} \chr{\Omega} \neq 0$.
\end{theorem}
\begin{proof}
	In the following, we denote $y_s := Gx_s$. As $\left\lbrace R_i, \enspace i=1,...,3\right\rbrace$ forms a basis of $\mathbb{R}^3$, it holds that 
	\begin{equation*}
		\forall t \in I_{t_s}  :\quad  \exists q_i \in \mathbb{R}, \enspace i=1,...,3: \quad  
		\sum_{i=1}^3 q_i R_i = \dot{y}_s\left( t\right). 
	\end{equation*} 
	Therefore, it exists $\tilde{q}_i \in \mathcal{C}\left( I_{t_s} \right), \enspace i=1,...,3 : \quad \sum_{i=1}^3 \tilde{q}_i R_i = \dot{y}_s$. Setting $T(R):=\left[ R_1 \enspace R_2 \enspace R_3\right]$, continuity of those functions $\tilde{q}_i$ follows from 
 	\begin{equation*}
		\left[ R_1 \enspace R_2 \enspace R_3\right] \tilde{q} = \dot{y}_s \quad \Rightarrow \tilde{q} = T(R)^{-1} \dot{y}_s, 
	\end{equation*} 
	as $\left\| T(R)^{-1}\right\| <\infty$. 
	For all $t$ with $\dot{y}_s\left( t\right) \neq 0$, it holds that 
	\begin{equation}
		\sum_{i=1}^3 \int_{\Omega_t } \tilde{q}_i \left( t\right) R_i^T \tilde{F}\left( \xi\right) \dot{y}_s\left( t\right) \dx{\xi } = \int_{\Omega_t }\underbrace{\dot{y}_s\left( t\right) ^T \tilde{F}\left( \xi\right) \dot{y}_s\left( t\right)}_{>0\enspace \forall \xi} \dx{\xi } >0. 
	\end{equation}
	For $t^* \in I_{t_s}  $ with $\dot{y}_s\left( t^*\right) \neq 0$ there exists, by the continuity of $\dot{x}_s$, $\varepsilon >0$ such that 
	\begin{align*}
		0 & < \int_{\left( t^*-\varepsilon, t^* + \varepsilon\right) } \left( \sum_{i=1}^3 \tilde{q}_i \left( t\right) \underbrace{\int_{\Omega_t }R_i^T \tilde{F}\left( \xi\right) \dot{y}_s\left( t\right) \dx{\xi}}_{:=u_i\left( t\right) }\right)^2 \dx{t} 
		 = \left\| \sum_{i=1}^3 \tilde{q}_i u_i\right\|_{L^2\left( t^*-\varepsilon,\enspace t^* + \varepsilon\right)} ^2 \\
		& \leq 4 \sum_{i=1}^3 \left\| \tilde{q}_i u_i \right\|_{L^2\left( t^*-\varepsilon,\enspace t^* + \varepsilon\right)} ^2 
		 \leq  4 \sum_{i=1}^3 \left\| \tilde{q}_i u_i \right\|_{L^2\left( I_{t_s} \right) } ^2 \\
		& \leq 4C_{\tilde{q}} \sum_{i=1}^3 \left\| u_i\right\| ^2_{L^2\left( I_{t_s} \right) } 
		 = 4C_{\tilde{q}} \left\| u\right\|^2_ {L^2\left( I_{t_s} \right) },
	\end{align*}
	where $C_{\tilde{q}}>0 $ denotes an upper bound on $\left\| \tilde{q}_i\right\|_{L^2\left( I_{t_s} \right) }$ for $i=1,..,3$, which exists as $\tilde{q_i} \in \mathcal{C}\left( I_{t_s}  \right)$.
	Denoting now 
	\begin{align}
	    \left( A_{t_s} \chr{\Omega}\right) \left( t\right)& =
	    - \int_{\Omega_t } \mu_0 m_0 R^T \tilde{F}\left( \xi\right) \left|\mathrm{det} \enspace G^{-1}\right|  G\dot{x}_s\left( t\right) \dx{\xi} 
	     = -\mu_0 m_0 \left|\mathrm{det} \enspace G^{-1}\right| u\left( t\right) 
	\end{align}
	yields the result
	  $\left\| A_{t_s} \chr{\Omega}\right\|_{L^2(I_{t_s} )} >0  $
 and thus 
$    A_{t_s}\chr{\Omega} \neq 0.$
\end{proof}
\begin{remark}
A standard MPI scanner usually consists of orthogonal receive coils such that $ \left\lbrace R_i, i=1,...,3 \right\rbrace$ forms even an orthogonal basis of $\mathbb{R}^3$. 
\end{remark}
\begin{remark}
    By assuming $\left\| \dot{x}_s\left(t\right) \right\|_2 \neq 0$ for all $t \in \left[0,T\right]$, it can also be shown by the same arguments that for the particular choice of $I_{t_s}=\{t_s\}$ the operator $A_{t_s}: L^l(\Omega) \to \R^3 $ with $c \mapsto \left( \int_\Omega c(x) s_i(x,t_s) \dx{x}\right)_{i=1,\hdots,3}$ fulfills  $(A_{t_s}\chr{\Omega}) \neq 0$ for all $t_s \in \left[0,T\right]$. For our application, this stronger assumption can be fulfilled by choosing an MPI scanner with sinusoidal excitation frequencies for the drive-field. This choice results in cosine functions for $\dot{x}_s\left(t\right)$, which do not share zeros within the timespan $\left[0,T\right]$ by the  choice of the frequency dividers. However, cosine excitation functions lead to sinusoidal derivatives, which share a zero at $t=0$ and thus do not fulfill the stronger assumption in general. 
\end{remark}
We have thus shown that the static MPI forward operator for a scanner with sinusoidal excitation fulfills the regularity assumption $ \left(A_{t_s} \chr{\Omega}\right) \neq 0$ as well as boundedness of the operators $A_{t_s}$ for desired $t_s$, which is necessary to prove existence of a minimizer in a joint motion estimation and image reconstruction problem. 
Note that the consideration of measurements in $L^2(I_{t_s})$ for time points $t_s$ already indicates that in the context of MPI we need to carefully consider time on different scales as MPI is characterized by a time-dependent measurement process. If one would assume motion in $c$ and the measurement process providing a measurement at each time on the same temporal scale, this would require reconstructing an entire spatially resolved image from one measurement in $\R^3$.
The issue of different temporal scales will be addressed in the following part on the dynamic MPI problem.

\subsection{Dynamic MPI}
\label{subsec:dynamicMPI}
In the following, we will derive the dynamic MPI forward model in order to understand model inaccuracies and modeling errors, which occur in dynamic MPI. 

Extending the MPI forward model to a dynamic tracer distribution $c: \Omega \times \left[0,T\right] \rightarrow \mathbb{R}^{+}$ and assuming a homogeneous coil sensitivity leads to 
\begin{align}
	u_i\left( t\right) &= -\mu_0 R_i^T \frac{\mathrm{d}}{\mathrm{d}t} \int_{\Omega} c(x,t)\bar{m}(x,t) \mathrm{d}x \notag \\
	&= -\mu_0 R_i^T \int_{\Omega} c(x,t) \dd{t}\bar{m}(x,t) + \bar{m}(x,t)\dd{t}c(x,t)\mathrm{d}x,
	\label{eq:fullVWModel}
\end{align}
for $i=1,..,L$.
This model was first mentioned in \cite{Gdaniec.2020} but was applied to the special case of periodic motions. It was proposed in \cite{brandt2021modeling} to use the full dynamic model \cref{eq:fullVWModel} for image reconstruction. However, although the reconstructions are of high quality, they are computationally too demanding to be implemented in practice \cite{brandt2021modeling}. 
Moreover, in the case of a fully calibrated system function it is still an unsolved problem how to obtain the full dynamic model. 
Limiting ourselves to dynamic distributions whose temporal change provides negligible contributions to the signal, yields
\begin{equation}
	u_i\left( t\right) \approx -\mu_0 R_i^T \int_{\Omega} c(x,t) \dd{t}\bar{m}(x,t) \mathrm{d}x = \int_{\Omega} s_i(x,t) c(x,t) \mathrm{d}x.
	\label{eq:simpModel} 
\end{equation}
This is a common assumption in dynamic MPI reconstruction which has also been used for periodic motion using virtual frames \cite{Gdaniec.2020}. In the following, we will reconstruct dynamic concentrations with non-periodic motion by suppressing artifacts by joint reconstruction of concentration and motion.


One important aspect which need to be taken into account in this context is the interplay of different temporal scales, a fine temporal scale of the measurement process itself, i.e., we need to consider time-series of voltage measurements to gather the spatial information of the image, and the temporal scale of the motion.
For this we denote the longest duration, for which the tracer distribution remains approximately static by $\Delta t$, i.e., we exploit the approximation 
\begin{equation}
	c(x,t) \approx c(x,t_s) \qquad \forall t \in \left[ t_s-\Delta t, t_s  \right], \forall x \in \Omega.
\end{equation}
Typically, a reconstruction of the concentration $c\left( x,t_s\right)$ at time $t_s$ requires at least data from a complete drive-field cycle, i.e., $\Delta t \geq T_R$. 
This poses no problem if the motion is sufficiently slow.
For $\Delta t\geq T_R$ we have a quasi-static tracer distribution that changes only slightly during one drive-field cycle and can be handled as static data in good approximation. However, if $\Delta t < T_R$ one has to consider more severe motion artifacts. 
Depending on the actual setting it can be seen as a tradeoff between potentially emerging motion artifacts and loss of spatial information.
In the following we restrict ourselves to the assumption $\Delta t = T_R$ which guarantees encoding the maximum amount of spatial information in the desired MPI setting but where with potentially higher probability motion artifacts emerge. 
As a result we consider
\begin{equation}
	\left(A_{t_s} c\right) \left( t \right)  = \int_{\Omega}  s\left( x,t\right) c\left( x,t_s\right) \mathrm{d}x, \quad t \in I_{t_s}:=\left[ t_s-T_R,t_s \right],
	\label{eq:fwop}
\end{equation}
for timepoint $t_s\geq T_R$ to reconstruct  $c(x,t_s)$ from $u$ obtained on $I_{t_s}$. In order to reduce the effects of this approximation in our reconstruction, we jointly reconstruct images and the motion in between time frames such that both tasks will endorse each other. If images and motion are sufficiently linked and arising gaps in the data are sufficiently interpolated, one might also use sub-frame reconstruction in case of dynamic data (cf. \cref{fig:timescales}). In this approach, we use measurements obtained during an interval of maximum length $\Delta t< T_R$ for reconstruction of the concentration which on its own will not be sufficient but can be so, if subsequent frames are tied strong enough.  

\begin{figure}[tbhp]
	\centering
	\includegraphics[width=\textwidth]{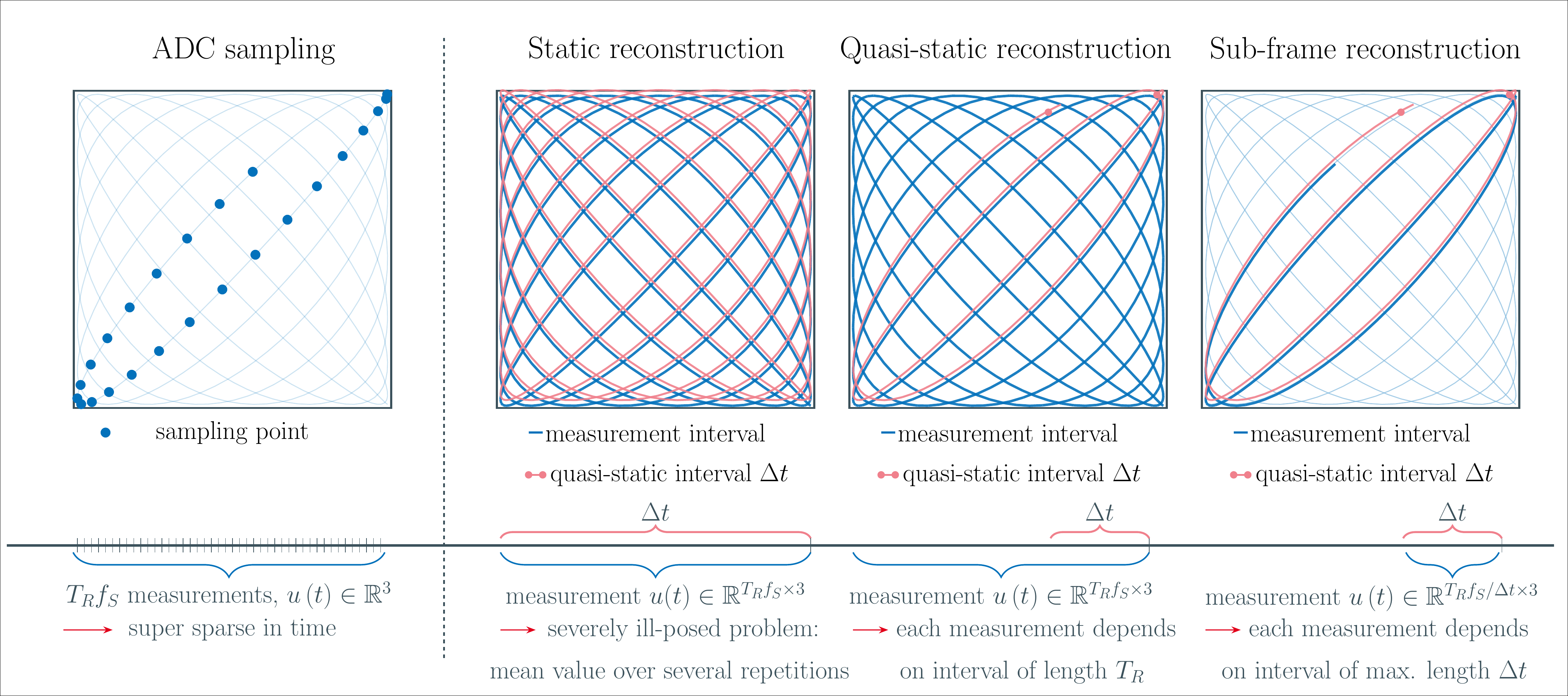}
	\caption{The finest time scale in MPI reconstruction is defined by the sampling rate of the analog-to-digital converter (ADC), which samples measurements at a frequency $f_S$ during the repetition time $T_R$. In static reconstruction, we pool all measurements obtained during one drive-field cycle and then average over the drive-field cycles in order to gain data with high SNR. The measurement interval is at most as long as the quasi-static interval. In dynamic reconstruction, the quasi-static interval $\Delta t$ can be smaller than $T_R$. Still, we need to use data from a full drive-field cycle in order to reconstruct the full FOV (cf. \cref{eq:fwop}). To reduce the computational complexity and reduce motion artifacts, one could use sub-frame reconstruction in combination with coupling of reconstructions at the different time steps through motion. Pooled measurements are then obtained during an interval of maximum length $\Delta t$.}
	\label{fig:timescales}
\end{figure}

For numerical reconstruction, we need discretization in space and time. For the time discretization, the finest time scale we can use is defined by the analog-to-digital converter which converts the analog measured signal to a digital one. The different time scales as described above are illustrated in Figure \ref{fig:timescales}, we are in the quasi-static setting, if not mentioned explicitly. The discretization in space is automatically obtained by the calibration of the system matrix.

We note that typical regularization terms for the image reconstruction task in MPI are covered by assumption \cref{eq:desC}. Depending on the specific application for which MPI is used, either the optical flow constraint (i.e. instrument tracking) or the mass conservation constraint (i.e. blood flow visualization) can be an appropriate choice for the motion model. In any case, existence of a minimizer is guaranteed by \cref{thm:existence}.  Before we state our numerical results for MPI reconstruction, we give some more details on the algorithms used to solve the minimization problem \cref{eq:min_prob2}.

\section{Algorithm}
\label{sec:algorithm}
To solve problem \cref{eq:min_prob2}, i.e. 
\begin{equation*}
\min_{c,v} \int_{0}^{T} D\left( A\left( c, t\right), u\left( t\right) \right) + \alpha R\left( c(\cdot,t)\right) + \beta S\left( v(\cdot,t)\right) + \gamma T\left( c\left( \cdot,t\right),\dd{t}c\left( \cdot,t \right) ,v\left( \cdot, t\right) \right)  \mathrm{d}t\,,
 \end{equation*}
 we apply alternating minimization, i.e. we solve for tracer concentration $c$ and motion (displacement field) $v$ in between subsequent frames alternately following the idea of \cite{Burger.2018}. We consider the algorithmic approaches to the two subproblems in the following sections. We derive the algorithmic details for $n=3$ spatial dimensions which then includes lower dimensional problems. 
\subsection{The motion estimation subproblem}
\label{subsec:motion_estimation_subproblem}
For a given concentration sequence $c$, the motion estimation subproblem is defined by 
 \begin{equation}
 	\label{eq:motion_est} \min_{v} \int_{0}^{T} \beta S\left( v(\cdot,t)\right) + \gamma T\left( c\left( \cdot,t\right), \dd{t}c\left( \cdot,t \right) ,v\left( \cdot, t\right) \right)  \mathrm{d}t.
 \end{equation}
 We derive an algorithm for each of the considered motion models. 
 %
\subsubsection{The optical flow motion model}
We start with the optical flow constrained problem. For a spatially but not time-dependent regularization term $S$ our minimization problem reads  
	\begin{equation}
	    \label{eq:probofme}
		\min_v  \int_0^T \beta S\left( v \left( \cdot, t\right)\right) + \gamma\left\| \dd{t}c\left(\cdot, t \right) + \nabla c\left(\cdot, t \right) \cdot v\left(\cdot, t\right)\right\|_{L^1\left( \Omega\right) }  \mathrm{d}t,
	\end{equation}
	where the $L^1$-norm could be easily replaced by an $L^2$-norm. However, we choose an $L^1$-norm as to allow for outliers, which might exist due to noise artifacts in the reconstructed image sequences. 
	The regularization functional $S$ is assumed to be proper, lower semicontinuous and convex as well as prox-tractable. For our application, we usually consider total variation (TV), i.e. $L^1$-, or $L^2$-regularization of the gradient, respectively, i.e. 
 	\begin{align}
		S_1\left( v\left(\cdot, t\right)\right) &= \left\| \nabla v \left(\cdot, t \right) \right\|_{L^1(\Omega)^n} \qquad \mathrm{or} \qquad
		S_2\left( v\left( \cdot, t \right)\right) = \frac{1}{2}\left\| \nabla v \left(\cdot, t\right) \right\|_{L^2(\Omega)^n}^2 .
	\end{align}
	In order to solve \cref{eq:probofme} with regularization functional $S = S_1$, we apply primal-dual splitting with primal functional $f$ and  $g$ to be Fenchel-conjugated defined by
	\begin{align*}
		f\left( v\right) &= \int_0^T \gamma \left\| \dd{t} c\left( \cdot, t \right) + \nabla c\left(\cdot, t \right) \cdot v\left(\cdot, t \right)\right\|_{L^1\left( \Omega\right) } \mathrm{d}t\\
		g\left( Bv\right) &= \int_0^T\beta \left\| \nabla v\left( \cdot, t \right) \right\|_{L^1(\Omega)^n} \mathrm{d}t.
	\end{align*}
with the linear operator 
	\begin{equation*}
		B = \left( \begin{array}{ccc}
			\nabla & 0 & 0 \\ 0 & \nabla &0\\ 0 & 0 & \nabla
		\end{array}\right) 
	\end{equation*}
	leading to the adjoint operator 
	\begin{equation*}
		B^* = \left( \begin{array}{ccc}
			-\mathrm{div} & 0 &0 \\ 0 & -\mathrm{div} &0\\ 0 & 0 & \mathrm{-div}
		\end{array}\right) .
	\end{equation*}
	The corresponding Fenchel-conjugated functional is given by 
	\begin{equation*}
		g^*\left( y_1, y_2, y_3\right) = \int_0^T \beta \left( \ind{\left\lbrace \left\|y_1\right\| _\infty \leq \beta\right\rbrace  } +  \ind{\left\lbrace \left\|y_2\right\| _\infty \leq \beta\right\rbrace  } +  \ind{\left\lbrace \left\|y_3\right\| _\infty \leq \beta\right\rbrace  } \right) \mathrm{d}t. 
	\end{equation*}

 %
	\subsubsection{The mass conservation motion model}
	The motion estimation problem in this case given is by 
	\begin{equation}
		\min_v \int_{0}^{T} \frac{\beta}{2} \left\| \nabla v \left(\cdot, t \right)\right\|_{L^2\left( \Omega\right)^n } ^2 + \gamma \left\| \dd{t}c\left(\cdot, t \right) + \nabla \cdot \left( c\left(\cdot, t \right)\cdot v\left(\cdot, t \right)\right) \right\|_{L^1\left( \Omega\right) }  \mathrm{d}t, 
	\end{equation}
	where we choose an $L^2$-norm on the gradient in order to derive the algorithmic scheme for both regularizers. In the numerical part of this manuscript, we will state clearly which regularization is used and we might combine any of the regularizers with each motion model.  
	The problem is again solved by primal-dual splitting with 
	\begin{align*}
		f\left( v\right) & = 0\\
		g\left(Bv\right) &  =  \int_0^T \frac{\beta}{2} \left\| \nabla v\left(\cdot, t \right)\right\|_{L^2\left( \Omega\right)^n } ^2 + \gamma \left\| \dd{t}c\left(\cdot, t \right) + \nabla \cdot \Bigl( c\left(\cdot, t \right)\cdot v\left(\cdot, t \right)\Bigr) \right\|_{L^1\left( \Omega\right) } \mathrm{d}t\\
	\end{align*}
	with the linear operator 
	\begin{equation*}
		B = \left( \begin{array}{ccc}
			 \nabla & 0 &0\\ 0 &\nabla &0 \\ 0 &0 & \nabla\\ \partial_x c & \partial_y c & \partial_z c \\
		\end{array}\right) 	
	\end{equation*}
	using the notation $(\partial_x c)(v) = \dd{x}(cv)$ and its corresponding adjoint operator 
	\begin{equation*}
	B^* = \left( \begin{array}{cccc}
		-\mathrm{div} & 0 & 0& -c\partial_x \\ 0 &-\mathrm{div}&0 & -c\partial_y \\ 0 & 0 & -\mathrm{div} & -c\partial_z
	\end{array}\right) .	
	\end{equation*}
	This leads to the Fenchel-conjugate 
 	\begin{multline*}
 	    g^*\left( y_1,y_2,y_3,y_4\right) = \int_0^T \frac{1}{2\beta} \left( \left\|y_1\right\|_{L^2(\Omega)^n}^2 + \left\|y_2\right\|_{L^2(\Omega)^n}^2 + \left\|y_3\right\|_{L^2(\Omega)^n}^2 \right) \\+ \gamma \left( \ind{\left\lbrace \left\| y_4\right\|_\infty \leq \gamma\right\rbrace  } - \left\langle \dd{t}c,y_4\right\rangle \right) \mathrm{d}t.
		\end{multline*}\\
	For both motion models, we can now state the saddle point problem equivalent to the original minimization problem
 \begin{equation*}
     \min_v f\left(v\right) + g\left(Bv\right).
 \end{equation*} 
 More particularly, the saddle point problem is given by 
 \begin{equation*}
     \min_v \sup_y \left \langle Bv,y\right \rangle -g^*\left(y\right) + f\left(v\right).
 \end{equation*}
 In imaging applications, such a problem is often solved by the primal-dual hybrid gradient (PDHG) method introduced by \cite{pock2009algorithm}.\\
\subsubsection{PDHG algorithm for motion estimation}
	Using the above notations, the main motion estimation algorithm is given by \cref{alg:PDHG}. 
	\begin{algorithm}[ht]
		\caption{PDHG algorithm for motion estimation}
		\label{alg:PDHG}
		\begin{algorithmic}[1]
			\STATE \textbf{input} initial value $\mathbf{v}^0 = \left( u,v,w\right) $, $\tilde{v}^0 = v^0$, initial value $y^0$, $\theta$, stepsize parameters $\sigma$, $\tau$, operators $B$, $B^*$  \\
			\FOR{k=0,1,2,...,K-1}
			\STATE $\bar{y}^{k+1} = y^k + \sigma B\tilde{\mathbf{v}}^k$\\
			\STATE $y^{k+1} = \mathrm{prox}_{\sigma g^*}\left( \bar{y}^{k+1}\right) \quad \quad $ \textit{\% dual update}\\
			\STATE $\bar{\mathbf{v}}^{k+1} = \mathbf{v}^k - \tau B^*y^{k+1}$\\
			\STATE $\mathbf{v}^{k+1} = \mathrm{prox}_{\tau f}\left( \bar{\mathbf{v}}^{k+1}\right) \quad \quad $  \textit{\% primal update }\\
			\STATE $\tilde{\mathbf{v}}^{k+1} = \mathbf{v}^{k+1} + \theta \left( \mathbf{v}^{k+1} - \mathbf{v}^{k}\right)  \quad \quad$ \textit{\% extrapolation step}\\
			\ENDFOR
		\end{algorithmic}
	\end{algorithm}\\
	The proximal operator $\mathrm{prox}_{\tau f}$, which occurs in the algorithm, is defined by 
 \begin{equation*}
    \mathrm{prox}_{\tau f}\left(x\right) = \arg \min_u f\left( u\right) + \frac{1}{2\tau} \left\| u-x\right\|^2\,, 
 \end{equation*}
 where the norm is chosen according to the underlying spaces. 
	\subsubsection{Operator discretization}
	As proposed by \cite{Burger.2018}, we use forward and backward differences for discretization of the differential operator and its adjoint, respectively, for the spatial regularization terms. Within the motion constraint, we use central differences, which are self-adjoint, for the spatial derivatives and forward differences for the time derivatives to obtain a stable scheme. 
	%
	\subsubsection{Coarse-to-fine strategy and warping}
	\label{sec:warping}
	The problem of handling large-scale motion is two-fold. 
	First, the linearized optical flow constraint is derived by a Taylor approximation, which will hold only for motion of small absolute value. 
	Second, the absolute value of our motion estimates is bounded by the choice of our discrete differential operator. \\
	There are two  different basic approaches to the multiscale strategy. One either computes only motion increments on each scale and warps images in between the different scales \cite{Brox.2004} or one uses the results from coarser scales for initialization but then updates the whole value in each iteration \cite{MeinhardtLlopis.2013}. \\
	\textbf{The optical flow motion model: Multiscale and warping}
	We use the first approach in case of the optical flow constraint as it was shown that the warping technique implements the non-linearized optical flow equation \cite{Brox.2004}.  The process of motion estimation with a coarse-to-fine strategy and warping is illustrated in \cref{fig:warping}. \\
    First, the image sequence is downsampled to the different scales yielding an image pyramid. Then, beginning on the coarsest scale, the motion in between succeeding frames is computed and prolongated to the next finer scale. It is then applied to the image of the second time step by warping. The whole procedure is then repeated on each scale for the motion increment until the finest scale is reached.  
	Note that image warping can be implemented by a forward mapping or a reverse mapping. When using a forward mapping, we iterate over the pixels of the source image and compute the new position $(x,y,z)_{new}$ after the motion is applied. The new position is then a floating point and we use linear interpolation to re-grid the image. However, when using the forward mapping many source pixels can map to the same destination and some pixels in the destination image might not be covered at all. For this reason, warping is usually defined as a reverse mapping. We then iterate over the pixels of the destination image and find the origin $(x,y,z)_{old}$ in the source image for every pixel. The gray value to be transported to the destination image is then again determined by linear interpolation. 
 
 Reverse warping thus in a sense fulfills the gray value constancy assumption and can be combined with the optical flow constraint, as that constraint will lead to invertible flow fields. However, in the case of mass conservation as motion model, there might exist sources and sinks and the flow field is then not invertible. \\
	\begin{figure}[tbhp]
 \centering
       
    \includegraphics[width=0.8\textwidth]{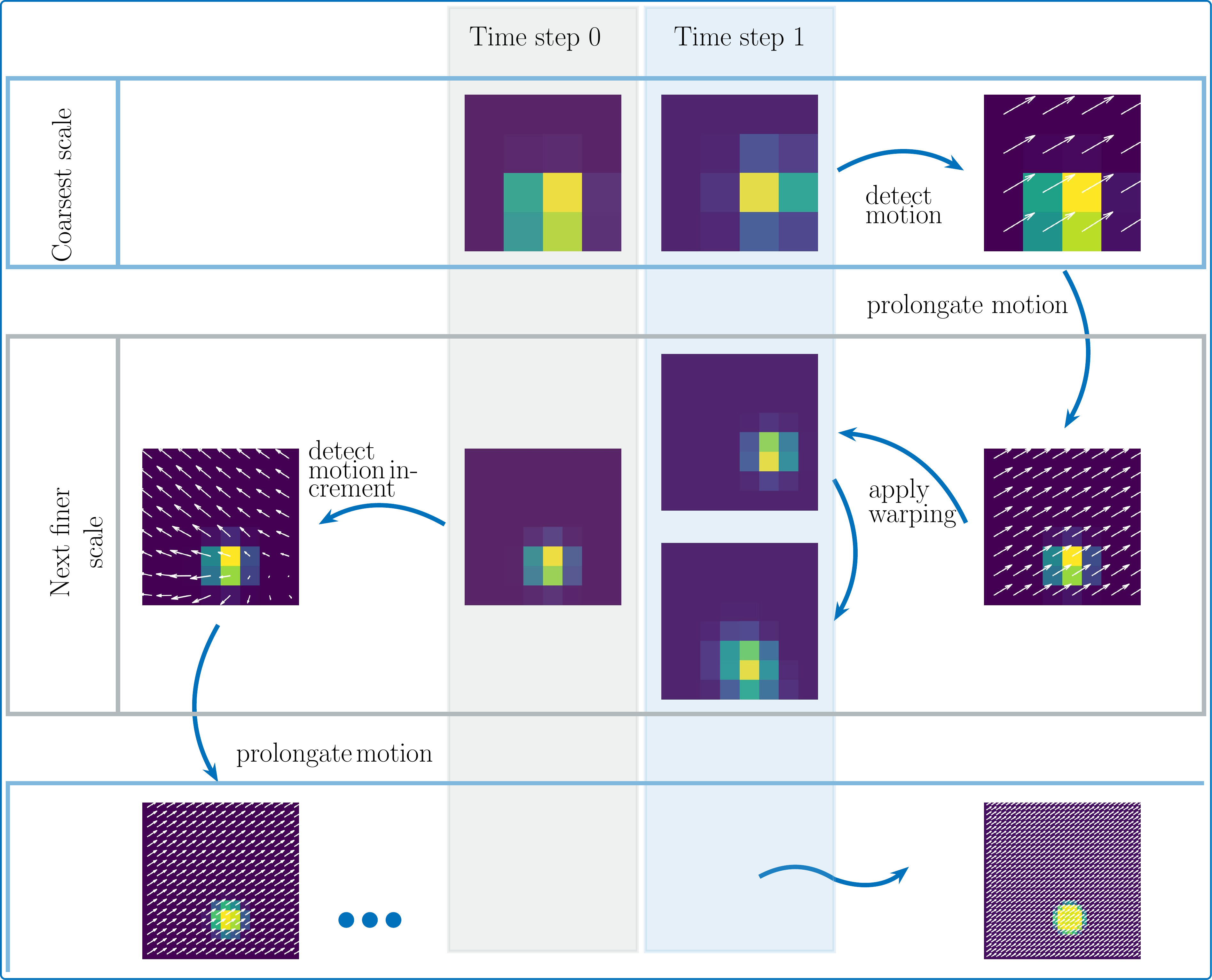}
		\caption{The multi-resolution framework in combination with warping works as follows. Starting from the coarsest scale, motion in between the images corresponding to two succeeding time steps is detected by a motion estimation algorithm. The flow field is then prolongated to the next finer scale and applied to the image of the second time step in terms of warping. These steps are repeated on each scale. }
		\label{fig:warping}
	\end{figure}%
\textbf{The mass conservation motion model: Multiscale only}	
In case of the mass conservation constraint, we cannot use warping due to the aforementioned reasons. However, we need a multi-scale strategy in order to deal with large motion estimates. We thus implement a coarse-to-fine strategy, where the motion estimate on each scale is initialized by the result of the last scale and updated after the increment is computed. 

The algorithm used on each scale is PDHG algorithm as explained in \Cref{subsec:motion_estimation_subproblem}. This allows for straight-forward integration of estimates on previous scales. 

\subsection{The image reconstruction subproblem}
The image reconstruction subproblem for a given motion estimate $v$ is defined by 
 \begin{equation}
 	\label{eq:im_reco} \min_{c} \displaystyle\int_{0}^{T} D\left( A\left( c, t\right), u\left( t\right) \right) + \alpha R\left( c(\cdot,t)\right) + \gamma T\left( c\left( \cdot,t\right),\dd{t}c\left( \cdot,t \right) ,v\left( \cdot, t\right) \right)  \mathrm{d}t\,.
 \end{equation}
We use Fused Lasso regularization as our default regularization on the concentration as that fits the expected behavior of the tracer material in MPI very well \cite{storath2016edge, zdun2021fast}. We thus have 
\begin{equation}
\label{eq:regularization_imagereco}
	\alpha R\left( c(\cdot,t)\right) = \alpha_1 \left\| c\left( \cdot,t\right) \right\| _{L^1\left( \Omega\right) } + \alpha_2\left\| \nabla c\left( \cdot, t\right) \right\|_{L^1\left( \Omega\right)^n }.  
\end{equation}
For the data fitting term, we choose either the $L^1$- or $L^2$-norm discrepancy \cite{KluthJin2020}, i.e. 
\begin{align*}
	D_1\left( A \left( c, t\right) , u\left( t\right) \right) &= \left\| A\left(c, t\right)-u\left( t\right)\right\|_{L^1\left( \mathbb{R}^d\right) },\\  
 	D_2\left( A \left( c, t\right) , u\left( t\right) \right) &= \frac{1}{2}\left\| A\left(c, t\right)-u\left( t\right)\right\|^2_{L^2\left( \mathbb{R}^d\right) }.
\end{align*}
 Again, we design one algorithm for each motion model.
 %
 \subsubsection{The optical flow motion model}
 As we expect the motion to exceed one pixel per frame, the linearized constraint
\begin{equation*}
		m_1\left( c,v\right) = \dd{t}c  + \nabla c \cdot v  =0
\end{equation*}
does not hold. 

We therefore implement the non-linearized version of the gray-value constancy assumption according to \eqref{eq:conv_gray_value}, i.e. in the semi-discretized setting with respect to time step $\delta_t$ we use 
\begin{equation*}
	T\left( c\left( \cdot, t\right),c\left( \cdot, t+\delta_t\right) ,v \left(\cdot, t \right)\right) = \left\| c(\cdot+\delta_t v\left(\cdot, t\right),t+\delta_t) - c(\cdot,t) \right\|_{L^1\left( \Omega\right)}= \left\| \mathcal{W}c\left( t\right) \right\|_{L^1\left( \Omega\right)}.    
\end{equation*}
with 
\begin{equation}
\left( 	\mathcal{W}c\right) \left( t\right) := -c\left( \cdot,t\right) + W^{t,\delta_t}c\left( \cdot,t+\delta_t\right) ,
\end{equation}
where $W^{t,\delta_t}$ maps a space-dependent image to a space-dependent image performing the reverse warping of the input image with respect to the displacement field estimate $v(\cdot,t)$.
Similar to the motion estimation subproblem, reverse warping (iterating over the destination image) is performed in order to cover all pixels of the destination image.

Applying primal-dual splitting to this setting and using the same notations as in the motion estimation subproblem leads to
\begin{align}
	f\left( c\right) &= \int_0^T \alpha_1 \left\| c\left( \cdot, t \right) \right\| _{L^1\left( \Omega\right) } \mathrm{d}t\\
	g\left( B c\right) & = \int_0^T \left\| A\left( c, t\right)-u\left(t\right)\right\|_{L^1\left( \mathbb{R}^d\right) } + \alpha_2\left\| \nabla c\left(\cdot, t \right) \right\|_{L^1\left( \Omega\right)^n } + \gamma \left\| \left(\mathcal{W} c\right) \left( t\right) \right\|_{L^1\left( \Omega\right)} \mathrm{d}t 
\end{align}
with the linear operator $	B = \left( \begin{array}{c}
		A, \nabla  , \mathcal{W}
	\end{array}\right)^T$
and its adjoint 
	$B^* = \left( \begin{array}{ccc}
		A^*, & -\mathrm{div},  & \mathcal{W}^*
	\end{array}\right).  	$
The corresponding Fenchel-conjugate is given by 
\begin{equation}
\label{eq:dualmotionest}
	g^*\left( y_1,y_2,y_3\right) = \int_0^T \ind{\left\lbrace \left\| y_1\right\|_\infty \leq 1\right\rbrace } + \left\langle y_1,u\right\rangle + \ind{\left\lbrace \left\|y_2\right\|_{3,\infty} \leq \alpha_2 \right\rbrace } + \ind{\left\lbrace \left\| y_3\right\|_\infty \leq \gamma\right\rbrace  } \mathrm{d}t.
\end{equation}
\subsubsection{The mass conservation constraint}
The mass conservation constraint 
\begin{equation*}
	m_2\left( c,v\right) = \dd{t}c + \nabla \cdot \left(cv \right) = 0.
\end{equation*}
is not linearized. We therefore use the constraint directly inside our penalty term, i.e. 
\begin{flalign*}
	T\left( c\left(\cdot, t \right), \dd{t}c\left(\cdot, t \right),v\left(\cdot, t \right)\right) &= \left\| \dd{t}c\left(\cdot, t \right) + \nabla \cdot \Bigl( c\left( \cdot, t \right)v\left(\cdot, t \right) \Bigr)\right\|_{L^1(\Omega)} \\
	&= \left\| \dd{t}c\left(\cdot, t \right) + c\left(\cdot, t \right)\nabla \cdot \Bigl( v \left(\cdot, t \right) \Bigr) + v\left(\cdot, t \right) \cdot \nabla c\left(\cdot, t \right) \right\|_{L^1(\Omega)} . 
\end{flalign*}
Applying again the notation from the motion estimation subproblem leads to 
\begin{align*}
	f\left( c\right) & =\int_0^T \alpha_1 \left\| c\left(\cdot, t \right) \right\| _{L^1\left( \Omega\right) } \mathrm{d}t\\
	g\left( Bc\right) &= \int_0^T \left\| A\left( c,t\right)-u\left(t\right)\right\|_{L^1\left( \mathbb{R}^d\right) } + \alpha_2\left\| \nabla c\left(\cdot, t\right) \right\|_{L^1\left( \Omega\right)^n }\\
	&+ \gamma \left\| \dd{t}c\left(\cdot, t \right) + c\left(\cdot, t \right)\nabla \cdot \Bigl( v\left(\cdot, t \right) \Bigr) + v\left(\cdot, t \right) \cdot \nabla c\left(\cdot, t \right) \right\|_{L^1(\Omega)} \mathrm{d}t , 
\end{align*}
where 
\begin{equation*}
		B = \left( \begin{array}{c}
		A\\ \nabla  \\ \dd{t} + v_1\dd{x} + v_2\dd{y} + v_3\dd{z} + \left( \nabla\cdot v\right)  I
	\end{array}\right),
\end{equation*}
and 
\begin{equation*}
	B^* = \left( \begin{array}{ccc}
		A^*, & -\mathrm{div},  & -\left( \dd{t} + v_1\dd{x} + v_2\dd{y} + v_3\dd{z} - \left( \nabla\cdot v\right)  I\right) 
	\end{array}\right).  	
\end{equation*} The Fenchel-conjugate corresponding to $g$ is the same as in \cref{eq:dualmotionest}. \\
\subsubsection{Stochastic PDHG}
To reduce the run time of the algorithm, we use the stochastic variant of PDHG introduced in 2018 by Chambolle et. al. \cite{chambolle2018stochastic} to solve the problem for both motion constraints. We split the dual variable $y_1$ related to the data fitting term row-wise into $M$ variables as $y_1 = \left( y_{1,1}, y_{1,2},...,y_{1,M} \right)^T$. This is possible due to the separability of the MPI forward operator. Accordingly, the linear operators $B_{1,1}, ... , B_{1,M}$ are defined by row-wise splitting of the forward operator, i.e. $A = B_1 = \left(B_{1,1}, B_{1,2}, ..., B_{1,M} \right)^T$. By default, we set $M=3$, relating the data batches to the three receive channels of a standard MPI scanner. For more details on appropriate data splitting for MPI data, we refer to \cite{zdun2021fast}. The resulting algorithm is stated in \cref{alg:SPDHG}. Note that $\sum_{j=1}^3 B_j^*y_j$ in the primal update step can also be updated instead of fully computed in each iteration. In this way, the computational burden of the algorithm is reduced, as $B_i^*y_i$ has to be computed for the updated index $i$ only, in case of $i=1$ even for $B_{1,m}^*y_{1,m}$ only if the $m$-th data batch is chosen. 

\begin{algorithm}[ht]
		\caption{SPDHG algorithm for image reconstruction}
		\label{alg:SPDHG}
		\begin{algorithmic}[1]
			\STATE \textbf{input} initial value $c^0 $, initial values $y_i^0$, $i \in \left\lbrace 1,...,3\right\rbrace$, $\theta$, stepsize parameters $\sigma$, $\tau$, operators $B_i$, $B^*$  \\
			\FOR{$k=0,1,2,...,K-1$}
            \STATE select {$j \in \left\lbrace1,2,3\right\rbrace $ randomly} \\
            \IF{$j=1$}
			\STATE \text{select} $m \in \left\lbrace1,...,M\right\rbrace$ \text{randomly}\\
			\STATE $\bar{y}_{1,m}^{k+1} = y_{1,m}^k + \sigma B_{1,m}\tilde{c}^k$\\
			\STATE $y_{1,m}^{k+1} = \mathrm{prox}_{\sigma g_{1}^*}\left( \bar{y}_{1,m}^{k+1}\right) \quad \quad $ \textit{\% dual update}\\
            \ELSE
			\STATE $\bar{y_j}^{k+1} = y_j^k + \sigma B_j\tilde{c}^k$\\
			\STATE $y_j^{k+1} = \mathrm{prox}_{\sigma g_{j}^*}\left( \bar{y_j}^{k+1}\right) \quad \quad $ \textit{\% dual update }\\
            \ENDIF
			\STATE $\bar{c}^{k+1} = c^k - \tau \sum_{j=1}^4 B_j^*y_j^{k+1}$\\
			\STATE $c^{k+1} = \mathrm{prox}_{\tau f}\left( \bar{c}^{k+1}\right) \quad \quad \hspace{0.8cm} $ \textit{\% primal update}\\
			\STATE $\tilde{c}^{k+1} = c^{k+1} + \theta \left( c^{k+1} - c^{k}\right) \quad \hspace{0.1cm} $ \textit{\% extrapolation step }\\
			\ENDFOR
		\end{algorithmic}
	\end{algorithm}
   
\section{Numerical experiments}
\label{sec:numerical}
Our numerical experiments are divided into three sections. We first use simulated data in a 3D setting in order to be able to compare our reconstructions to a ground truth. We can thus quantify improvements by the joint approach in the image reconstruction task as well as in the motion estimation task compared to the standard MPI reconstruction scheme. 
The second section then uses 3D measured data of a phantom that fulfills the optical flow constraint. Measurement sequences are available for different motion speeds during measurements such that we can compare the impact of the joint approach for different velocities. In this section, we primarily focus on the image reconstruction task 
and we also show that obtained motion estimates are more reliable by the proposed approach compared to estimates from images reconstructed by the Kaczmarz method by comparing the trajectories derived from the obtained flow fields.  
In the third section, we reconstruct in-vivo measurement data of a mouse heart. This application fulfills mass conservation but not the optical flow constraint. In contrast to the previous section, we will focus on the resulting motion estimates. 

\subsection{Experiments on 3D simulated data}

In order to have a ground truth available, we use simulated data for these first experiments. The scanner was modeled based on the pre-clinical MPI scanner (Bruker Biospin, Ettlingen) at the University Medical Center Hamburg-Eppendorf. The complete parameter setup can be found in \cref{app:sim_par}.  
The phantom is based on the rotation phantom used in \cite{Gdaniec2017} and in our experiments in \cref{sec:NumrotationPhantom}, but we add a movement along the $z$-direction yielding a ball moving along a spiral. \\
We compare several reconstruction techniques. First, we use standard frame-by-frame static reconstructions obtained by the Kaczmarz method using standard Tikhonov $L^2$- regularization combined with a positivity constraint \cite{dax1993row}, which corresponds to one of the most popular MPI reconstruction techniques \cite{knopp2017review}. Moreover, we consider frame-by-frame reconstruction by SPDHG algorithm using fused lasso regularization. In the proposed joint reconstruction setting, we test several different formulations. They do all apply fused lasso regularization on the tracer concentration and total variation regularization on the motion estimates, but differ in terms of the data fitting norm and the motion model. We use both $L^1$- and $L^2$-data fitting (denoted as cases $L^1$-D and $L^2$-D) and optical flow (OF) and mass conservation (MC) constraint, resulting in a total amount of four different combinations. All joint approaches are tested on 10 frames simultaneously in the simulated data setting. The quality of the reconstructed image sequences is assessed by structural similarity index measure (SSIM). All experiments are performed for different noise levels, more particularly we add Gaussian noise with mean value 0 and standard deviation of $0\%, 50\%$ and $100\%$ of the maximum absolute value of the measured signal to the measurements (in the time domain) before transforming the measurements to the Fourier domain.  
Prior to reconstruction, we use a data pre-processing similar to the one typically used for measured data. We transform matrix and data to the Fourier domain and, instead of applying an SNR threshold, we limit the maximum mixing order of the frequencies, which yields a similar outcome \cite{szwargulski2017influence}. Moreover, the data is split into real and imaginary part in order to perform the reconstruction on real numbers like in \cite{storath2016edge}.

The noise levels added at early stages in the simulation as mentioned above correspond to noise with standard deviation of 0\%, 2.7\% and 5.3\% on the signal used for reconstruction. 
Note that the noise-free setting is used for comparison only, but is not of practical relevance. The high noise level reconstructions resemble what we expect of measured data, whereas the middle noise level might be achieved in an optimal experimental setup. Note that our simulated forward operator is not corrupted by noise, which defines a fundamental difference to the setting in practice, where the forward operator is typically measured. 
Moreover, our simulation setting uses ideal magnetic fields, is based on the simple Langevin model for magnetization and neglects further noise sources such as background signals such that the reconstruction task is performed even in case of noisy simulated data under ideal conditions compared to the reconstruction of measured data.

\begin{table}[tbhp]
\footnotesize
    \caption{Best obtained values for SSIM for reconstructed image sequences of simulated data.}
    \label{tab:quality_simdata}
    \centering   
    \begin{tabular}{|c|c|c|c|c|c|c|}
        \hline
        Noise level & Kaczmarz & SPDHG& OF-$L^1$-D & OF-$L^2$-D & MC-$L^1$-D & MC-$L^2$-D \\
        \hline
        0.0 & 0.654 &0.955& \textbf{0.963} & 0.948 & \textbf{0.963} & 0.946 \\
        \hline
        0.5 & 0.653 &0.954& \textbf{0.963} & 0.947 & 0.962 & 0.946 \\
        \hline
        1.0 & 0.652 &0.949& \textbf{0.957} & 0.944 & 0.956 & 0.943 \\
        \hline
    \end{tabular}
\end{table}

Best reconstruction parameters (obtaining highest SSIM values) are stated in \cref{app:par_search}.
The comparison in \cref{tab:quality_simdata} shows that the Kaczmarz method yields solutions with lowest SSIM values. The reconstructed image sequences with highest SSIM values are heavily smoothed and do not resemble the phantom closely, in particular the concentration values are approximately 10 are 5 times smaller compared to the phantom. The illustration in \cref{fig:sim_phant} thus shows a reconstruction which obtained a SSIM value of only 0.26 due to severe noise artifacts, but features reasonably high concentration values.  

Using a more suitable regularization method and SPDHG algorithm considerably improves the reconstructed image sequences even when using a frame-by-frame reconstruction approach. Best results are obtained for the proposed joint reconstruction technique for an $L^1$-D term.  \Cref{fig:sim_phant} shows exemplary averaged intensity projections of the reconstructed images. 

\begin{figure}[tbhp]
    \centering
    \includegraphics[width=0.91\textwidth]
    {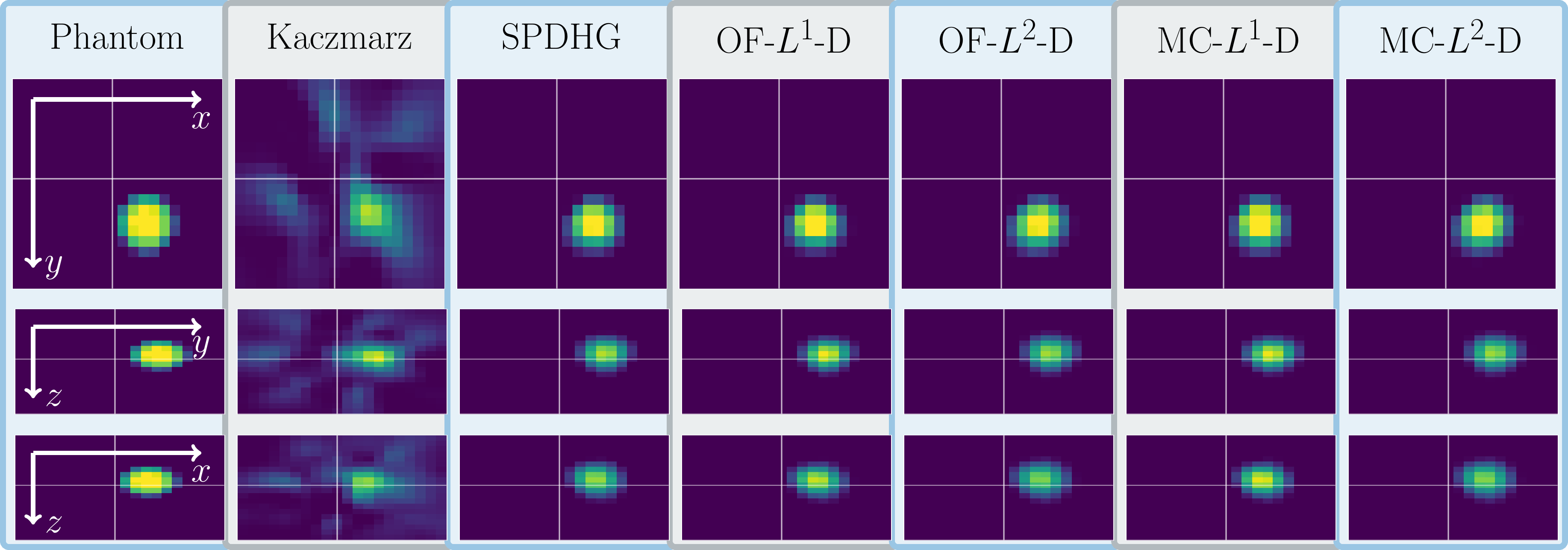}
    \centering\includegraphics[width=0.082\textwidth]{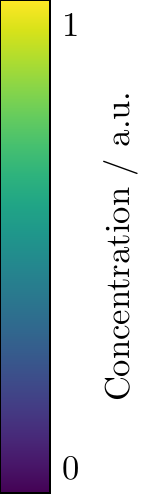}
    \caption{Reconstructed images of the simulated data for the second time step. The upper row depicts averaged intensity projections onto the $x-y$-plane, the middle row projects onto the $y-z$-plane and the bottom row onto the $x-z$-plane. Each column corresponds to one reconstruction algorithm. All images are plotted with the same colorbar viridis and windowed based on the full dynamic range of the phantom.}
    \label{fig:sim_phant}
\end{figure}

In a standard static setting, the poor image quality resulting from Kaczmarz method under noise is improved by averaging the measurements over multiple frames. This procedure yields a higher SNR in the measurements and thus reconstructed images of higher quality with less noise artifacts. However, if we apply frame averaging in the dynamic case, this results in severe motion artifacts as can be seen in \cref{fig:sim_av_artifacts}. For averaging of very few frames (5 in this example) the SNR is barely improved and motion artifacts occur already, but for averaging of 160 frames the motion artifacts make it impossible to locate the object as observed in \cite{Gdaniec2017} as well. 

 \begin{figure}[tbhp]
    \centering\includegraphics[width=0.5\textwidth]{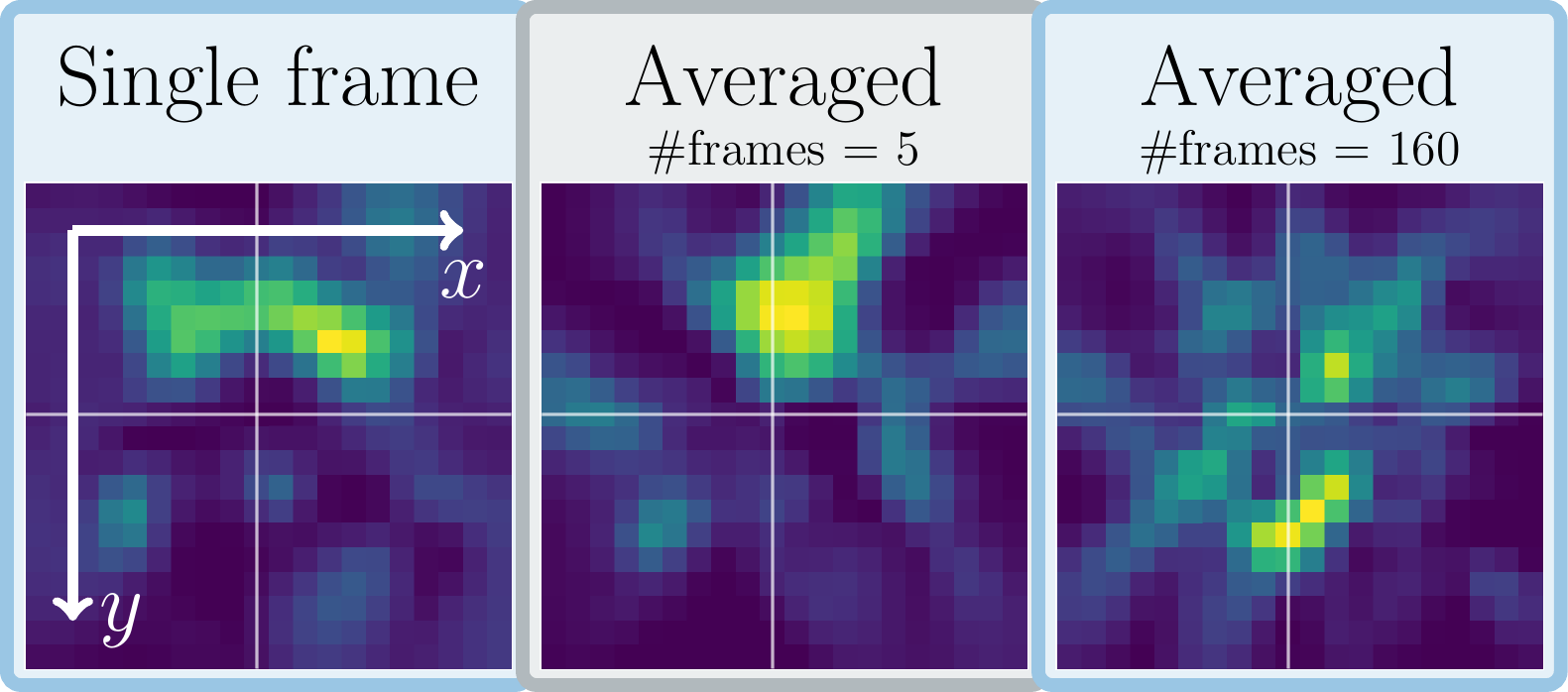}
    \centering\includegraphics[width=0.073\textwidth]{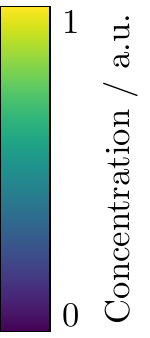}
    \caption{Averaged intensity projections of the reconstructed images of the rotation phantom for noise level with standard deviation 50$\%$ in the time domain. Reconstructions are obtained via Kaczmarz method using either single frame measurements as well as averaged measurements over multiple frames. The best regularization parameter $\lambda_{\mathrm{single}}$ weighting the Tikhonov regularization term was determined for the single frame reconstruction and then adjusted to the amount of averaging (using $\lambda_{\mathrm{single}} / \sqrt{\# \mathrm{frames}}$)  in order to stay proportional to the expected noise level. The amount of frame averaging is increased from left to right column.}
    \label{fig:sim_av_artifacts}
\end{figure}

\subsection{Experiments on a 3D optical flow like data set}
\label{sec:NumrotationPhantom}
The data used in this section was first published in \cite{Gdaniec2017}.
The measurements were taken by a pre-clinical MPI scanner (Bruker Biospin, Ettlingen). The scanner uses sinusoidal excitation frequencies and thus creates 3D Lissajous-type trajectories. The drive-field frequencies are given by $f_x = \qty{2.5}{\mega\hertz}/102$, $f_y = \qty{2.5}{\mega\hertz}/96$ and $f_z = \qty{2.5}{\mega\hertz}/99$ with an amplitude of \qty{14}{\milli\tesla\per \mu_0}, resulting in a repetition time of \qty{21.54}{\milli\second}. Frames are not averaged such that the temporal resolution is defined by the repetition time. The selection field gradient strength is \qty{0.75}{\tesla\per\meter\per\mu_0} in $x$- and $y$-direction and \qty{1.5}{\tesla\per\meter\per\mu_0} in $z$-direction. The delta sample used for system matrix acquisition was of size \numproduct{2x2x1} \unit{\cubic\milli\meter} and covers \numproduct{25x25x25} positions.

The phantom consists of a rotating class capillary and is constructed in the following way: Two round disks of the same diameter as the scanning bore are connected by three rods. An additional rod is placed in the center and a glass capillary filled with tracer material (\qty{20}{\micro\liter} diluted ferucarbotran with factor $1/10$) is attached to it. The central rod can be attached to a screwdriver and rotated during measurements. The central rod has a diameter of \qty{10}{\milli\meter} and the attached glass capillary has an inner diameter of \qty{1.3}{\milli\meter} and an outer diameter of \qty{1.7}{\milli\meter}. Considering the amount of tracer material, this yields a fill height of approximately \qty{15}{\milli\meter} in $x$ direction (corresponding to 7.5 voxel). The construction results in a circular path of the tracer material with a diameter of approximately \qty{11.5}{\milli\meter}. Image sequences for three different rotation speeds were obtained, the respective motion frequencies are \qty{1}{\hertz}, \qty{3}{\hertz} or \qty{7}{\hertz}. A more detailed description of the phantom as well as illustrations can be found in \cite{Gdaniec2017}.

We use a standard procedure for data pre-processing, applying a frequency selection based on an SNR threshold of 5, and cutting off frequencies lower than \qty{80}{\kilo\hertz} (see e.g., \cite{KluthJin2019a} for a formal description of the pre-processing). 
Moreover, we use a weighting approach by row normalization as introduced in \cite{knopp2010weighted}.

We compare the visual quality of our reconstructed image sequences to a sequence reconstructed by standard frame-wisely applied Kaczmarz algorithm using an $L^2$-Tikhonov regularization as well as to frame-wisely applied SPDHG algorithm for fused lasso regularization. 
Full image sequences are appended to the manuscript.
Since the phantom fulfills the optical flow constraint and thus both possible constraints, we will use it to compare both mass conservation and optical flow constraint. 
\begin{figure}[tbhp]
    \centering
    \noindent\includegraphics[width=0.92\textwidth]{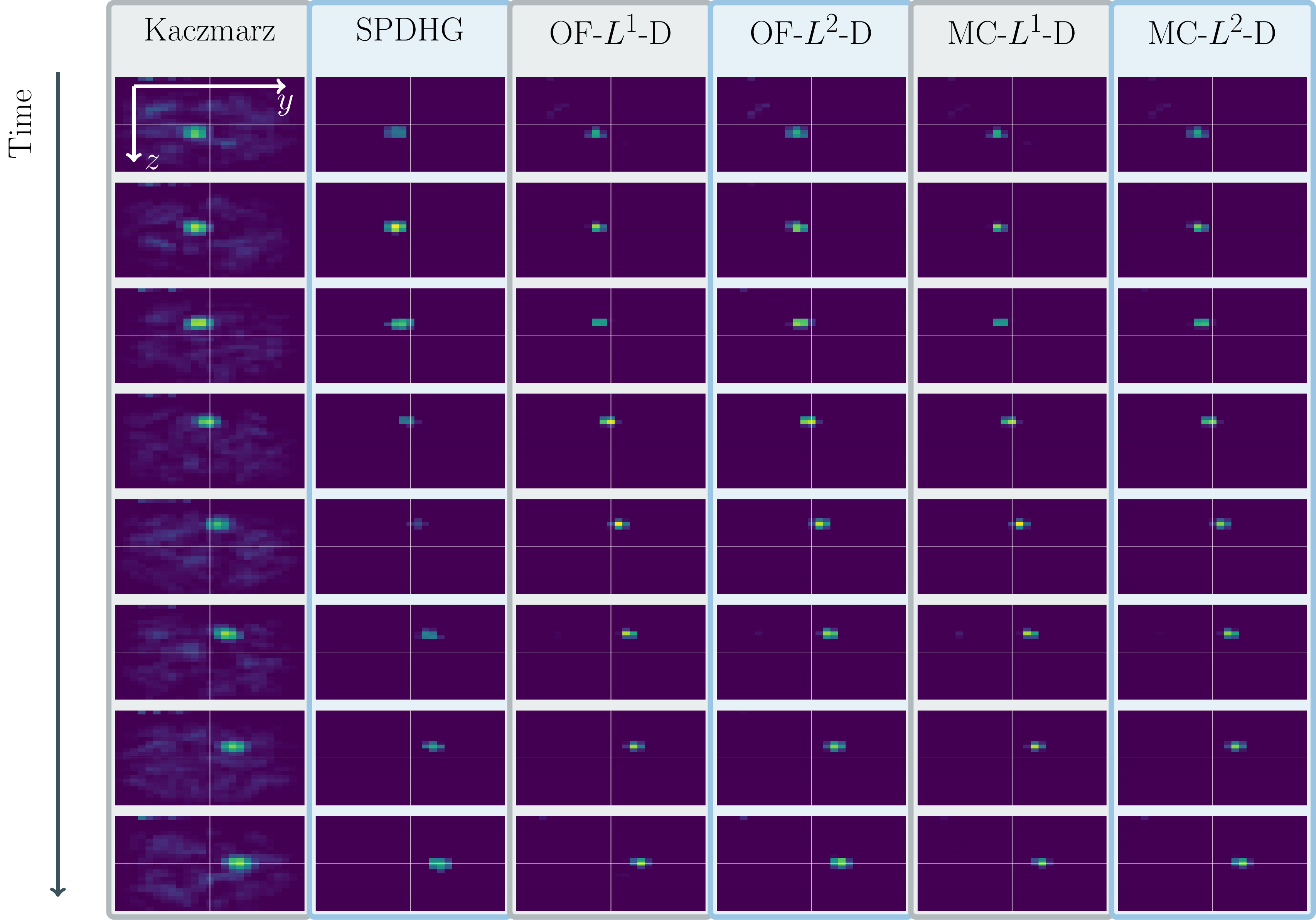}
     \noindent\includegraphics[width=0.0635\textwidth]{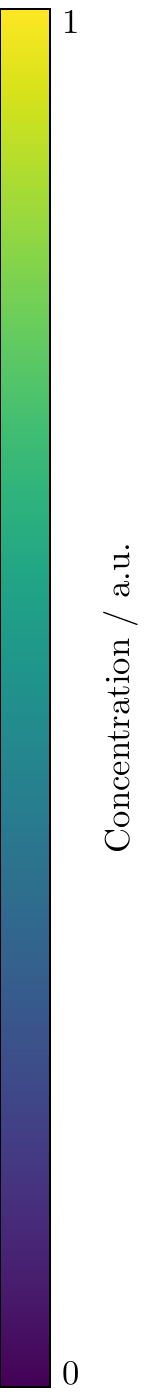}
    \caption{Reconstructed image sequences for the \qty{3}{\hertz} rotation dataset. Each column corresponds to one algorithm, each row represents one frame in time. We display averaged intensity projections onto the $y$-$z$-plane in this Figure. }
    \label{fig:reconstruction_3hz}
\end{figure}
\begin{figure}[tbhp]
    \centering
    \begin{minipage}{0.45\textwidth}
    \centering
    \includegraphics[width=0.9\textwidth]{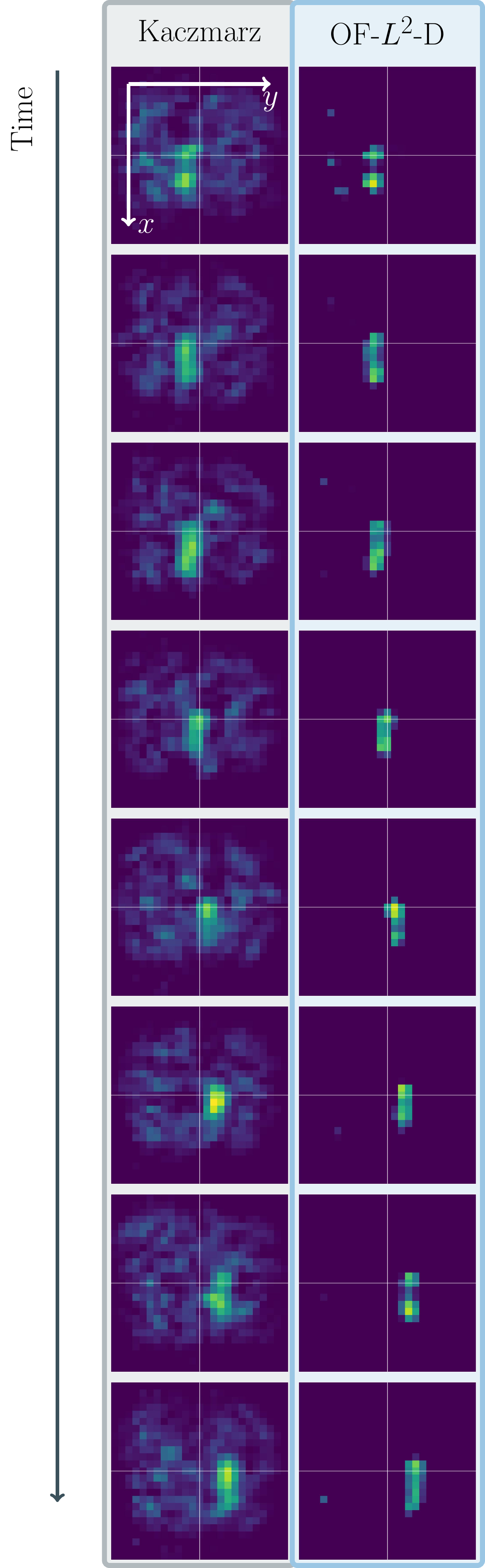} 
    \end{minipage}
      \begin{minipage}{0.45\textwidth}
      \centering
    \includegraphics[width=0.9\textwidth]{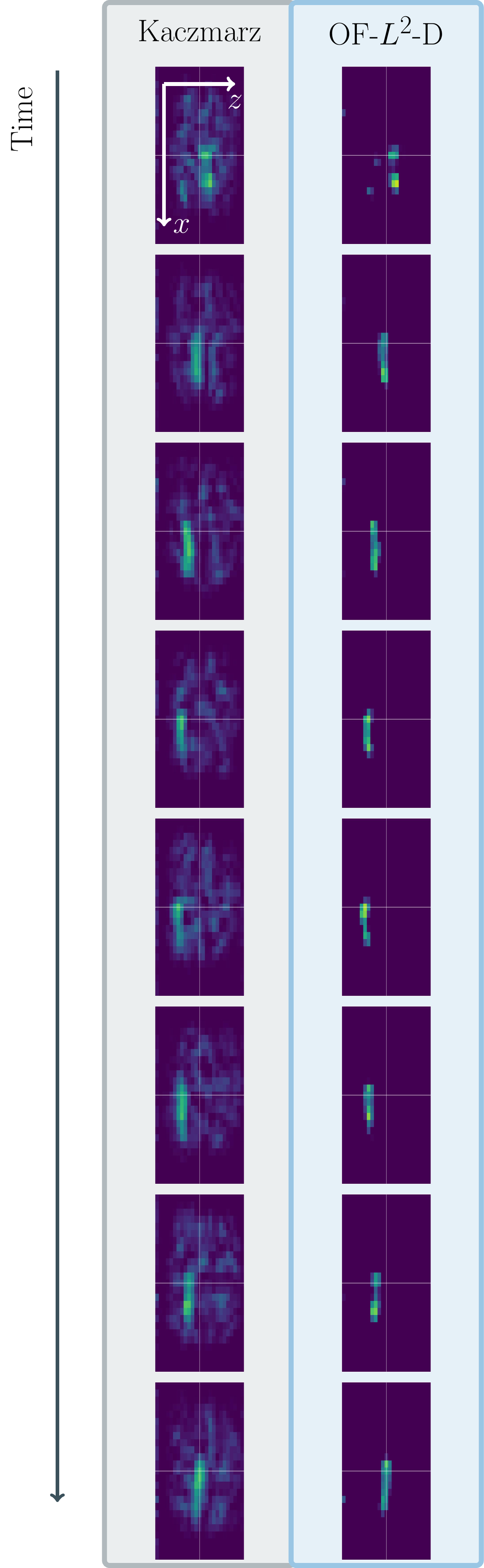} 
    \end{minipage}
    \begin{minipage}{0.08\textwidth}
    \vspace{1.3cm}\includegraphics[width=1.\textwidth]{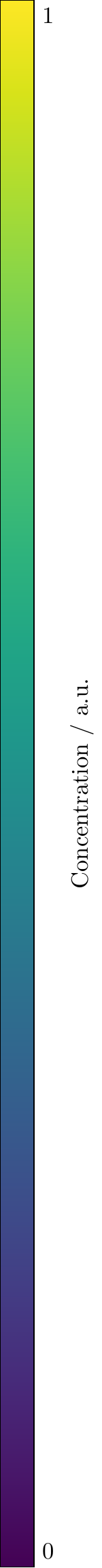}\\
    \end{minipage}
    \caption{Reconstructed image sequence of the \qty{3}{\hertz} rotation phantom for Kaczmarz method (applied frame-by-frame) and the joint OF-$L^2$-D algorithm. This figure displays averaged intensity projections onto the $y$-$x$- and $z$-$x$-planes for the first eight timeframes (each row corresponds to one step in time).}
    \label{fig:bohrerdaten_3hz_yx_zx}
\end{figure}

\Cref{fig:reconstruction_3hz} shows that our method yields significantly less noisy reconstructed images than the Kaczmarz method. Moreover, the dot shaped phantom is less blurry, although still at the correct position for each frame. The qualitative differences between $L^1$-D and $L^2$-D term are negligible which might be due to the chosen weighting in the data pre-processing. It can be observed that the reconstructed rod by an $L^1$-D term are slightly better resolved and the shape is closer to the actual phantom size, which covers less than one voxel in the $yz$-plane. Frame-wise SPDHG algorithm yields less noise artifacts than the Kaczmarz method, but has severe problems in achieving comparable mass in each frame. All reconstruction parameters are given in \cref{tab:bohrerdaten_3hz_parameter}. They were chosen based on visual inspection out of a wide range of tested parameters, for details on the parameter tests see \cref{app:par_search}. The joint algorithm proposed here is used in combination with early stopping in order to limit the computation time, however, the reconstructed image sequences are close to convergence at that point. More precisely, the early stopping iteration is given for the image reconstruction subtask in the inner loop of the alternating algorithm. The alternations are performed 3 times. \Cref{fig:bohrerdaten_3hz_yx_zx} displays reconstructed image sequences in the other two planes. The proposed method yields reconstructed images, in which the phantom stays constant in size and has a reasonable shape (length of 7 voxels times width of 2 voxels). For the Kaczmarz method reconstructions the phantom seems to vary in size. In time frames 4 to 6, it appears significantly shorter than in other frames. Moreover, it is often difficult to distinguish between phantom and background noise. 

\begin{table}[tbhp]
\footnotesize
    \caption{The parameters used for the reconstruction of the \qty{3}{\hertz} rotation phantom data. The early stopping index indicates at which iteration the algorithm was terminated. The parameter $\lambda$ denotes the weight on the Tikhonov $L^2$-penalty term used by Kaczmarz method. As defined in \cref{eq:regularization_imagereco}, $\alpha_1$ denotes the weight on the $L^1$-penalty on the concentration, $\alpha_2$ defines the weight on the TV-penalty term on the concentration. $\beta$ denotes the weight on the TV-penalty on the flow and $\gamma$ weights the motion model penalty term (see \cref{eq:motion_est}).}
    \label{tab:bohrerdaten_3hz_parameter}
    \centering
    \begin{tabular}{|l|l|c|c|c|c|c|c|}
    \hline
       Algorithm  & Motion model & Early stopping & $\lambda$ & $\alpha_1$ & $\alpha_2$ & $\beta$ & $\gamma$ \\
       \hline
       Kaczmarz  & None & $k=10\hphantom{^2}$ & 5.62 &  &  &  &  \\
       SPDHG & None & $k=10^3$ & & $5.0\cdot10^{-1}$& $3.0\cdot 10^{-1}$& & \\
       Joint, $L^1$-D & OF & $k=10^2$ & & $6.0\cdot10^{-1}$ & $1.0\cdot 10^{-1}$ & $10^{-1}$ & 100  \\
       Joint, $L^2$-D & OF & $k=10^2$ & & $2.5\cdot10^{-1}$ & $1.0\cdot 10^{+2}$ & $10^{-1}$& 100  \\
       Joint, $L^1$-D & MC & $k=10^2$ & & $7.0\cdot10^{-1}$ & $5.0\cdot 10^{-1}$ & $10^{-1}$& 100  \\
       Joint, $L^2$-D & MC & $k=10^2$ & & $2.5\cdot10^{-1}$ & $2.5\cdot 10^{-1}$ & $10^{-1}$& 100  \\
       \hline
    \end{tabular}
\end{table}

\begin{table}[tbhp]
\footnotesize
    \caption{The coefficient of variation of the total reconstructed mass in each frame over the first 30 frames of the sequence for the \qty{3}{\hertz} dataset. }
    \label{tab:stddev_bohrerdaten}
    \centering
    \begin{tabular}{|l|c|c|c|c|c|c|}
        \hline
        Algorithm & Kaczmarz & SPDHG& OF-$L^1$-D & OF-$L^2$-D & MC-$L^1$-D & MC-$L^2$-D \\
        \hline
Coefficient of variation & 0.149 &0.556& 0.116 & 0.141 & 0.110 & 0.129 \\
        \hline
    \end{tabular}
\end{table}
In order to underline the improvements by the optical flow or mass preservation constraint, we list the coefficient of variation, i.e. standard deviation divided by mean value, of the reconstructed mass in each frame over the first 30 frames for each reconstruction approach in \cref{tab:stddev_bohrerdaten}. All joint reconstruction approaches outperform the Kaczmarz method based on this measure. The value of frame-wise SPDHG is in completely different ranges compared to the other algorithms, which confirms the visual impression. 

In order to investigate the impact of our approach to the motion estimation, we compare the motion fields obtained during the joint reconstruction task to motion fields obtained from the sequences reconstructed by Kaczmarz method. In \cref{fig:bohrerdaten_motion_3hz} we compare motion fields for the \qty{3}{\hertz} dataset by comparing the estimated particle trajectories obtained from the different motion estimates. The optimal particle trajectories (left column, depicted in white) are obtained step by step from the reconstructed images and are almost identical for both methods. For guidance, we fitted a circular path with the correct diameter of approximately \qty{11.5}{\milli\meter} as approximate ground truth to the plot (depicted in green). However, the motion field obtained from the Kaczmarz reconstructions highly underestimates the magnitude of motion, probably due to the severe noise artifacts in the image sequence. The piecewise linear as well as the smoothed approximation of the trajectory computed with the help of the motion field is too narrow in case of the Kaczmarz reconstructed image sequence. The motion estimates obtained by the joint MC-$L^2$-D algorithm however is close to the optimal trajectory. 
\begin{figure}[tbhp]
    \centering
    \subfloat[Particle trajectories for frames 3 to 15 of the \qty{3}{\hertz} dataset.]{\label{fig:bohrerdaten_motion_3hz}
    \includegraphics[width=0.45\textwidth]{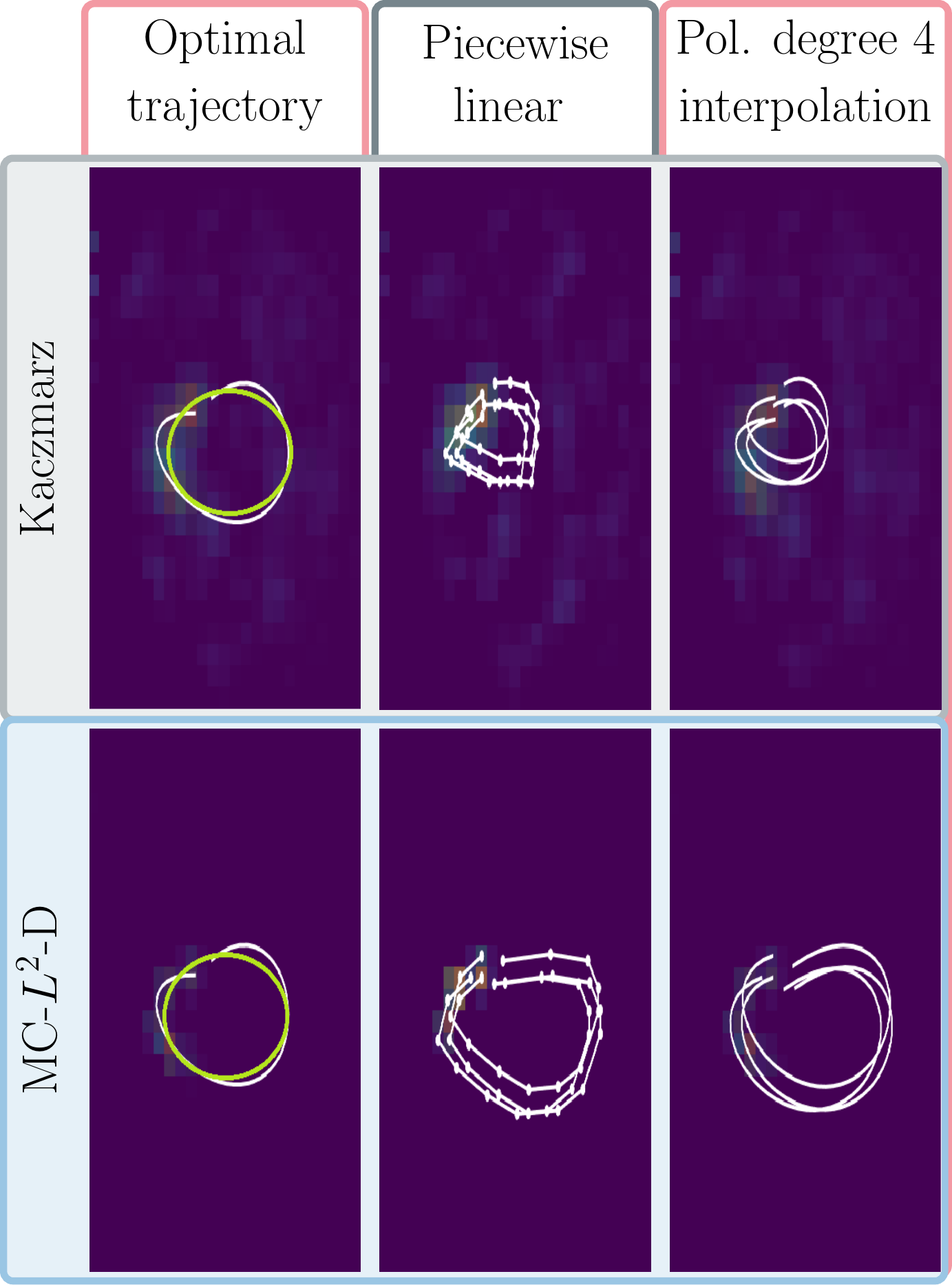}}
    \subfloat[Particle trajectories for frames 3 to 10 of the \qty{7}{\hertz} dataset.]{ \label{fig:bohrerdaten_motion_7hz}   \includegraphics[width=0.45\textwidth]{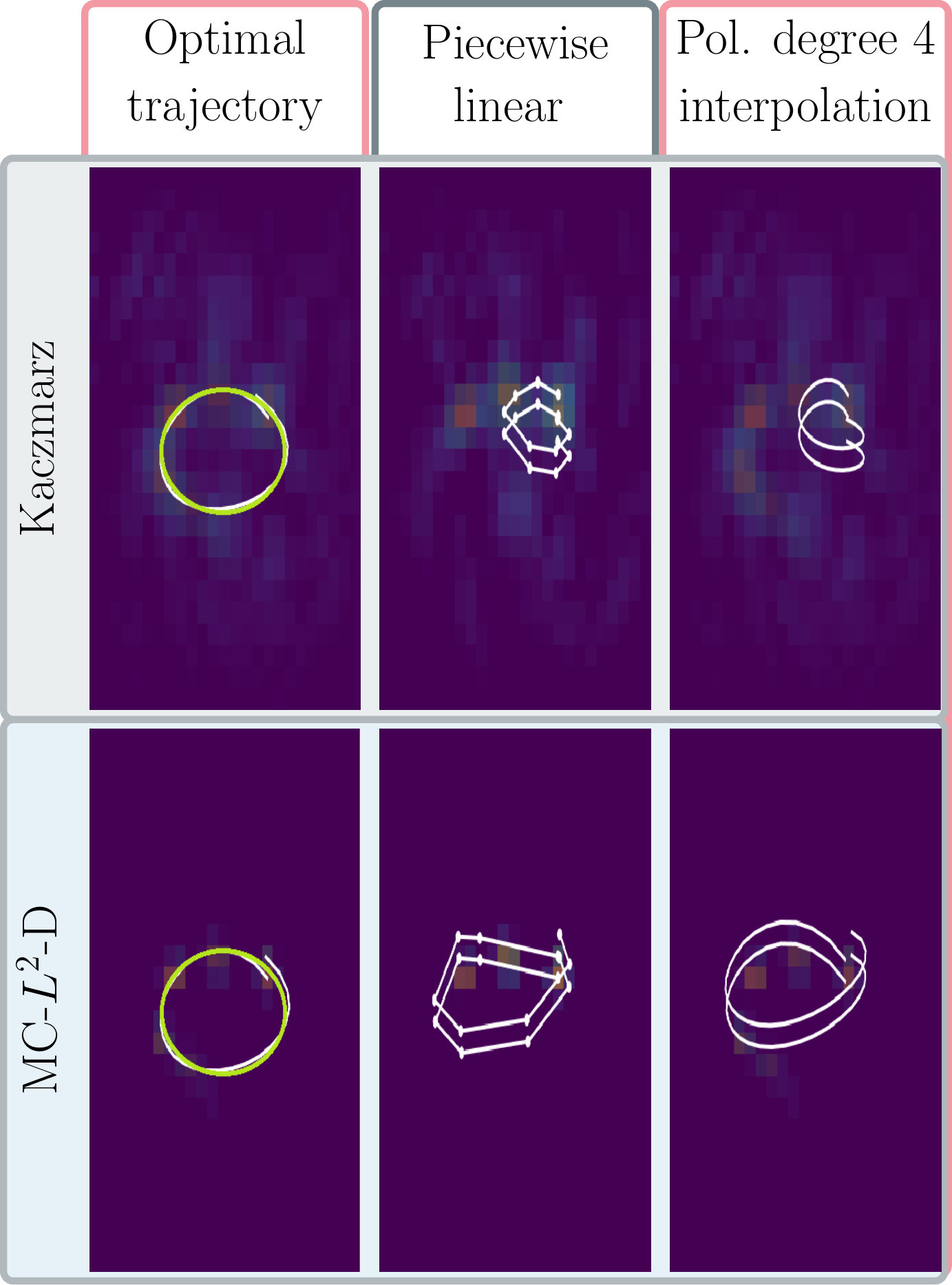}}
    \caption{The left column displays the particles trajectory as observed in the reconstruction and smoothed by a 4th order polynomial approximation (in white) as well as a circle with correct diameter fitted to the curve (in green). The middle column shows the piecewise linear trajectories, which are obtained by adding up the computed motion at each time step. 
    The right column shows a 4th order polynomial approximation to the curve in the middle.}
\end{figure}

\Cref{fig:bohrerdaten_motion_7hz} shows similar results for the \qty{7}{\hertz} dataset. The more challenging image reconstruction task leads to severe problems in the motion estimation task based on the images reconstructed by Kaczmarz method. Our method again yields trajectory estimates that are close to the best possible guess based on the image sequences. 

Reconstruction of the rotation phantom datasets for different speeds shows that the image reconstruction task benefits from the joint approach. The more appropriate priors (in comparison to the Tikhonov regularization used by Kaczmarz method) minimize noise artifacts in the reconstruction. Moreover, the reconstructed mass over the time frames is more constant due to the additional mass conservation prior. Moreover, the motion estimates are also improved by the joint approach. The enhanced image sequences lead to more accurate motion estimates. Video sequences of reconstructed images for the \qty{3}{\hertz} and \qty{7}{\hertz} datasets reconstructed by frame-by-frame Kaczmarz method as well as a joint approach are available as additional material. 

\subsection{Experiments on a 3D in-vivo blood flow dataset}
\label{sec:NumMousedata}
The in-vivo mouse data were first published in \cite{graeser2017towards} in the context of presenting a highly sensitive gradiometric receive coil. 
The system matrix has been calibrated with a small capillary of size \qty{0.7}{\milli\meter} in order to capture the fine structures of the cardiovascular system. This yields a system matrix of size $46\times36\times19$ voxels capturing a volume of size \numproduct{32.2x25.2x13.3} \unit{\cubic\milli\meter}. The measurement sequence 
is collected by the above-mentioned gradiometric receive coil for an injected bolus of volume \qty{10}{\micro\liter}. 
No time averaging was applied resulting in a repetition time of \qty{21.54}{\milli\second} like for the previous phantom experiment.
We use the same data-prepocessing as for the rotation phantom data with an additional correction for background signal, where dedicated empty measurements are subtracted from measurements and system matrix \cite{them2015sensitivity}. 
Anatomic background information on the inner tissue structure is available from MRI scans and are combined with the MPI reconstructions as grayscale background images in order to interpret the resulting images. Exemplary images of the particle inflow into the cardiovascular system of the mouse heart via the vena cava inferior are displayed in \cref{fig:reconstruction_mouse_mri}. 

\begin{figure}[tbhp]
    \begin{minipage}{0.95\textwidth}
    \centering
    \noindent\includegraphics[width=\textwidth]{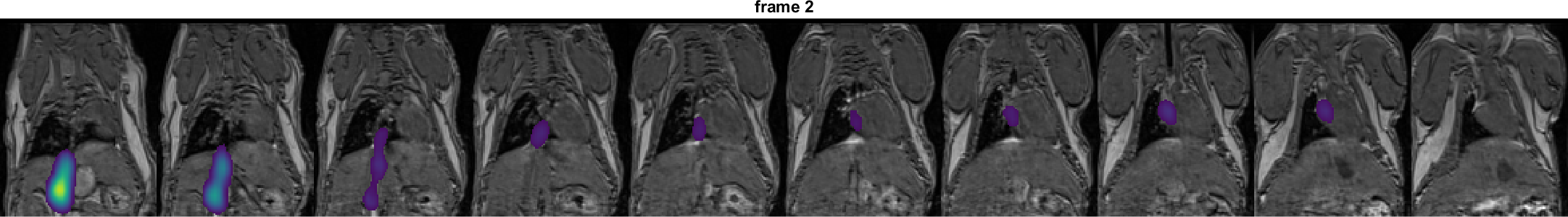}
    \includegraphics[width=\textwidth]{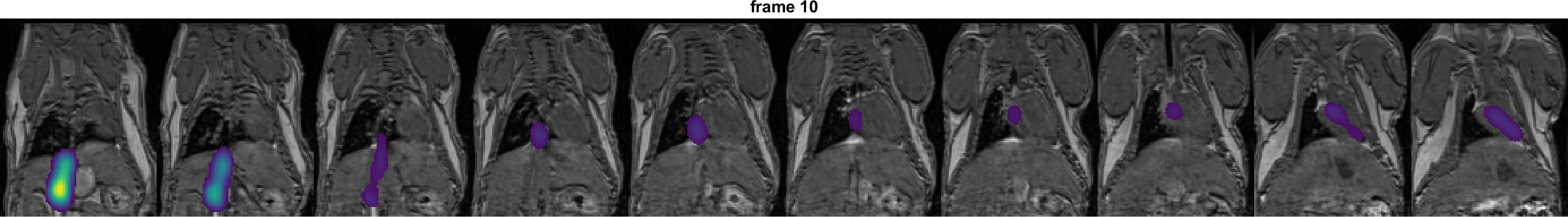}
    \end{minipage}
    \begin{minipage}{0.04\textwidth}
    \vspace{0.135cm}
        \includegraphics[width=0.94\textwidth]{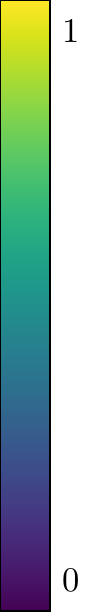}
    \end{minipage}
    \caption{Coronal slices of reconstructed concentration obtained from the in-vivo mouse dataset using the MC-$L^2$-D method with an $L^2$ penalty term for $v$. Each column corresponds to one slice of tissue, MRI data are integrated for background information. Two different time steps are depicted in the different rows. Concentrations larger than 0.4\,\% of the maximum concentration are illustrated.}
\label{fig:reconstruction_mouse_mri}
\end{figure}

A full video sequence showing the reconstructed images for the MC-$L^1$-D algorithm with the parameters $\alpha_1= 0.05$, $\alpha_2 =100$, $\beta = 0.1$ and $\gamma = 100$ is appended to this manuscript. In the absence of a ground truth reference we now focus on the reconstructed motion fields. This is particularly useful, as MPI has potential applications in cardiovascular imaging \cite{kaul2018magnetic}.
As the blood flow through the cardiovascular system is not as homogeneous in space as for the previous phantoms, we do not use regularization on the gradient of the flow. Instead we use standard $L^2$ Tikhonov regularization on the flow fields. The motion estimation subproblem thus reads
\begin{equation*}
    \min_v \int_0^T \beta \left\|v\left(\cdot, t\right) \right\|_{L^2(\Omega)^n}^2 + \gamma \left\| \dd{t}c\left(\cdot, t \right) + \nabla \cdot \left( c\left(\cdot, t \right)\cdot v\left(\cdot, t \right)\right) \right\|_{L^1\left( \Omega\right) }  \mathrm{d}t.
\end{equation*}
The dual update step in \cref{alg:PDHG} simplifies as the Tikhonov term belongs to the primal problem. The primal update step consists of a different but easy to compute proximal operator. Parts of the reconstructed image sequence as well as the estimated motion field in 2D slices for a later time point after the particle inflow are depicted in \cref{fig:reconstruction_mouse}. The flow is illustrated by two different ways. First, we have a quiver plot where the direction and strength of the flow at a certain position is indicated by an arrow and its direction and length. The strength of the flow is clearly visible in this plot. However, direction is difficult to observe, which is why we add a color wheel plot, where each direction is linked to a certain color. Additionally, the flow in the third spatial direction, i.e. orthogonal to the slices depicted in \cref{fig:reconstruction_mouse} is illustrated in \cref{fig:reconstruction_mouse_z}.
In general, we observe flow estimates, which are plausible with respect to the underlying structure provided by the MRI image and illustrate the blood flow through the cardiovascular system. Such flow information is already used in other imaging modalities like 4D-flow magnetic resonance imaging \cite{strater20184d} and thus would be very interesting for future clinical real-world applications of MPI.


\begin{figure}[tbhp]
    \centering
    \noindent\includegraphics[height=\textheight]{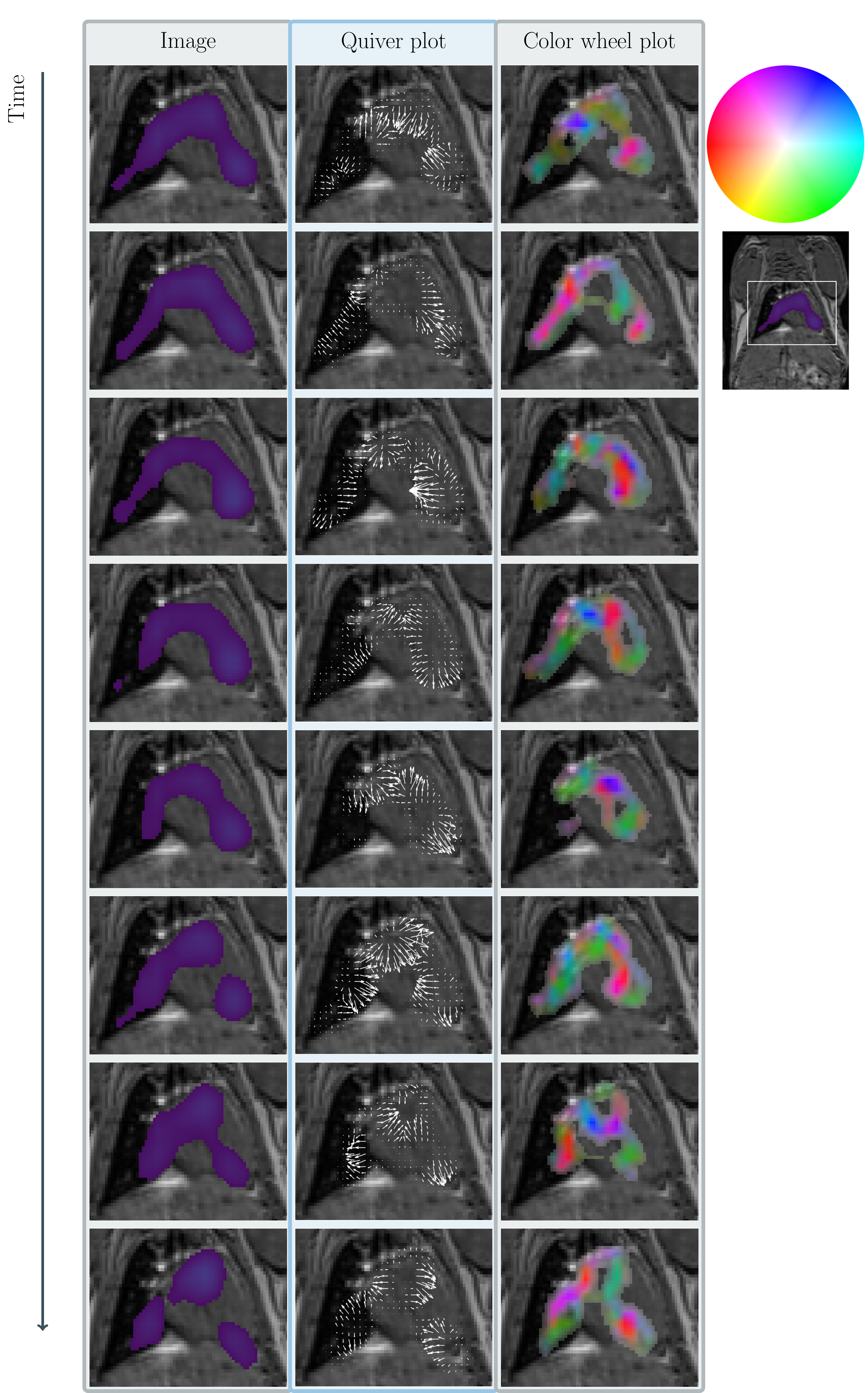}
    \caption{Reconstructed image sequences for the in-vivo flow dataset using the MC-$L^2$-D method with an $L^2$- penalty term for $v$. For illustrative purposes we depict a slice corresponding to height $[9.1,9.8]$\unit{\milli\meter} (index 14) of the 3D concentration and the 3D+time flow over eight subsequent frames. The left column depicts the reconstructed concentration in the snippet indicated in the most right column (same colormap as before). Concentrations larger than 0.4\,\% of the maximum concentration are illustrated. The second left plot indicates the flow fields by a quiver plot, whereas the second right column depicts the flow field by a color wheel plot using the color wheel in the right in order to illustrate the direction of the flow.  }
    \label{fig:reconstruction_mouse}
\end{figure}

\begin{figure}[tbhp]
    \centering
    \noindent\includegraphics[width=\textwidth]{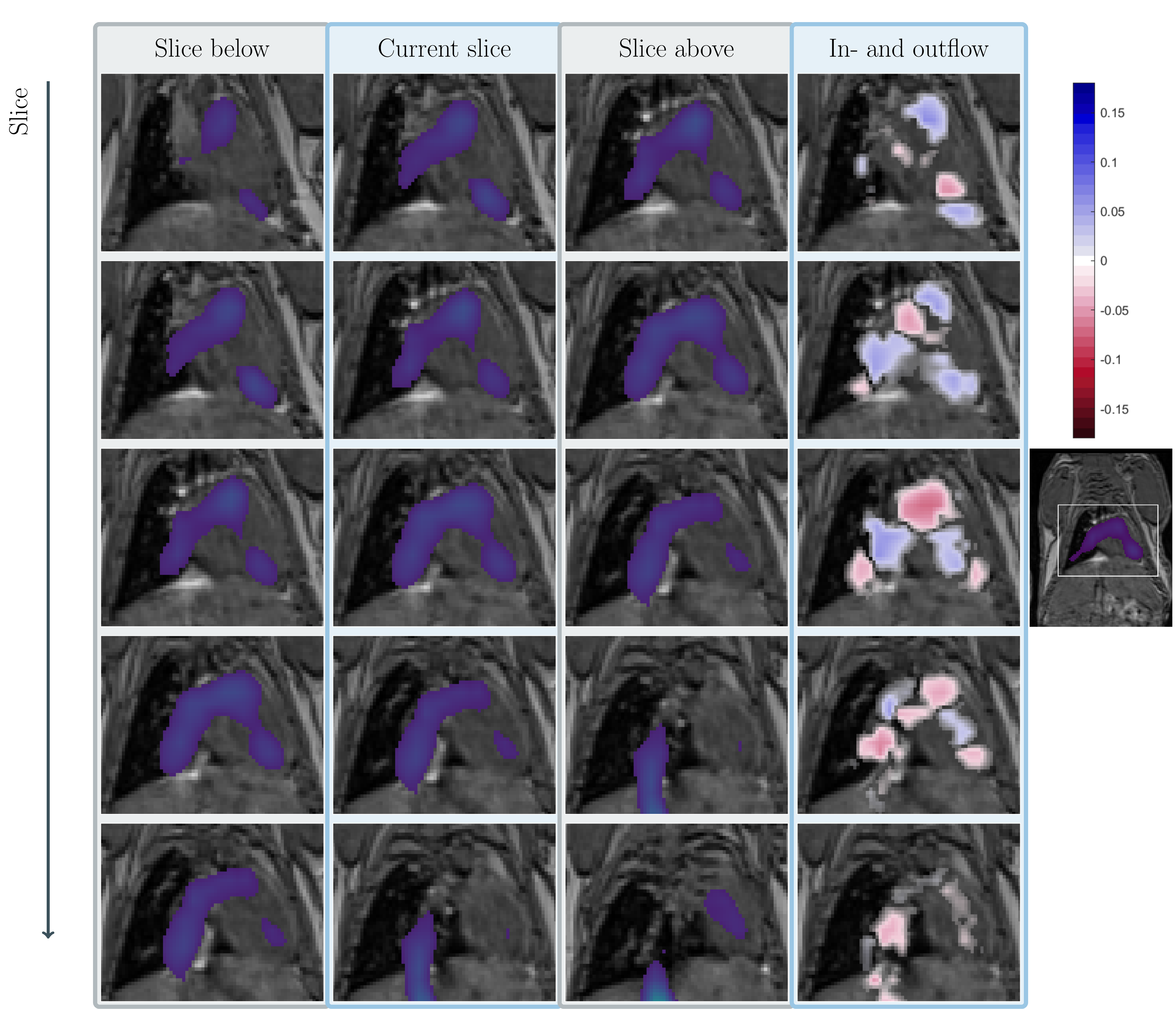}
    \caption{Reconstructed image at one time step for the in-vivo flow dataset using the MC-$L^2$-D method with an $L^2$ penalty term for $v$. For illustrative purposes we depict three consecutive slices of the 3D concentration in the first three columns (same colorbar as before) and the 3D+time flow leaving the current slice in the most right column. Concentrations larger than 0.6\,\% of the maximum concentration are illustrated. The colorbar in the right identifies the direction of the in- and outflow into the current slice, where a positive flow means into the slice above and a negative flow means going into the direction of the slice below.}
    \label{fig:reconstruction_mouse_z}
\end{figure}

\section{Conclusion and outlook}
\label{sec:conclusion}
In the present work we consider a joint image reconstruction and motion estimation approach in the dynamic inverse problem setting of MPI where a time-dependent sequence of 3D concentration images is reconstructed. 
The joint image and flow parameter estimation is based on \cite{Burger.2018} taking into account an optical flow or a mass preservation constraint for 2D image processing tasks and in the present work it is particularly extended to the 3D+time case in MPI investigating the existence of a solution to the minimization problem in this setting. This particularly includes a theoretical discussion of the required properties in the MPI equilibrium model. In addition an efficient solution to the problem is developed using alternating minimization with primal-dual algorithms in both subproblems. The image reconstruction subproblem is solved by stochastic PDHG exploiting the structure of the forward operator in MPI. The algorithm solving the motion estimation subproblem combines PDHG algorithm with a multi-scale scheme and warping. 
The theoretical part is accompanied by an extensive numerical investigation in the context of MPI based on simulated data showing the advantages of the proposed joint approach compared to standard MPI reconstruction algorithms. In addition the proposed method is evaluated on measured 3D+time data including rotation phantoms as well as in-vivo data from a mouse. Here, not only the improved reconstruction quality in the joint approach is discussed but also the benefits of having access to the jointly estimated motion parameters are illustrated.  

In summary, the present manuscript provides a comprehensive investigation of the dynamic reconstruction problem in MPI particularly taking into account motion priors and it thus paves the way for various directions of future work.
On the theoretical as well as numerical side, future works include, for example, an extension allowing for the consideration of an in- and outflow into the field-of-view being included in  the mass preservation or optical flow constraint or inclusion of prior knowledge about e.g. edges obtained from MRI images. In the field of application, MPI, future works include further qualitative as well as quantitative evaluations with respect to image reconstruction quality as well accuracy in flow estimation. This requires a careful experimental design to obtain access to ground truth information not only of the phantom but also of the flow field. In the present work, the flow is solely included in the motion model with respect to the concentration but recent studies \cite{MoeddelSchloemerkemperKluth2023IWMPI} indicate that there also might be an explicit influence on the obtained measurement. In this case the obtained measurement have an immediate dependence on the flow field and it thus needs to be taken into account explicitly as an argument of the forward operator $A$, i.e., consider $A(c,v,t)$ in the general problem setting. Suitable models for the operator $A$ in the context of MPI and the general theoretical investigation of this approach remain future work.




\bibliographystyle{abbrv}
\bibliography{literature,mypub}


\appendix
\section{Proof of \cref{thm:existence}}
\label{app:proofExistence}
\begin{proof}
	We follow the proof of \cite{Burger.2018, Dirks.2015},  but we have three main differences in our assumptions. First, our image sequence is in 3+t dimensions; second, our regularizers are of a more general form and third, our linear operator is time-dependent. We follow the technique of the direct method in the calculus of variations. The line of argumentation is as follows: 
	\begin{enumerate}
        \item The functional $J: L^{\hat{p}}\left( 0,T; BV\left( \Omega\right) \right) \times L^q\left( 0,T; BV\left( \Omega\right)^n \right) \rightarrow \R$ is bounded from below such that $\inf J\left( c,v\right) > -\infty$. 
        	This implies existence of a minimizing sequence of admissible $(c_m, v_m)\in L^{\hat{p}}\left( 0,T; BV\left( \Omega\right) \right) \times L^q\left( 0,T; BV\left( \Omega\right)^n \right), m\in  \N,$ (here, admissible means $c_m\in L^p(0,T;BV(\Omega))$, $\left\| v_m\right\|_{L^\infty(0,T;L^\infty(\Omega)^n)} \leq k_\infty$, $\left\Vert \nabla\cdot v_m\right\Vert _{\theta}\leq k_{\theta}$ , $ m_i\left(c_m,v_m\right)=0 \text{ in } \mathcal{D}'(\Omega \times [0,T])$), such that $\lim_{m\rightarrow \infty} J \left(c_m,v_m\right) = \inf J\left( c,v\right)$. 
		\item The sublevel sets of $J$ are weak-$^*$ compact. As $J$ is continuous, the sublevel sets are closed. It remains to show for this step in the argumentation that the sublevel sets are bounded in norm. This then implies existence of a subsequence $\left(c_{m_k},v_{m_k}\right)\overset{* \phantom{a}}{\rightharpoonup} \left(c^*,v^*\right)$, which then identifies the candidate $\left(c^*,v^*\right)$ for the minimizer. 

		\item The functional $J$ is lower semi-continuous with respect to the weak-$^*$ topology (as a sum of such functionals), which yields 
        \begin{equation*}
             \inf J\left( c,v\right) \leq J\left(c^*,v^*\right) \leq \liminf_{m \rightarrow \infty} J\left(c_{m_k},v_{m_k}\right) = \inf J\left( c,v\right).
        \end{equation*}
		\item In order to obtain the weak-$^*$ closedness of the admissible set (in the weak-$^*$ topology), we prove convergence of the constraint in a distributional sense, i.e., we verify that $\left(c^*,v^*\right)$ is admissible. 
		We consider both optical flow and mass conservation constraint and complete the proofs in 3D.		
	\end{enumerate}

\underline{\textit{Weak-$^*$ compactness of sublevel sets:}}
Denote the sublevel set to $\nu \in \mathbb{R}$ by 
	\begin{equation*}
		S_{\nu}:= \left\lbrace \left( c,v\right) \in L^{\hat{p}}\left( 0,T; BV\left( \Omega\right) \right) \times L^q\left( 0,T; BV\left( \Omega\right)^n \right) : \quad J\left( c,v\right) \leq \nu, \|v\|_{L^\infty\left(0,T;L^\infty\left(\Omega\right)^n\right)} \leq k_\infty \right\rbrace. 
	\end{equation*}
We will use that $S_{\nu}\subset  L^{\hat{p}}\left( 0,T; BV\left( \Omega\right) \right) \times L^q\left( 0,T; BV\left( \Omega\right)^n \right) \cong  X^*$ for $X^*$ being the dual space of a Banach space $X$. Therefore $S_{\nu}$ is weak-$^*$ compact if and only if $S_{\nu}$ is closed in weak-$^*$ topology (this holds as J is continuous) and the sublevel set is bounded in norm. \\
It remains to show that $\left\| c\right\| _{L^{\hat{p}}\left( 0,T; BV\left( \Omega\right) \right) } = \left(\displaystyle \int_0^T \left( \left\| c\right\|_{L^1(\Omega)} + \left| c\right|_{\mathrm{BV}(\Omega)}\right)^{\hat{p}}  \mathrm{d}t\right) ^{1/\hat{p}}  \leq k$ and \\ $\left\| v\right\| _{L^q\left( 0,T; BV\left( \Omega\right)^n \right) } \leq k$ for a $k < \infty$. \\
Let $\left( c,v\right) \in S_{\nu}$ (i.e, it holds additionally  $\left\| v\right\|_{L^\infty(0,T;L^\infty(\Omega)^n)} \leq k_\infty$). \\	
We start with the bound on the norm of $c$:
	 From $\left( c,v\right) \in S_{\nu}$ it follows that 
	\begin{eqnarray*}
		\int_0^T \frac{1}{2}\left\| \left(A_tc\right)\left( t\right) - u\left(t\right)\right\|^2_{Y } \mathrm{d}t = \frac{1}{2}\left\| A_t c- u\right\| ^2_{Y }  \leq \nu  \\
		\Rightarrow \left( A_tc -u\right) \in Y \quad \mathrm{a.\enspace e. \enspace in } \enspace \left[ 0,T\right]. 
	\end{eqnarray*}
	Define $k_A\left( t\right) := \left\| \left(A_tc\right)\left(  t \right)-u\left( t \right)\right\| _{Y}$, then 
	\begin{equation*}
		\int_0^T k_A^{\hat{p}} \mathrm{d}t = 	\int_0^T \left\| A_tc-u\right\| _{Y}^{\hat{p}} \mathrm{d}t
	\end{equation*}
	is bounded for $1 \leq \hat{p} \leq 2$ due to the continuous embedding $L^2\left( 0,T; Y\right) \hookrightarrow L^{\hat{p}}\left( 0,T; Y\right)$.\\
	To bound $\left\| c\right\|_{L^{\hat{p}}\left( 0,T;L^1(\Omega)\right) }$, we start by an upper bound for arbitrary but fixed $t \in \left[ 0,T\right]$. 
	Defining $\bar{c} = \frac{1}{\left| \Omega\right| }\int_{\Omega} c\left( x,t\right) \mathrm{d}x$, $c_0 = c\left( \cdot,t\right) -\bar{c}$ we note that
	\begin{equation*}
		\int_{\Omega}c_0 \mathrm{d}x = 0.
	\end{equation*} 
	Moreover, 
	\begin{equation*}
		\left| c_0 \right|_{\mathrm{BV}\left( \Omega\right) }  =  \left| c\left( \cdot,t\right) \right|_{\mathrm{BV}\left( \Omega\right) }  \leq k \int_0^T \left| c\left( \cdot,\tau \right) \right|^p_{\mathrm{BV}\left( \Omega\right) } \mathrm{d}\tau \leq \frac{k\nu}{\alpha},
	\end{equation*}
for a constant $k \in \mathbb{R}$. 
	With Poincaré-Wirtinger, it follows that there exists $k_1>0$ such that 
	\begin{equation}
	\label{eq:01}
		\left\| c_0\right\|_{L^{l}\left( \Omega\right) } \leq k_1\left| c \right|_{\mathrm{BV}\left( \Omega\right) }  \leq k_1 \frac{k\nu}{\alpha} =: \tilde{k}_1 
	\end{equation} 
	for $l \leq \frac{n}{n-1}$. 
	We derive the following estimate:  
	\begin{flalign*}
	\left\| A_t\bar{c}\right\|^2 _{Y} &- 2\left\| A_t\bar{c}\right\| _{Y } \left( \left\| A_t \right\| \left\| c_0\right\|_{L^l\left(\Omega\right)} + \left\| u\right\| _{Y}\right) \\
	& \leq \left\| A_t\bar{c}\right\|^2 _{Y } - 2\left\| A_t c_0 -u\right\|_{Y}\left\| A_t \bar{c}\right\|_{Y}\\
	& \leq \left\| A_tc_0 -u\right\| ^2_{Y} + \left\| A_t\bar{c}\right\|^2 _{Y } - 2\left\| A_t c_0 -u\right\|_{L^2(Y)}\left\| A_t \bar{c}\right\|_{Y}\\
	& = \left( \left\| A_t c_0-u\right\|_{Y} -\left\| A_t\bar{c}\right\| _{Y}\right) ^2 \\
	&\leq \left\| A_t c_0-u + A_t \bar{c}\right\| ^2_{Y}\\
	& = \left\| A_t\left( c_0 + \bar{c}\right) -u\right\| ^2_{Y} = k_A^2(t) 	   
	\end{flalign*}
	Using the estimate of \cref{eq:01} and $u \in L^2\left( 0,T; Y \right)$ yields  
	\begin{equation*}
		0 \leq \left\| A_t  \right\| \left\| c_0\right\| _{L^l(\Omega)} + \left\| u\right\| _{Y} \leq \left\| A_t \right\|\tilde{k}_1  + \left\| u\right\| _{Y} :=k_2\left( t\right).
	\end{equation*}
    Note that, as $\left\|A_t \right\|$ is bounded for all $t$, we also have a bound on the time integral over $k_2$. 
	We can now use the above estimate to obtain the following bound on $\left\| A_t \bar{c}\right\|_{Y}$: 
	\begin{flalign*}
		\left\| A_t \bar{c}\right\| ^2_{Y} &- 2\left\| A_t \bar{c}\right\| _{Y}\left( \left\| A_t\right\| \left\| c_0\right\| _{L^l(\Omega)} + \left\| u\right\| _{Y}\right)  \\ 
        &+ \left( \left\| A_t\right\| \left\| c_0\right\| _{L^l(\Omega)} + \left\| u\right\|_{Y}\right) ^2
		\leq k_A^2\left( t\right) + k_2^2\left( t \right)\\
		&\Leftrightarrow \left( \left\| A_t \bar{c}\right\| _{Y} - \left\| A_t\right\| \left\| c_0\right\| _{L^l(\Omega)} + \left\| u\right\| _{Y}\right) ^2 \leq k_A^2\left( t\right) + k_2^2\left( t \right)\\
		&\Rightarrow \left\| A_t \bar{c}\right\| _{Y} \leq \left( k_A^2\left( t\right) + k_2^2\left( t \right)\right) ^{1/2} + \left\| A_t\right\| \left\| c_0\right\| _{L^l(\Omega)} + \left\| u\right\| _{Y} \\
		&\leq \left( k_A^2\left( t\right) + k_2^2\left( t \right) \right) ^{1/2} + k_2\left( t \right) := k_3\left( t\right) 
	\end{flalign*}
	Finally, we deduce a bound on $\left\| c\left( \cdot,t\right) \right\| _{L^1(\Omega)}$ by
	\begin{flalign*}
		0 &\leq \left\| c\left( \cdot,t\right) \right\| _{L^1(\Omega)} \leq k_4\left\| c\left( \cdot,t\right) \right\| _{L^l(\Omega)} = k_4 \left\|c_0 + \bar{c}\right\| _{L^l(\Omega)} \\
		& \leq k_4\left( \left\| c_0\right\| _{L^l(\Omega)} + \left| \bar{c}\right|\left\| 1_\Omega\right\|_{L^l(\Omega)}  \right) 
		\leq k_4 \left( \tilde{k}_1 + \frac{k_3\left( t\right) \left\|1_\Omega \right\|_{L^l\left( \Omega \right)}}{\left\| A_t1_\Omega\right\|_Y }\right)   , 
	\end{flalign*}
    where we used that by assuming $\left\| A_t \mathds{1}_{\Omega} \right\|_Y \neq 0$ it holds that
    \begin{equation*}
    \left| \bar{c} \right| \left\| A_t \mathds{1}_\Omega \right\|_Y = \left\| A_t \bar{c} \right\|_Y \leq k_3\left( t\right)      
    \end{equation*}
    and thus 
    \begin{equation*}
        \left| \bar{c} \right| \leq \frac{k_3\left(t\right)}{\left\| A_t \mathds{1}_\Omega \right\|_Y}. 
    \end{equation*}
	We now have
	\begin{equation*}
		\int_0^T \left\| c\left( \cdot, t\right) \right\|^{\hat{p}} _{L^1(\Omega)} \mathrm{d}t \leq \int_0^T \left( k_4\left( \tilde{k}_1 + \frac{k_3\left( t\right)\left\|1_\Omega\right\|_{L^l \left(\Omega\right)}} {\left\| A_t1_\Omega\right\|_Y}\right)\right) ^{\hat{p}} \mathrm{d}t \leq : k_M < \infty \quad \mathrm{for } \quad 1 \leq \hat{p} \leq 2,
	\end{equation*}
    where $k_M$ exists as the expression is bounded for all $t$ and thus the supremum is also bounded.
	Putting it all together, we arrive with $j := \frac{p}{\hat{p}}$ at 
	\begin{flalign*}
		\left\| c\right\|^{\hat{p}}_{L^{\hat{p}}(0,T; \mathrm{BV}(\Omega))} &=  \int_0^T \left\| c\right\| ^{\hat{p}}_{\mathrm{BV}(\Omega)} \mathrm{d}t \\
		& \leq 2^{\hat{p}-1}\left( \int_0^T\left| c\right|_{\mathrm{BV}(\Omega)}^{\hat{p}} \mathrm{d}t + \int_0^T \left\| c\right\| ^{\hat{p}}_{L^1(\Omega)} \mathrm{d}t \right) \\
		& \leq 2^{\hat{p}-1} \int_0^T\left| c\right|_{\mathrm{BV}(\Omega)}^{\hat{p}} \mathrm{d}t +  2^{\hat{p}-1} k_M \\
		& \leq 2^{\hat{p}-1}\left( \int_0^T \left| \mathds{1}_{\left[ 0,T\right] }\right|^{j^*}\mathrm{d}t \right) ^{\frac{1}{j^*}}\left( \int_0^T \left| c\right|^{\hat{p}j}_{\mathrm{BV}(\Omega)} \mathrm{d}t \right) ^{\frac{1}{j}} + 2^{\hat{p}-1} k_M\quad \mathrm{with} \quad \frac{1}{j}+\frac{1}{j^*} = 1, \\
		& = 2^{\hat{p}-1} k_M + 2^{\hat{p}-1}\left\| \mathds{1}_{\left[ 0,T\right] }\right\|_{L^{j^*}(0,T)}\left( \int_0^T \left| c\right|^p_{\mathrm{BV}(\Omega)} \mathrm{d}t \right) ^{\frac{1}{j}}\\
		& \leq  2^{\hat{p}-1} \left( k_M + \left\| \mathds{1}_{\left[ 0,T\right] }\right\|_{L^{j^*}(0,T)}\left( \frac{\nu}{\alpha}\right) ^{\frac{1}{j}}\right) < \infty.
	\end{flalign*}
Consider now the norm of $v$: 
By assumption, we know that there exists $k_\infty$ such that $\left\| v\right\| _{L^\infty(\Omega)^n} \leq k_\infty$ almost everywhere on $\left[ 0,T\right] $. We proceed by the following estimate:
\begin{flalign*}
	\left\| v\right\| ^q_{L^q(0,T; \mathrm{BV}(\Omega)^n)} & = \int_0^T \left( \left\| v\right\| _{L^1(\Omega)^n} + \left| v\right| _{\mathrm{BV}(\Omega)^n}\right) ^q \mathrm{d}t \\
	& \leq 2^{q-1}\left( \int_0^T  \left\| v\right\| _{L^1(\Omega)^n}^q \mathrm{d}t+ \int_0^T \left| v\right| _{\mathrm{BV}(\Omega)^n} ^q \mathrm{d}t\right) \\
	& \leq 2^{q-1}\left( \int_0^T \left(k_\infty \left| \Omega \right| \right)^q \mathrm{d}t + \frac{\nu}{\beta}\right) \\
	& = 2^{q-1}\left( T\left(k_\infty \left| \Omega \right| \right)^q + \frac{\nu}{\beta}\right) <\infty.
\end{flalign*}
For the second inequality, we used that $v \in S_\nu$ and thus $\int_0^T \beta \left| v\right| _{\mathrm{BV}(\Omega)^n} ^q \mathrm{d}t \leq \nu$ by assumption.\\
We conclude that the sublevel sets $S_\nu$ are bounded in norm for fixed $\nu \in \mathbb{R}$, as $\left\| v\right\|_{L^q(0,T;\mathrm{BV}(\Omega)^n)}$ and $\left\| c\right\|_{L^{\hat{p}}(0,T; \mathrm{BV}(\Omega))}$ are bounded and the bounds depend on $\nu$ only. Moreover, we know that $\mathrm{BV}(\Omega)$ is isometrically isomorphic to the dual space of some Banach space $Z$ and thus 
\begin{flalign*}
	L^{\hat{p}}\left( 0,T; \mathrm{BV}(\Omega)\right) &\cong \left( L^{p^*}\left( 0,T; Z\right)\right)  ^* \quad \mathrm{ with } \quad \frac{1}{\hat{p}} + \frac{1}{p^*} = 1, \\
	L^{q}\left( 0,T; \mathrm{BV}(\Omega)^n\right) &\cong \left( L^{q^*}\left( 0,T; Z^n\right) \right) ^* \quad \mathrm{ with } \quad \frac{1}{q} + \frac{1}{q^*} = 1. 
\end{flalign*}  
Consequently, both spaces are dual spaces and thus $S_\nu \subset L^{\hat{p}}\left( 0,T; BV\left( \Omega\right) \right) \times L^q\left( 0,T; BV\left( \Omega\right)^n \right) \cong X^*$ for $X^*$ being the dual of a Banach space $X$ (here $X=L^{p^*}\left( 0,T; Z\right) \times L^{q^*}\left( 0,T; Z^n\right)$). Moreover, $S_\nu$ is closed in weak-$^*$ topology as $J$ is continuous and the sublevel sets are bounded in norm. It follows that $S_\nu$ is weak-$^*$ compact by Banach-Alaoglu. \\	
\underline{\textit{Lower semicontinuity of the functional $J$ with respect to the weak-$^*$ topology:}}
	As a sum of weak-$^*$ lower semicontinuous functionals, $J$ itself is weak-$^*$ lower semicontinuous, cf. \cite{Dirks.2015}. \\
\underline{\textit{Convergence of the constraint: Optical Flow}:}\\
Let $(c_m, v_m)\in L^{\hat{p}}\left( 0,T; BV\left( \Omega\right) \right) \times L^q\left( 0,T; BV\left( \Omega\right)^n \right), m\in  \N,$ be an admissible sequence (i.e., $c_m\in L^p(0,T,BV(\Omega))$, $\left\| v_m\right\|_{L^\infty(0,T;L^\infty(\Omega)^n)} \leq k_\infty$, $\left\Vert \nabla\cdot v_m\right\Vert _{\theta}\leq k_{\theta}$ , $ m_1\left(c,v\right)=0 \text{ in } \mathcal{D}'(\Omega \times [0,T])$), which also fulfills $(c_m,v_m) \in S_\nu$ for some $\nu \in \mathbb{R}$. Then $c_m$ and $v_m$ are bounded and it exist $c$ and $v$ such that by passing over to a subsequence (again denoted by $c_m$, $v_m$) we have 
	\begin{equation*}
		c_m \overset{* \phantom{a}}{\rightharpoonup} c, \quad v_m\overset{* \phantom{a}}{\rightharpoonup} v.
	\end{equation*}
	We want to show that 
	\begin{flalign*}
		\dd{t}c_m + \nabla c_m \cdot v_m &\longrightarrow \dd{t}c + \nabla c \cdot v \quad \text{ in } \mathcal{D}'(\Omega \times [0,T]),\\
	\end{flalign*}
	at least in a distributional sense. Therefore we need strong convergence of at least one factor of $\nabla c_m \cdot v_m$ in a certain sense. In order to use the Lemma of Aubin-Lions, we need boundedness of the time derivative of the sequence $c_m$ as well as some specific embeddings.
	We start with the bound for $\dd{t}c$: \\
	\begin{flalign*}
	\left|  \int_0^T \int_\Omega \dd{t}c \varphi \enspace\mathrm{d}x \mathrm{d}t\right| &= \left|  \int_0^T \int_\Omega c \nabla \cdot\left( v \varphi \right) \mathrm{d}x \mathrm{d}t\right|  \quad \forall \varphi \in \mathcal{C}^\infty_C\left( \Omega \times (0,T)\right) \\
    & \leq \int_0^T \left\|c \nabla \cdot\left( v \varphi \right) \right\|_{L^1\left( \Omega \right)} \dx{t}\\
	& \leq  \int_0^T \left\| c\right\|_{L^l\left(\Omega \right)} \left\| \nabla \cdot\left( v \varphi \right) \right\|_{L^{l^*}\left( \Omega \right)} \mathrm{d}t , \quad \left( \mathrm{Hoelder: }\enspace \frac{1}{l}+ \frac{1}{l^{^*}} = 1, \enspace l\leq \frac{n}{n-1}\right) \\
	& \leq \int_0^T \left\| c\right\| _{L^l(\Omega)}\left( \left\| \varphi \nabla \cdot v\right\| _{L^{l^*}(\Omega)} + \left\| v \cdot \nabla \varphi\right\| _{L^{l^*}(\Omega)}\right) \mathrm{d}t, \quad \left( \mathrm{Minkowksi}\right) \\
	& = \underbrace{\int_0^T \left\| c\right\| _{L^l(\Omega)}\left\| \varphi \nabla \cdot v\right\| _{L^{l^*}(\Omega)} \mathrm{d}t}_{\left( I\right) } + \underbrace{\int_0^T \left\| c\right\| _{L^l(\Omega)}\left\| v \cdot \nabla \varphi\right\| _{L^{l^*}(\Omega)}\mathrm{d}t}_{\left( II\right) }.
	\end{flalign*}

The first term  $\left( I\right)$ can be estimated by
\begin{flalign*}
	  \quad &\int_0^T \left\| c\right\| _{L^l(\Omega)}\left\| \varphi \nabla \cdot v\right\| _{L^{l^*}(\Omega)} \mathrm{d}t = \left\| \left\| c\right\| _{L^l(\Omega)}\left\| \varphi \nabla \cdot v\right\| _{L^{l^*}(\Omega)}\right\|_{L^1(0,T)} \\
	&\leq  \left\| c\right\| _{L^p(0,T;L^l(\Omega))} \left\| \varphi \nabla \cdot v\right\| _{L^{p^*}(0,T; L^{l^*}(\Omega))}, \quad \left( \mathrm{Hoelder: }\enspace \frac{1}{p} + \frac{1}{p^*} = 1, \enspace 1< p \leq 2\right) , \\
	&=\left\| c\right\| _{L^p(0,T;L^l(\Omega))}  \left\| \left\| \varphi \nabla \cdot v\right\|_{L^{l^*}(\Omega)}\right\|  _{L^{p^*}(0,T)}\\
	& \leq k_c \left\| \left\| \varphi\right\|_{L^{l^*k^*}(\Omega)}\left\|   \nabla \cdot v\right\|_{L^{l*k}(\Omega)}\right\|  _{L^{p^*}(0,T)}, \quad \left( \mathrm{Hoelder: }\enspace \frac{1}{l^*k} + \frac{1}{l^*k^*} = \frac{1}{l^*}, \enspace 1\leq k \leq \infty\right) , \\
	& \leq k_c \left\| \varphi\right\|_{L^{p^*s^*}(0,T;L^{l^*k^*}(\Omega))} \left\|   \nabla \cdot v\right\|_{L^{p^*s}(0,T;L^{l^*k}(\Omega))}, \quad \left( \mathrm{Hoelder: }\enspace \frac{1}{p^*s} + \frac{1}{p^*s^*} = \frac{1}{p^*}, \enspace 1\leq s <\infty \right) \\
	& \leq k_c \left\| \varphi\right\|_{L^{p^*s^*}(0,T;L^{l^*k^*}(\Omega))} k_\theta, 
\end{flalign*}
where we used that $c$ is bounded in $L^p(0,T; \mathrm{BV}(\Omega))$ by assumption and as $\mathrm{BV}(\Omega) \hookrightarrow L^l(\Omega)$ there exists a constant $k_c$ such that $\left\| c\right\| _{L^p(0,T;L^l(\Omega))} \leq k_c$.
For the second term $\left( II\right)$ we find 
\begin{flalign}
\label{eq:II545}
	  \quad & \int_0^T \left\| c\right\| _{L^l(\Omega)}\left\| v \cdot \nabla \varphi\right\| _{L^{l^*}(\Omega)}\mathrm{d}t   \\
	& \leq \int_0^T  \left\| c\right\| _{L^l(\Omega)} \left\| \left| v \right| \cdot \left| \nabla \varphi \right| \right\|_{L^{l^*}(\Omega)}\mathrm{d}t, \quad\left(  \mathrm{Cauchy-Schwarz \enspace inequality}\right) ,  \notag \\
	& \leq  \int_0^T  \left\| c\right\| _{L^l(\Omega)} \left\| v\right\| _{L^\infty(\Omega)^n} \left\| \nabla \varphi \right\|_{L^{l^*}(\Omega)^n} \mathrm{d}t  \notag \\
	& \leq  \int_0^T  \left\| c\right\| _{L^l(\Omega)} k_\infty \left\| \varphi\right\| _{W^{1,l^*}(\Omega)} \mathrm{d}t  \notag \\
	& \leq k_\infty \left\| c\right\| _{L^p(0,T; L^l(\Omega))} \left\| \varphi\right\| _{L^{p^*}(0,T;W^{1,l^*}(\Omega))}, \quad \left(\mathrm{Hoelder: }\enspace \frac{1}{p} + \frac{1}{p^*} = 1, \enspace 1<p \leq 2 \right) \notag \\
	& \leq k_\infty k_c  \left\| \varphi\right\| _{L^{p^*}(0,T;W^{1,l^*}(\Omega))}.
	\label{eq:II551} 
\end{flalign}
Combining those results for $\left( I\right)$ and $\left( II\right)$ yields 
\begin{flalign*}
	\left| 	\int_0^T \int_{\Omega} \dd{t}c \varphi \mathrm{d}x \mathrm{d}t\right| & \leq k_c k_\theta \left\| \varphi\right\|_{L^{p^*s^*}(0,T;L^{l^*k^*}(\Omega))} + k_\infty k_c  \left\| \varphi\right\| _{L^{p^*}(0,T;W^{1,l^*}(\Omega))}\\ 
	& \leq k_c\left( k_\theta \left\| \varphi\right\|_{L^{p^*s^*}(0,T;L^{l^*k^*}(\Omega))} + k_\infty k_6 \left\| \varphi\right\| _{L^{p^*s^*}(0,T;W^{1,l^*}(\Omega))}\right) .
\end{flalign*}
The last estimate was obtained by using $s^*>1$ and thus $L^{p^*s^*}(\Omega) \hookrightarrow L^{p^*}(\Omega)$, which defines the constant $k_6$ by $\left\| x \right\|_{L^{p^*}(\Omega)} \leq k_6\left\| x \right\|_{L^{p^*s^*}(\Omega)} $.   
We now use the dimension dependent embedding 
\begin{equation*}
	W^{1,l^*}\left( \Omega\right) \hookrightarrow L^{l^*k^*}\left( \Omega\right) ,
\end{equation*}
which exists for all $1<k^*<\infty$, as $l^*\geq n$ and then $l^*k^*\geq l^*$. 
\\
This yields
\begin{flalign*}
	\left| 	\int_0^T \int_{\Omega} \dd{t}c \varphi \mathrm{d}x \mathrm{d}t\right| & \leq k_c\left( k_\theta \left\| \varphi\right\|_{L^{p^*s^*}(0,T;W^{1,l^*}(\Omega))} + k_\infty k_6 \left\| \varphi\right\| _{L^{p^*s^*}(0,T;W^{1,l^*}(\Omega))}\right)\\
	& = \left\| \varphi\right\| _{L^{p^*s^*}(0,T;W^{1,l^*}(\Omega))}\left( k_c k_\theta +k_c k_\infty k_6\right) .
\end{flalign*}
We conclude that $\dd{t}c$ acts as a bounded linear functional on $L^{p^*s^*}\left( 0,T; W_0^{1,l^*}(\Omega)\right) $ and thus $\dd{t}c \in \left( L^{p^*s^*}\left( 0,T; W_0^{1,l^*}(\Omega)\right)\right) ^* =  L^{\frac{ps}{p+s-1}}\left( 0,T; W^{-1,l}(\Omega)\right)$.\\
We now use the Lemma of Aubin-Lions \cite{aubin1963theoreme, lions1969quelques}: 
\begin{remark}
	The Lemma of Aubin-Lions reads: Let $X, Y, Z$ be Banach spaces with $ X\Subset Y$, $Y\hookrightarrow Z$, let $c_i$ be a sequence of bounded functions in $L^p\left( 0,T;X\right) $ and $\dd{t}c c_i$ be bounded in $L^q\left( 0,T;Z\right) $ with either $q=1$ and $1\leq p <\infty$ or $q>1$ and $1 \leq p \leq \infty$. Then $c_i$ is relatively compact in $L^p\left( 0,T;Y\right) $. We want to identify the Banach spaces as $X = \mathrm{BV}(\Omega)$, $Y = L^r(\Omega)$ and $Z = W^{-1,l}(\Omega)$.
\end{remark}
The embedding $\mathrm{BV}(\Omega) \Subset L^r(\Omega)$ holds for $r < \frac{n}{n-1}$ \cite{ambrosio2000functions}, i.e. 
\begin{equation*}
	\left\lbrace \begin{array}{ll}
		r<2, & n=2 \\
		r< 1.5, & n=3.
	\end{array} \right. .
\end{equation*}
It remains to identify for which $r$ the continuous embedding $L^r(\Omega) \hookrightarrow W^{-1,l}(\Omega)$ exists. The embedding holds if $W^{1,l^*}(\Omega) \hookrightarrow L^{r^*}(\Omega)$ which by the Sobolev embedding theorem for the different dimensions and $l^*\geq n$ translates to 
\begin{equation}
	W^{1,l^*}(\Omega) \hookrightarrow L^{r^*}(\Omega) \enspace \mathrm{for } \enspace l^* \leq r^* < \infty.
\end{equation}
In terms of $r$, this results in $1 < r \leq l$. \\
Applying Aubin-Lions thus yields $\left\lbrace c \in L^p(0,T; \mathrm{BV}(\Omega)) : \exists k>0 \enspace \mathrm{s. t. }\enspace \left\| c\right\| _{L^p(0,T; \mathrm{BV}(\Omega))} \leq k,\right.$ $\left. \enspace \dd{t}c + v\cdot \nabla c = 0 \text{ in } \mathcal{D}'(\Omega \times [0,T]) \right\rbrace $ is relatively compact in $L^p(0,T; L^r(\Omega))$ for $r$ satisfying 
\begin{equation}
	\label{eq:boundsonr}
	\left\lbrace \begin{array}{ll}
		1 \leq r<2, & n=2 \\
		1 \leq r < \frac{3}{2}, & n=3
	\end{array} \right. .
\end{equation}
The sequence $c_m$ thus converges even strongly to $c$ in $L^p(0,T; L^r(\Omega))$.\\
We are now settled to prove convergence of the constraint. In the following, let $\varphi \in \mathcal{C}_0 ^\infty\left( \Omega \times \left[ 0,T\right] \right) $. 
We start with the time derivative: 
\begin{flalign*}
	\int_0^T \int_{\Omega} \left( \left( \dd{t}c\right) _m-\dd{t}c\right) \varphi \mathrm{d}x \mathrm{d}t = -\int_0^T \int_{\Omega} \left( c_m - c\right) \dd{t}\varphi \mathrm{d}x \mathrm{d}t, 
\end{flalign*}
by integration by parts. We know that $c_m \overset{* \phantom{a}}{\rightharpoonup} c$ in $L^p(0,T; \mathrm{BV}(\Omega))\cong \left( L^{p^*}(0,T; Z)\right) ^*$ for $\mathrm{BV}(\Omega)$ being isometrically isomorphic to the dual space of $Z$. As $\varphi \in \mathcal{C}_0 ^\infty\left( \Omega \times \left[ 0,T\right] \right) $, it follows that $\dd{t} \varphi \in L^{p^*}(0,T; Z)$ and this yields
\begin{equation*}
	-\int_0^T \int_{\Omega} \left( c_m - c\right) \dd{t}\varphi \mathrm{d}x \mathrm{d}t \longrightarrow 0 \quad \mathrm{for } \enspace m \longrightarrow \infty.
\end{equation*}
For the product term, it holds that 
\begin{align*}
    - & \int_0^T \int_{\Omega} \left( \nabla c_m \cdot v_m - \nabla c \cdot  v\right) \varphi\, \mathrm{d}x\, \mathrm{d}t 
	= \int_0^T \int_{\Omega} c_m \nabla \cdot \left( v_m \varphi\right)  - c \nabla \cdot \left( v \varphi\right) \mathrm{d}x\, \mathrm{d}t \\
	 = & \underbrace{ \int_0^T \int_{\Omega} \left( c_m-c\right) \nabla \cdot \left( v_m \varphi\right) \mathrm{d}x\, \mathrm{d}t}_{\left( I\right)} +  \underbrace{ \int_0^T \int_{\Omega} c \nabla \cdot \left( \varphi v_m - \varphi v\right)   \mathrm{d}x\, \mathrm{d}t}_{\left( II\right)}.
\end{align*} 
For $\left( I\right) $  we obtain 
\begin{flalign}
	  & \int_0^T \int_{\Omega} \left( c_m-c\right) \nabla \cdot \left( v_m \varphi\right) \mathrm{d}x \mathrm{d}t \notag \\
	& \leq \int_0^T \left\| c_m-c\right\|_{L^{r}(\Omega)} \left\|  \nabla \cdot \left( v_m \varphi\right)\right\| _{L^{r^*}(\Omega)} \mathrm{d}t, \quad \left(\mathrm{Hoelder: }\enspace \frac{1}{r} + \frac{1}{r^*} = 1, \enspace 1<r < \infty \right)  \notag  \\
	& \leq \left\| c_m-c\right\|_{L^p(0,T; L^{r}(\Omega))} \left\|  \nabla \cdot \left( v_m \varphi\right)\right\| _{L^{p^*}(0,T;L^{r^*}(\Omega))} , \quad \left(\mathrm{Hoelder: }\enspace \frac{1}{p} + \frac{1}{p^*} = 1, \enspace 1<p \leq 2 \right) \notag  \\  
	& \leq  \left\| c_m-c\right\|_{L^p(0,T; L^{r}(\Omega))} \left( \left\|  \nabla \varphi \cdot v_m\right\| _{L^{p^*}(0,T;L^{r^*}(\Omega))} + \left\| \varphi \nabla \cdot \left( v_m \right)\right\| _{L^{p^*}(0,T;L^{r^*}(\Omega))}\right)\\
	\label{eq:i570}
	& \leq  \left\| c_m-c\right\|_{L^p(0,T; L^{r}(\Omega))} \left(\underbrace{ \left\|  \nabla \varphi \cdot v_m\right\| _{L^{p^*}(0,T;L^{r^*}(\Omega))} }_{(a)}\right. \notag \\
	&+\left. \underbrace{\left\| \varphi \right\|_{L^{p^*s^*}(0,T;L^{r^*} (\Omega))}}_{(b)} \underbrace{\left\| \nabla \cdot \left( v_m \right)\right\| _{L^{p^*s}(0,T;L^{r^*}(\Omega))}}_{(c)}\right) \quad 
\;\; \left(\enspace \frac{1}{p^*s} + \frac{1}{p^*s^*} = \frac{1}{p^*}, \enspace 1<s < \infty\right) \,.
\end{flalign}
We note that $(b)$ is bounded as $\varphi$ is a test function. Consider now $(c)$: $\left\| \nabla \cdot \left( v_m\right) \right\| _{L^{p^*s}(0,T; L^{l^*k}(\Omega))}$ is bounded by $k_\theta$ by assumption. We thus need an embedding 
\begin{equation*}
	L^{p^*s}(0,T; L^{l^*k}(\Omega))\hookrightarrow L^{p^*s}(0,T; L^{r^*}(\Omega))
\end{equation*} and therefore $r^* \leq l^*k$ which is equivalent to $r \geq \frac{l^*k}{l^*k-1}$. We need to adjust the bounds on $k$  as follows. 
\begin{itemize}
	\item Case $n = 2$: We have $1 \leq r < 2$ and $l^* \geq 2$. 
    $1 \leq \frac{l^*k}{l^*k-1}$ is always fulfilled, but in order to guarantee $\frac{l^*k}{l^*k-1} <2$, we have to restrict $k$ such that $\left[\frac{l^*k}{l^*k-1},2 \right) \neq \emptyset$. This is ensured $l^*k>2$, which is equivalent to $k>1$, as $l^* \geq 2$.
	\item Case $n = 3$: We have $1 \leq r < \frac{3}{2}$ and $l^* \geq 3$. 
    $1 \leq \frac{l^*k}{l^*k-1}$ is always fulfilled, but in order to guarantee $\frac{l^*k}{l^*k-1} <\frac{3}{2}$, we have to restrict $k$ such that $\left[\frac{l^*k}{l^*k-1},\frac{3}{2} \right) \neq \emptyset$. This is ensured $l^*k>3$, which is equivalent to $k>1$, as $l^* \geq 3$.   
\end{itemize}
It remains to show that $(a)$ is bounded. This holds as $\varphi$ is a test function and thus $\nabla \varphi$ is bounded and $\left\| v\right\|_{L^\infty(\Omega)^n} \leq k_\infty$. We showed that $(I) \longrightarrow 0 $ as $m \longrightarrow \infty$. Now consider $(II)$, i.e.
\begin{flalign}
 \quad	\int_0^T \int_{\Omega} c \nabla \cdot \left( \varphi v_m - \varphi v\right)   \mathrm{d}x\, \mathrm{d}t 
& = \underbrace{\int_0^T \int_{\Omega} c \varphi \nabla \cdot \left( v_m-v\right) \mathrm{d}x\, \mathrm{d}t}_{(a)}+ 	\underbrace{\int_0^T \int_{\Omega}c   \left( v_m-v\right)\cdot  \nabla \varphi\, \mathrm{d}x \mathrm{d}t}_{(b)}.
\label{eq:II574}
\end{flalign}
 Consider $(a)$ first. By assumption, $\nabla \cdot \left( v_m-v\right) \in L^{p^*s}\left( 0,T;L^{l^*k}(\Omega)\right)$. We need to show $c \varphi \in \left( L^{p^*s}\left( 0,T;L^{l^*k}(\Omega)\right) \right)^*$. Starting from $c \in L^p\left( 0,T; \mathrm{BV}(\Omega)\right) $, the first embedding needed is $\mathrm{BV}(\Omega) \hookrightarrow L^{\frac{l^*k}{l^*k-1}}(\Omega) = L^{\left( l^*k\right) ^*}(\Omega)$. 
 \begin{itemize}
 	\item Case $n = 2$: $\mathrm{BV}(\Omega) \hookrightarrow L^{\frac{l^*k}{l^*k-1}}(\Omega)$ exists continuously for $1 \leq \frac{l^*k}{l^*k-1} \leq 2$. This is fulfilled for $l^*k\geq 2$, which is ensured by $l^*\geq 2$ and $k>1$.
 	\item Case $n = 3$: $\mathrm{BV}(\Omega) \hookrightarrow L^{\frac{l^*k}{l^*k-1}}(\Omega)$ exists for $1 \leq \frac{l^*k}{l^*k-1} \leq 1.5$. Again, this holds for $l^*k \geq 3$, which is ensured by $l^* \geq 3$ and $k>1$.
 \end{itemize} 
Note that $\left( p^*s\right) ^* = \frac{ps}{ps-p+1}$ and $\frac{ps}{ps-p+1}<p$ if $1<p$, which again holds by assumption. \\
In summary, we have
\begin{equation*}
	L^p\left( 0,T;\mathrm{BV}(\Omega)\right) \hookrightarrow 
	L^p\left( 0,T;L^{\left( l^*k\right) ^*}(\Omega)\right) \hookrightarrow
	L^{\left( p^*s\right) ^*}\left( 0,T;L^{\left( l^*k\right) ^*}(\Omega)\right).
\end{equation*}
As $\varphi$ and its derivatives are bounded, $c\varphi \in L^{\left( p^*s\right) ^*}\left( 0,T;L^{\left( l^*k\right) ^*}(\Omega)\right)$. Moreover, from $v_m \overset{* \phantom{a}}{\rightharpoonup} v$ it follows that $\nabla \cdot v_m \overset{* \phantom{a}}{\rightharpoonup} \nabla \cdot v$ and thus altogether 
\begin{equation*}
	\int_0^T \int_{\Omega} c \varphi \nabla \cdot \left( v_m-v\right) \mathrm{d}x \mathrm{d}t \longrightarrow 0 \enspace \mathrm{for } \enspace m \longrightarrow \infty. 
\end{equation*}
Considering $(b)$, we denote first that 
\begin{equation*}
	\mathrm{BV}\left( \Omega \right) \hookrightarrow L^1\left( \Omega \right) \enspace \Rightarrow \enspace L^p\left( 0,T; \mathrm{BV}(\Omega)\right) \hookrightarrow L^p\left( 0,T; L^1(\Omega)\right) \hookrightarrow L^1\left( 0,T; L^1(\Omega)\right).
\end{equation*}
Therefore, we have $c \in L^1\left( 0,T; L^1(\Omega)\right)$ and by the boundedness of $\varphi$ and its derivatives, we obtain $c\nabla \varphi \in L^1\left( 0,T; L^1(\Omega)^n\right)$. Moreover, by $\left\| v\right\|_{L^\infty(\Omega)^n} \leq k_\infty$ we see that $\left( v_m-v\right) \in L^\infty\left( 0,T; L^\infty(\Omega)^n\right) $. Now by 
\begin{equation*}
	 L^\infty\left( 0,T; L^\infty(\Omega)^n\right) \cong \left( L^1\left( 0,T; L^1(\Omega)^n\right)\right) ^*
\end{equation*}  
and the weak-$^*$ convergence of $v_m$ to $v$, we reach
\begin{equation*}
	\int_0^T \int_{\Omega}c  \nabla \varphi \cdot \left( v_m-v\right)  \mathrm{d}x \mathrm{d}t\longrightarrow 0 \enspace \mathrm{for } \enspace m \longrightarrow \infty.
\end{equation*}
Combining all the results yields the convergence of the constraint 
\begin{equation*}
	\int_0^T \int_{\Omega} \left( \left( \left( \dd{t}c\right)_m -  \nabla c_m \cdot v_m \right) - \left( \dd{t}c  - \nabla c \cdot v\right)\right)  \varphi 	 \mathrm{d}x \mathrm{d}t \longrightarrow 0 \enspace \mathrm{for } \enspace m \longrightarrow \infty.
\end{equation*}
\underline{\textit{Convergence of the constraint: Mass conservation}:}\\
	Let $(c_m, v_m)\in L^{\hat{p}}\left( 0,T; BV\left( \Omega\right) \right) \times L^q\left( 0,T; BV\left( \Omega\right)^n \right), m\in  \N,$ be an admissible sequence (i.e., $c_m\in L^p(0,T,BV(\Omega))$, $\left\| v_m\right\|_{L^\infty(0,T;L^\infty(\Omega)^n)} \leq k_\infty$, $\left\Vert \nabla\cdot v_m\right\Vert _{\theta}\leq k_{\theta}$ , $ m_2\left(c,v\right)=0 \text{ in } \mathcal{D}'(\Omega \times [0,T])$), which also fulfills $(c_m,v_m) \in S_\nu$ for some $\nu \in \mathbb{R}$. Then $c_m$ and $v_m$ are bounded and it exist $c$ and $v$ such that by passing over to a subsequence (again denoted by $c_m$, $v_m$) we have 
\begin{equation*}
	c_m \overset{* \phantom{a}}{\rightharpoonup} c, \quad v_m \overset{* \phantom{a}}{\rightharpoonup} v.
\end{equation*}%
We want to show that 	
	\begin{equation*}
		\left( \dd{t}c\right) _m + \nabla \cdot\left(  c_m  v_m\right)  \longrightarrow \dd{t}c + \nabla \cdot \left( c  v\right) \quad \text{in } \mathcal{D}'(\Omega \times [0,T]) , 
	\end{equation*}
at least in a distributional sense. 
We start again with a bound on $\dd{t}c$:  
\begin{flalign}
	\left| \int_0^T \int_{\Omega} \dd{t}c \varphi \mathrm{d}x \mathrm{d}t\right| & = 	\left| \int_0^T \int_{\Omega} - \nabla \cdot \left( cv\right)  \varphi \mathrm{d}x \mathrm{d}t\right| 
	=	\left| \int_0^T \int_{\Omega}  \left( cv\right) \cdot \nabla \varphi \mathrm{d}x \mathrm{d}t\right| \notag \\
	& \leq \int_0^T \int_{\Omega} \left|  \left( cv\right) \cdot  \nabla \varphi\right|  \mathrm{d}x \mathrm{d}t  \notag  \\
	& \leq \int_0^T \left\| c\right\| _{L^l(\Omega)}\left\| v \cdot \nabla \varphi\right\| _{L^{l^*}(\Omega)}\mathrm{d}t, \quad \left( \mathrm{Hoelder: }\enspace \frac{1}{l} + \frac{1}{l^*} = 1, \enspace l\leq \frac{n}{n-1}\right) ,   \notag  \\
	& \leq k_\infty k_c  \left\| \varphi\right\| _{L^{p^*}(0,T;W^{1,l^*}(\Omega))} , \quad \left( \mathrm{c.f. \enspace \cref{eq:II545} - \cref{eq:II551}}\right) .
\end{flalign}
We thus know that $\dd{t}c$ acts as a bounded linear functional on $L^{p^*}\left( 0,T; W^{1,l^*}(\Omega)\right) $ and thus $\dd{t}c \in L^{p}\left( 0,T; W^{-1,l}(\Omega)\right) $.
The Lemma of Aubin-Lions can be applied similarly to the optical flow case and yields strong convergence of $c_m \longrightarrow c $ in $L^p(0,T; L^r(\Omega))$ with the same constraints on $r$. 
Also, the arguments for the convergence of the time derivative 
\begin{equation*}
	-\int_0^T \int_{\Omega} \left( c_m - c\right) \dd{t}\varphi \mathrm{d}x \mathrm{d}t \longrightarrow 0 \quad \mathrm{for } \enspace m \longrightarrow \infty.
\end{equation*}
are the same as in the optical flow case. 
We now come to the product term. It remains to show that 
\begin{equation*}
	\nabla \cdot \left( c_mv_m\right) \rightharpoonup \nabla \cdot \left( cv\right) .
\end{equation*}
\begin{flalign}
	- \int_0^T \int_{\Omega}&\left( \nabla \cdot \left( c_m v_m\right)  - \nabla \cdot \left( cv\right) \right) \varphi \mathrm{d}x \mathrm{d}t = \int_0^T \int_{\Omega}\left( c_mv_m - cv \pm cv_m\right) \cdot \nabla \varphi \mathrm{d}x \mathrm{d}t \notag \\
	&= \int_0^T \int_{\Omega}\left( \left( c_m-c\right) v_m + c\left( v_m-v\right) \right) \cdot \nabla \varphi \,\mathrm{d}x\, \mathrm{d}t \notag \\
	& \leq \int_0^T \left\| c_m-c\right\|_{L^r(\Omega)}\left\| v_m \cdot \nabla \varphi \right\|_{L^{r^*}(\Omega) } + \int_{\Omega} c\left( v_m-v\right)  \cdot \nabla \varphi \mathrm{d}x \mathrm{d}t, \quad \left( \mathrm{ }\enspace \frac{1}{r} + \frac{1}{r^*} = 1\right) ,  \notag \\
	& \leq \underbrace{\left\| c_m-c\right\|_{L^p(0,T;L^r(\Omega))}\left\| v_m \cdot \nabla \varphi \right\|_{L^{p^*}(0,T;L^{r^*}(\Omega)) } }_{\longrightarrow 0 \quad  \mathrm{for} \enspace m \longrightarrow \infty \enspace \left( \mathrm{cf.} \enspace\cref{eq:i570} (a)\right) } + \underbrace{\int_0^T \int_{\Omega} c\left( v_m-v\right)  \cdot \nabla \varphi \mathrm{d}x \mathrm{d}t}_{\longrightarrow 0 \quad  \mathrm{for} \enspace m \longrightarrow \infty \enspace \left( \mathrm{cf.} \enspace \cref{eq:II574} (b)\right)}.
\end{flalign}

\end{proof}
\section{Extension Remark~\ref{rem:regularization_properties}}
\label{app:extended_discussion_rem}

Analysis of the regularization properties of the unconstrained approach \eqref{eq:min_prob2} might be carried out by exploiting the established theory for Tikhonov-type regularization in Banach spaces in the nonlinear problem setting (see, e.g., \cite{hofmann2007convergence, schuster2012regularization}). 
Even if the image reconstruction problem with respect to $c$ is linear, the reconstruction of the tuple $(c,v)$ is nonlinear due to the PDE constraint defined by the motion prior $m_i$.
Regularization properties of the proposed method can be guaranteed by suitable choices of penalty terms ${S}$ and ${R}$, e.g., in line with the previous setting are choices such that $\| {R}(c(\cdot)) \|_{L^1(0,T)} $ and $\| {S}(v(\cdot)) \|_{L^1(0,T)} $ are proper, convex, and weak-$^\ast$ lower semi-continuous functionals with respect to $(c,v)\in  L^{\hat{p}}( 0,T; BV(\Omega))  \times L^q( 0,T; BV(\Omega)^n)$.
The constraint $\|v\|_{L^\infty(0,T;L^\infty(\Omega)^n)} \leq k_\infty $ is then again sufficient to guarantee weak-$^\ast$ closedness of the sublevel sets (see the proof of Theorem~\ref{thm:existence}).
The crucial and more sophisticated remaining part is to show that the nonlinear operator $F:\mathrm{dom}(F) \subset U \rightarrow L^2(0,T,Y) \times L^s(0,T,L^r(\Omega)) $, for certain $s,r\geq 1$, with 
\begin{equation}
    (c,v) \mapsto (A(c(\cdot),\cdot),m_i(c,v))  
\end{equation}
is weak-$^\ast$-to-weak continuous with respect to a suitable space $$U\subset ( L^{\hat{p}}( 0,T; BV(\Omega)) \times L^q( 0,T; BV(\Omega)^n))$$ and a suitable weak-$^\ast$ closed domain $\mathrm{dom}(F)$ of $F$.
In the first component one has to deal with a bounded linear operator mapping to the reflexive Banach space $L^2(0,T;Y)$. 
It would thus be sufficient to verify that the linear operator is weak-$^\ast$-to-weak-$^\ast$ continuous with respect to $c$. 
For the  second component we need weak-$^\ast$-to-weak continuity with respect to the tuple.
The interplay between weak-$^\ast$-to-weak continuity and a suitable choice of spaces replaces the motion constraint being fulfilled in a distributional sense in the setting of Theorem~\ref{thm:existence}. Choosing $L^s(0,T,L^r(\Omega))$ includes the motion constraint in a stronger sense which requires further regularity of $c$, e.g., we need existence of $\frac{\partial}{\partial t}c$ by making suitable assumptions on $U$ implying weak convergence of the derivatives of $c$.
A simple but more restrictive example would be the reflexive space $U=H^1(0,T;H^1(\Omega)) \times L^2(0,T;L^2(\Omega))$ (where weak-$^\ast$ and weak convergence are equivalent) and $p=q=r=s=2$ where $\mathrm{dom}(F)=\{ (c,v)\in U| \|v\|_{L^\infty(0,T;L^\infty(\Omega)^n)} \leq k_\infty \land \|\nabla c\|_{L^\infty(0,T;L^\infty(\Omega)^n)} \leq k'_\infty  \}$. 
A more general investigation in this direction is beyond the scope of the present work but under these assumptions the regularization properties would then be given by the general theory for noisy measurements $(u^\delta,0)$, see, for example \cite{hofmann2007convergence}.
This perspective on PDE-constraint parameter identification in the dynamic setting has also been considered as an all-at-once approach to dynamic inverse problems from a more general point of view \cite{Kaltenbacher17,Tram19}.

\section{Simulation parameters}
\label{app:sim_par}
The scanner is modeled as follows: 
The drive-field frequencies are given by $f_x = \qty{2.5}{\mega\hertz}/102$, $f_y = \qty{2.5}{\mega\hertz}/96$ and $f_z = \qty{2.5}{\mega\hertz}/99$ with an amplitude of \qty{14}{\milli\tesla\per \mu_0}, resulting in a repetition time of \qty{21.54}{\milli\second}. The selection field gradient strength is \qty{0.5}{\tesla\per\meter\per\mu_0} in $x$- and $y$-direction and \qty{1.0}{\tesla\per\meter\per\mu_0} in $z$-direction. The magnetic nanoparticles are modeled as stated in \cref{tab:particle_pars}. 

\begin{table}[tbhp]
\footnotesize
    \caption{Physical constants and parameters for modeling the tracer particles. }
    \label{tab:particle_pars}
    \centering
    \begin{tabular}{|l|c|}
        \hline
        Parameter & Value \\
        \hline
        Permeability constant $\mu_0$ & $4\pi\cdot10^{-7}$\unit{N\per A^2}\\
        \hline
        Particle core diameter & $2\cdot 10^{-8}$\unit{\meter} \\
        Core saturation magnetization & \qty{0.6}{\tesla \per \mu_0}\\
        
        \hline
    \end{tabular}
\end{table}

We model ideal magnetic fields and use the Langevin magnetization model. 
In the forward operator voltage measurements for three receive channels in respective unit vector directions are used. To avoid inverse crime, simulations are performed on a finer grid compared to the reconstruction, i.e. we use $40^3$ voxels of size \qtyproduct{1 x 1 x 0.5}{\mm} for simulation and $20^3$ voxels of size \qtyproduct{2 x 2 x 1}{\mm} for reconstruction.
For reconstruction, the time-dependent signals are transformed to Fourier domain, a frequency selection is performed and real an imaginary part are split, i.e., concatenated.

\section{Information on parameter search regions for the numerical experiments}
\label{app:par_search}

In each subsection in \Cref{sec:numerical}, we stated the parameters used for the illustrative reconstructed image sequences. These were obtained by detailed parameter searches in each case and for each algorithm. 

\subsection{The simulated data experiments}

For the Kaczmarz algorithm, we tested for the early stopping index $k$ and the Tikhonov regularization parameter $\lambda$. The parameter $\lambda$ was tested in the range between $10^{-5}$ and $10^3$, $k$ in between 1 and 10. 

Frame-by-frame SPDHG algorithm has two parameters, $\alpha_1$ corresponding to the $L^1$-penalty term and $\alpha_2$ corresponding to the TV penalty term. The testing range for $\alpha_1$ was in between $10^{-3}$ and $10^{-1}$, for $\alpha_2$ in between $10^{-9}$ and $10^{-5}$. 

Parameter tests for the joint approaches were performed on 10 time steps simultaneously within the ranges stated in \cref{tab:par_search_simulateddata}.   

\begin{table*}[tbhp]
\footnotesize
    \caption{Parameter search area for the joint approaches for the simulated data.}
    \label{tab:par_search_simulateddata}
    \centering
    \begin{tabular}{|l|c|c|}
        \hline
        Parameter & Min. value & Max. value  \\
        \hline
        $\alpha_1$ & $10^{-3}$ & $10^{-1}$\\
        $\alpha_2$ & $10^{-8}$ & $10^{-5}$ \\
        $\beta$ & $10^{-2}$ & 1 \\
        $\gamma$ & $10^{-5}$ & 1 \\
        \hline
    \end{tabular}
\end{table*}
The resulting image sequences for each parameter combination was tested for the SSIM value. The best parameter combination for either criterion is stated in \cref{tab:best_parameter_simdata_ssim}. 

The resulting parameters are stated below for completeness. 
\begin{table*}[tbhp]
\footnotesize
    \centering
        \caption{The best parameters obtained for the reconstruction of the simulated data in terms of SSIM value. }
        \label{tab:best_parameter_simdata_ssim}
    \begin{tabular}{|l|l|c|c|c|c|c|c|}
    \hline
       Algorithm  & Motion model & Early stopping & $\lambda$ & $\alpha_1$ & $\alpha_2$ & $\beta$ & $\gamma$ \\
       \hline
       Kaczmarz  & None & $k=3\hphantom{^2}$ & $10^3$ &  &  &  &  \\
       SPDHG & None & $k=2\cdot10^3$ & & $10^{-2}$& $10^{-7}$& & \\
       Joint, $L^1$-D & Optical Flow & $k=10^2$ & & $10^{-2}$ & $10^{-7}$ & $10^{-1}$ & $10^{-4}$  \\
       Joint, $L^2$-D & Optical Flow & $k=10^2$ & & $10^{-2}$ & $10^{-7}$ & $10^{-1}$& $10^{-4}$  \\
       Joint, $L^1$-D & Mass Conservation & $k=10^2$ & & $10^{-2}$ & $10^{-8}$ & $10^{-1}$& $10^{-4}$  \\
       Joint, $L^2$-D & Mass Conservation & $k=10^2$ & & $10^{-2}$ & $10^{-5}$ & $10^{-1}$& $10^{-4}$  \\
       \hline
    \end{tabular}
\end{table*}

\subsection{The rotation phantom data experiments}

For the Kaczmarz algorithm, we tested for the early stopping index $k$ and the Tikhonov regularization parameter $\lambda$. The parameter $\lambda$ was tested in the range between $10^{-4}$ and 30, $k$ in between 1 and 100. 

Frame-by-frame SPDHG algorithm has two parameters, $\alpha_1$ corresponding to the $L^1$-penalty term and $\alpha_2$ corresponding to the TV penalty term. The testing range for $\alpha_1$ was in between $10^{-4}$ and 1, for $\alpha_2$ in between $10^{-4}$ and 10. 

Parameter tests for the joint approaches were performed on 15 time steps simultaneously within the ranges stated in \cref{tab:par_search_rotation}.   

\begin{table*}[tbhp]
\footnotesize
    \caption{Parameter search area for the joint approaches for the rotation phantom data.}
    \label{tab:par_search_rotation}
    \centering
    \begin{tabular}{|l|c|c|}
        \hline
        Parameter & Min. value & Max. value  \\
        \hline
        $\alpha_1$ & $10^{-2}$ & 10\\
        $\alpha_2$ & $10^{-2}$ & 100 \\
        $\beta$ & $10^{-2}$ & 1 \\
        $\gamma$ & 1 & 200 \\
        \hline
    \end{tabular}
\end{table*}
The resulting image sequences were visually inspected and the most convincing ones were chosen. The search area was iteratively reduced and refined until no visual differences were observed. 

The resulting parameters for the \qty{3}{\hertz} dataset are stated in \cref{tab:bohrerdaten_3hz_parameter} and below for completeness. 
\begin{table*}[tbhp]
\footnotesize
    \centering
        \caption{The parameters used for the reconstruction of the \qty{3}{\hertz} rotation phantom data. }
    \begin{tabular}{|l|l|c|c|c|c|c|c|}
    \hline
       Algorithm  & Motion model & Early stopping & $\lambda$ & $\alpha_1$ & $\alpha_2$ & $\beta$ & $\gamma$ \\
       \hline
       Kaczmarz  & None & $k=10\hphantom{^2}$ & 5.62 &  &  &  &  \\
       SPDHG & None & $k=10^3$ & & $5.0\cdot10^{-1}$& $3.0\cdot 10^{-1}$& & \\
       Joint, $L^1$-D & Optical Flow & $k=10^2$ & & $6.0\cdot10^{-1}$ & $1.0\cdot 10^{-1}$ & $10^{-1}$ & 100  \\
       Joint, $L^2$-D & Optical Flow & $k=10^2$ & & $2.5\cdot10^{-1}$ & $1.0\cdot 10^{+2}$ & $10^{-1}$& 100  \\
       Joint, $L^1$-D & Mass Conservation & $k=10^2$ & & $7.0\cdot10^{-1}$ & $5.0\cdot 10^{-1}$ & $10^{-1}$& 100  \\
       Joint, $L^2$-D & Mass Conservation & $k=10^2$ & & $2.5\cdot10^{-1}$ & $2.5\cdot 10^{-1}$ & $10^{-1}$& 100  \\
       \hline
    \end{tabular}
\end{table*}
Parameter tests for the \qty{1}{\hertz} and \qty{7}{\hertz} datasets are performed analogously and yield similar results.

\subsection{The in-vivo mouse data data experiments}
Motivated by the specific application, we use the mass conservation constrained $L^1$ algorithm. The parameter search area is indicated in \cref{tab:par_search_rotation_mousedata}, tests were performed on 30 time steps simultaneously. 

\begin{table*}[tbhp]
\footnotesize
    \caption{Parameter search area for the mass conservation $L^1$-D algorithm for the in-vivo mouse data reconstruction.}
    \label{tab:par_search_rotation_mousedata}
    \centering
    \begin{tabular}{|l|c|c|}
        \hline
        Parameter & Min. value & Max. value  \\
        \hline
        $\alpha_1$ & $10^{-2}$ & $10^{-1}$\\
        $\alpha_2$ & $10^{-3}$ & 200 \\
        $\beta$ & $10^{-3}$ & $10^{-1}$ \\
        $\gamma$ & 1 & 1000 \\
        \hline
    \end{tabular}
\end{table*}
The resulting reconstruction parameters are given by $\alpha_1= 0.05$, $\alpha_2 =100$, $\beta = 0.1$ and $\gamma = 100$. 

\end{document}